\documentclass{amsart}
\usepackage[dvips]{graphicx}
\usepackage{amssymb}
\usepackage{amscd}
\usepackage{enumerate}
\usepackage[T1]{fontenc}
\usepackage[utf8]{inputenc}
\usepackage{amssymb}
\usepackage[capitalize]{cleveref}

\newtheorem{lemma}{Lemma}[section]
\newtheorem{theorem}{Theorem}

\newtheorem{mcorollary}[theorem]{Corollary}
\newtheorem{othertheorem}[lemma]{Theorem}
\newtheorem{corollary}[lemma]{Corollary}
\newtheorem{proposition}[lemma]{Proposition}
\theoremstyle{definition}
\newtheorem{definition}[lemma]{Definition}

\newtheorem{assump}[lemma]{Assumption}
\newtheorem{claim}[lemma]{Claim}

\newcommand{\blackboard}[1]{\ensuremath{\mathbb{#1}}}

\newcommand{\complexes}{\blackboard{C}}
\newcommand{\hyperbolic}{\blackboard{H}}

\newcommand{\integers}{\blackboard{Z}} %
\newcommand{\reals}{\blackboard{R}}
\newcommand{\naturals}{\blackboard{N}}

\newcommand{\PSL}{\ensuremath{\mathrm{PSL}}}

\newcommand{\PML}{\mathcal{PML}}
\newcommand{\ML}{\mathcal{ML}}
\newcommand{\UML}{\mathcal{UML}}
\newcommand{\teich}{\mathcal{T}}
\newcommand{\EL}{\mathcal{EL}}
\newcommand{\CC}{\mathcal{CC}}

\newcommand{\len}{\mathrm{length}}
\let \length \len
\newcommand{\Fr}{\mathrm{Fr}}
\newcommand{\Int}{\mathrm{Int}}
\newcommand{\abparam}{\mathcal{T}(S) \times \mathcal{T}(\bar{S})}
\newcommand{\subord}{\stackrel{d}{\searrow}}
\newcommand{\supord}{\stackrel{d}{\swarrow}}
\newcommand{\seeq}{\underset{=}{\searrow}}
\newcommand{\sweq}{\underset{=}{\swarrow}}
\newcommand{\pre}{\mathrm{prec}}
\newcommand{\suc}{\mathrm{succ}}
\newcommand{\base}{\mathrm{base}}

\newcommand{\ie}{i.e.\ }
\newcommand{\parreal}{\partial_\mathrm{real}}

\begin{document}

\title[quasi-Fuchsian spaces]{Divergence, exotic convergence and self-bumping in quasi-Fuchsian spaces}
\author{Ken'ichi Ohshika}
\thanks{Partially supported by JSPS Grant-in-Aid Scientific Research 17H02843}
\address{Department of Mathematics, Graduate School of Science, Osaka University, Osaka 560-0043, Japan}
\maketitle
\renewcommand{\abstractname}{Résumé}
\begin{abstract}
Dans cet article, on étudie la topologie des bords des espaces quasi-fuchsiens.
D'abord on montre comment on peut savoir les invariants des bouts du groupe limite pour une suite convergente de groupes quasi-fuchsiens donnée, en utilisant les informations sur le comportement asymptotique des structures conformes à l'infini des groupes dans la suite. 
Ce résultat donne lieu à une condition suffisante pour la divergence des groupes quasi-fuchsiens, laquelle est une généralisation du résultat d'Ito qui n'a traité que le cas des groupes du tore une fois perforé.
On démontre de plus que des groupes quasi-fuchsiens ne peuvent approcher un b-groupe  hors de la tranche de Bers que si la limite admet un locus parabolique isolé.
Ce résultat-ci  permet également de donner une condition nécessaire pour qu'un point au bord de l'espace de déformations soit un point de «\,l'entrechoquement\,».
Pour démontrer ces résultats, on utilise des variétés modèles construites par Minsky et leurs limites géométriques étudiées par Ohshika-Soma.
Pour que les lecteurs n'aient pas besoin de se reporter à l'article d'Ohshika-Soma, le présent article aussi contient les arguments simplifiés mais assez détaillés d'Ohshika-Soma qui sont nécessaires pour les démonstrations des théorèmes principaux.
\end{abstract}
\renewcommand{\abstractname}{Abstract}
\begin{abstract}
In this paper, we study the topology of the boundaries of quasi-Fuchsian spaces.
We  first show for a given convergent sequence of quasi-Fuchsian groups, how we can know the end invariant of the limit group from the information on the behaviour of  conformal structures at infinity of the groups.
This result gives rise to a sufficient condition for divergence of quasi-Fuchsian groups, which generalises Ito's result in the once-punctured torus case to  higher genera.
We further show that quasi-Fuchsian groups can approach a b-group not along Bers slices only when the limit has isolated parabolic loci.
This makes it possible to give a necessary condition for points on the boundaries of  quasi-Fuchsian spaces to be self-bumping points.
We use model manifolds invented by Minsky and their geometric limits studied by Ohshika-Soma to prove these results.
This  paper has been made as self-contained as possible so that the reader does not need to consult the paper of Ohshika-Soma directly.
\end{abstract}
\sloppy
\section{Introduction}
In the theory of Kleinian groups, after the major problems like Marden's tameness conjecture and the ending lamination conjecture were solved, the attention is now focused on studying the topological structure of deformation spaces.
Although we know, by the resolution of the Bers-Sullivan-Thurston density conjecture (\cite{Brom}, \cite{BB}, \cite{OhL}
, \cite{KS}, \cite{L}, \cite{Ohp}, \cite{NS}), that every finitely generated Kleinian group is an algebraic limit of quasi-conformal deformations of a (minimally parabolic) geometrically finite group, the structure of deformation spaces as topological spaces is far from completely understood.
For instance, as was observed by work of Anderson-Canary \cite{AC} and McMullen \cite{Mc}, even in the case of  Kleinian groups isomorphic to closed surface groups, the deformation spaces are fairly complicated, and in particular it is known that they are not manifolds since they have singularities caused by a phenomenon called `bumping'.
Actually, this kind of phenomenon also makes the deformation space not locally connected as was shown by Bromberg \cite{Bromp} and was generalised by Magid \cite{Mg}.

The interior of a deformation space is known to be a disjoint union of quasi-conformal deformation spaces of  minimally parabolic Kleinian groups, which is well understood by work of Ahlfors, Bers, Kra, Maskit, Marden, and Sullivan.
In particular, their theory gives rise to a parametrisation of the quasi-conformal deformation space by the Teichm\"{u}ller space of the boundary at infinity of the corresponding quotient hyperbolic $3$-manifold.
Therefore, to understand the global structure of a deformation space, what we need to know is how the boundary is attached to the quasi-conformal deformation space.
More concretely, we need to determine, for a  sequence of quasi-conformal deformations given as a sequence in the Teichm\"{u}ller space using this parametrisation, first whether it converges or not, and if it does, what is the limit of the sequence.
Also, we need to know in what cases sequences of quasi-conformal deformations can approach the same group from different directions as in the example of Anderson-Canary and McMullen.
This paper tries to answer such problems for the case of Kleinian surface groups using the technique of model manifolds of geometric limits which we developed in Ohshika-Soma \cite{OhSo}.


The interior of a deformation space of Kleinian surface group $AH(S)$ is the quasi-Fuchsian space $QF(S)$.
The parametrisation in this case is the Ahlfors-Bers map  $qf$  from $\abparam$ to $QF(S)$.
The celebrated work of Bers on compactification of Teichm\"{u}ller space (\cite{Be}) shows that in $AH(S)$, the subspace of $QF(S)$ in the form $\mathcal{T}(S) \times \{\mathrm{pt.}\}$ or $\{\mathrm{pt.}\} \times \mathcal{T}(\bar{S})$ is relatively compact.
Considering a direction in $QF(S)$ (with respect to the parametrisation of Ahlfors-Bers) quite different from that of Bers, Thurston proved that a sequence $\{(m_i ,n_i)\} \in \abparam$ converges if $\{m_i\}$ converges to a projective lamination $[\lambda]$ and $\{n_i\}$ converges to another projective lamination $[\mu]$  both in the Thurston compactification of $\teich(S)$ such that the supports of $\lambda$ and $\mu$ are distinct and every component of $S \setminus (\lambda \cup \mu)$ is simply connected (\cite{Th2}).

In contrast to these results on convergence, we showed in \cite{OhG} that if $\{m_i\}$ and $\{n_i\}$ converge to arational laminations with the same support, then $\{qf(m_i,n_i)\}$ always diverges.
Since the Teichm\"{u}ller space of $S$ is properly embedded in $AH(S)$ as a diagonal set of $\teich (S) \times \teich(\bar S)$, this may look very natural.
However, there is an example by Anderson-Canary in \cite{AC} which shows that for a given hyperbolic structure $m_0$, if we consider $(\tau^i(m_0), \tau^{2i}(m_0)) \in \teich(S) \times \teich(\bar S)$, where $\tau$ denotes the deformation of the metric induced by the Dehn twist around a simple closed curve $\gamma$ on $S$, then its image $qf(\tau^i(m_0), \tau^{2i}(m_0))$ in $AH(S)$ converges.
These two results show that the situation is quite different depending on the types of the limit projective laminations.

In the case when $\teich(S)$ has dimension $2$, i.e., if $S$ is either a once-punctured torus or a four-times punctured sphere, a measured lamination is either arational or a weighted simple closed curve, which means that there is nothing between these two situations above.
In this case, Ito has given a complete criterion for convergence/divergence (\cite{Ito}).
In the general case when $\teich(S)$ has dimension more than $2$ however, there is a big room between these two extremes.

Therefore quite naturally, we should ask ourselves what would happen in the cases in between.
One of our main theorems in this paper (Theorem \ref{main}) is an answer to this question.
Consider sequences $\{m_i\}$ in $\teich(S)$ and $\{n_i\}$ in $\teich(\bar S)$ converging to $[\mu^+]$ and $[\mu^-]$ such that their supports  $|\mu^+|$ and $|\mu^-|$ share a component which is not a simple closed curve.
We shall prove that $\{qf(m_i,n_i)\}$ diverges in $AH(S)$ in this setting.
More generally, we shall show that if $\mu^+$ and $\mu^-$ have components $\mu_0^+, \mu_0^-$ whose minimal supporting surfaces share a boundary component up to isotopy, then $\{qf(m_i,n_i)\}$ diverges.

This theorem is derived from another of our main results, Theorem \ref{main limit laminations}, which asserts that if $\{qf(m_i,n_i)\}$ converges, then we can determine the ending  laminations of the limit group by considering the shortest pants decompositions of $(S, m_i)$ and $(S, n_i)$ and their Hausdorff limits.
It also implies Theorem \ref{+-side --side} stating that if we consider the limits of $\{m_i\}$ and $\{n_i\}$ in the Thurston compactification of the Teichm\"{u}ller space, then every non-simple-closed-curve component of the limit of $\{m_i\}$  is an ending lamination of an upper end of the limit, and every non-simple-closed-curve component of the limit of $\{n_i\}$ is that of a lower end.
These combined with Theorem \ref{generalised Ito}  can be regarded as a partial answer to the problem of determining limit groups of sequence of quasi-Fuchsian groups given in terms of the parametrisation by the Teichm\"{u}ller spaces.

In the case when the laminations $|\mu^+|$ and $|\mu^-|$ share only simple closed curves,  the convergence or divergence of the sequence depends on whether there is  a simple closed curve component of  either $|\mu^+|$ or $|\mu^-|$ which is contained in the boundary of  the minimal  supporting surface of a non-simple-closed-curve component of the other of $|\mu^-|$ or $|\mu^+|$.
Theorem \ref{scc case} asserts that if there is such a simple closed curve component, then the sequence diverges.
Even in the case when such a component does not exist, $\{qf(m_i,n_i)\}$ can converge only in a special situation which is analogous to an example of Anderson-Canary \cite{AC}.
Theorem \ref{generalised Ito} describes the situation where the sequence can converge.
We should note that these are the best possible answers for divergence and convergence for sequence of quasi-Fuchsian groups when the asymptotic behaviour  of the corresponding parameters are expressed in terms of the Thurston compactification of Teichm\"{u}ller space.

The same example of Anderson-Canary also shows that there is a point in $AH(S)$ where $QF(S)$ bumps itself as explained above.
In such a point, the self-bumping is caused by what is called the \lq exotic convergence'.
A sequence of quasi-Fuchsian group is said to converge to a b-group exotically when the groups in the sequence are not contained in Bers slices approaching the one containing the limit b-group.
The construction of Anderson-Canary gives a sequence converging exotically to a regular b-group.
We shall prove such a convergence can occur only for b-groups which  have $\integers$-cusps not touching geometrically infinite ends.

As for self-bumping, we conjecture that  such a phenomenon cannot occur for  geometrically infinite b-groups all of whose $\integers$-cusps touch geometrically infinite ends.
What we shall prove in \cref{bumping theorem} is a weaker form of this conjecture:
 if there are two sequences $\{qf(m_i,n_i)\}$ and $\{qf(m_i', n_i')\}$ both converging to the same geometrically infinite group all of whose $\integers$-cusps touch geometrically infinite ends, then for any small neighbourhood $U$ of the   quasi-conformal deformation space of the limit group, if we take large $i$, then $qf(m_i,n_i)$ and $qf(m_i',n_i')$  are connected by an arc in $U$.
In particular this shows that under the same condition on $\integers$-cusps, if the limit group is either quasi-conformally rigid or a b-group whose upper conformal structure at infinity is rigid, it cannot be a self-bumping point.
More generally, even when there is a cusp not touching a geometrically infinite end,  the same argument shows that a Bers slice cannot bump itself at a b-group whose upper conformal structure at infinity is rigid.
This latter result has been obtained independently by Brock-Bromberg-Canary-Minsky \cite{BBCM} by a  different approach.

In Theorems \ref{generalised limit laminations} and \ref{generalised main}, we shall generalise the results for quasi-Fuchsian groups in Theorems \ref{main limit laminations} and \ref{main} to general Kleinian surface groups.

We also note that the results obtained in this paper has an application which has appeared in \cite{OhF} (see also Papadopoulos \cite{Papa2}).
There, we have considered a quotient space of a Bers boundary, which is called the reduced Bers boundary, and have proved its automorphism group coincides with the extended mapping class group.

Our tools for proving all these theorems are model manifolds invented by Minsky \cite{Mi} and their geometric limits studied in Ohshika-Soma \cite{OhSo}.
The technique which we develop in this paper shows that model manifolds are quite useful for studying the asymptotic behaviour of sequence in deformation spaces.
We have tried to make this paper readable independently of \cite{OhSo}:
each time we need arguments in \cite{OhSo}, we provide their sketches or summaries so that the reader does not need to look into details of  \cite{OhSo}.

Recently, after the first version of the present paper had been put on the arXiv, there appeared two papers which made a further progress along the line of our results, using different techniques.
The first is Brock-Bromberg-Canary-Minsky \cite{BBCM} and the other is Brock-Bromberg-Canary-Lecuire \cite{BBCL}.

The author would like to express his gratitude to the anonymous referee, whose comments and suggestions are very helpful to revise the text.
In particular, the referee's comments have made it possible to substantially shorten the argument in \S 4.
\section{Preliminaries}
\subsection{Generalities}
Kleinian groups are discrete subgroups of $\PSL_2 \complexes$.
In this paper we always assume Kleinian groups to be torsion free.
When we talk about deformation spaces, we only consider finitely generated Kleinian groups.
However, we also need to consider infinitely generated Kleinian groups which will appear as geometric limits.
We refer the reader to Marden \cite{Mar, Mar2} for a general reference for the theory of Kleinian groups.

Let $S$ be an oriented hyperbolic surface of finite area.
In this paper, we focus on Kleinian groups which are isomorphic to $\pi_1(S)$ in such a way that  punctures of $S$ correspond to parabolic elements.
We define the deformation space $AH(S)$ to be the quotient space of 
\begin{eqnarray*}
R(S)=&\{(G, \phi)\mid \phi: \pi_1(S) \rightarrow \PSL_2 \complexes \text{ is a faithful discrete representation} \\ &\text{taking punctures to parabolic elements with }\phi(\pi_1(S))=G\}
\end{eqnarray*}
by conjugacy in $\PSL_2 \complexes$.
The space $R(S)$ has a topology coming from the representation space and we endow $AH(S)$ with its quotient topology.
We denote an element of $AH(S)$ also by $(G, \phi)$ for some representative of the equivalence class.
We call $\phi$ a marking of the Kleinian group $G$.

The set of faithful discrete representations of $\pi_1(S)$ into $\PSL_2 \reals$, which induce the same orientation as the one given on $S$, modulo conjugacy  constitutes the Teichm\"{u}ller space of $S$, which we denote by $\mathcal T(S)$.
Therefore, $\mathcal T(S)$ is naturally contained in $AH(S)$.
More generally, the space of quasi-Fuchsian groups $QF(S)$ lies in $AH(S)$.
A quasi-Fuchsian group is a Kleinian group whose domain of discontinuity is a disjoint union of two simply connected components.
By the theory of Ahlfors-Bers, $QF(S)$ is parametrised by a homeomorphism $qf: \mathcal T(S) \times \mathcal T(\bar S) \rightarrow QF(S)$.
Here both $\mathcal T(S)$ and $\mathcal T(\bar S)$ are the Teichm\"{u}ller space of $S$, but the latter one is identified with $\mathcal T(S)$ by an orientation reversing automorphism of $S$.
For $(m,n) \in \mathcal T(S) \times \mathcal T(\bar S)$, its image $qf(m,n)$ is obtained by starting from a Fuchsian group and solving a Beltrami equation so that the conformal structure on the quotient of  the lower Jordan domain is $m$ and that on the quotient of the upper Jordan domain is $n$.
We call $m$ and $n$  the lower and upper {\em conformal structures at infinity} of $qf(m,n)$ respectively.
The set of Fuchsian groups corresponds to a slice of the form $\{qf(m, m)\}$.
By the theory of Ahlfors-Bers combined with Sullivan's stability theorem \cite{Su2}, we know that $QF(S)$ is the interior of the entire deformation space $AH(S)$.
On the other hand, the Bers-Sullivan-Thurston density conjecture, which was solved by Bromberg \cite{Br}, Brock-Bromberg \cite{BB} in this setting, or is obtained as a corollary of the ending lamination conjecture \cite{BCM} combined with \cite{OhI}, $AH(S)$ is the closure of $QF(S)$.

By Margulis' lemma, there is a positive constant $\epsilon_0$ such that for any hyperbolic $3$-manifold, its $\epsilon_0$-thin part is a disjoint union of cusp neighbourhoods and tubular neighbourhoods of short closed geodesics, which are called Margulis tubes.
For a hyperbolic 3-manifold $M$, we denote by $M_0$ the complement of its cusp neighbourhoods.

Bonahon showed in \cite{BoA} that for any $(G, \phi) \in AH(S)$, the hyperbolic 3-manifold $\hyperbolic^3/G$ is homeomorphic to $S \times (0,1)$.
We denote by $\Phi$ a homeomorphism from $S \times (0,1)$ to $\hyperbolic^3/G$ inducing $\phi$ between the fundamental groups.
In general, for an element in $AH(S)$, we denote a homeomorphism from $S \times (0,1)$ to the quotient hyperbolic 3-manifold by the letter in the upper case corresponding to a Greek letter denoting the marking.
Bonahon also proved the every end of $(\hyperbolic^3/G)_0$ is either geometrically finite or simply degenerate.
Here an end is called {\em geometrically finite} if it has a neighbourhood which is disjoint from any closed geodesic, and {\em simply degenerate} if it has a neighbourhood of the form $\Sigma \times (0,\infty)$ for an incompressible subsurface $\Sigma$ of $S$ (\ie a subsurface each of whose frontier components is non-contractible in $S$) and there are simple closed curves $c_n$ on $\Sigma$ which are homotopic to closed geodesics $c_n^*$ going to the end.
For a simply degenerate end, the {\em ending lamination} is defined to be the support of the projective lamination to which $[c_n]$ converges in the projective lamination space $\PML(\Sigma)$.
We shall explain what  are laminations and the projective lamination space below.

For $(G,\phi) \in AH(S)$, we choose an embedding $f: S \rightarrow \hyperbolic^3/G$ inducing $\phi$ between the fundamental groups.
Such an embedding is unique up to ambient isotopy in $\hyperbolic^3/G$.
An end of $\hyperbolic^3/G$ is called {\em upper} when it lies above $f(S)$ and {\em lower} when it lies below $f(S)$ with respect to the orientations of $\hyperbolic^3/G$ and $f(S)$.
Let $C$ be a compact core of $(\hyperbolic^3/G)_0$ intersecting each component of $\Fr (\hyperbolic^3/G)_0$ by a core annulus.
Let $P$ denote $C \cap \Fr (\hyperbolic^3/G)_0$.
Each component of $P$ is called a {\em parabolic locus} of $\hyperbolic^3/G$ or $C$, and its core curve is called a {\em parabolic curve}.
Parabolic loci corresponding to punctures of $S$ are called {\em peripheral parabolic loci}.
For non-peripheral parabolic loci, those contained in $S \times \{1\}$ are called upper and those contained in $S\times \{0\}$ lower.
The ends of $(\hyperbolic^3/G)_0$ correspond one-to-one to the components of $\Fr C$.
Since $C$ is homeomorphic to $S \times [0,1]$, each upper end faces a subsurface of $S \times \{1\}$, which is a component of $S \times \{1\} \setminus P$.
A non-peripheral parabolic locus $p$ in $P$ is called {\em isolated} when the components of $S \times \{0,1\} \setminus P$ adjacent to $p$ (there are one or two such components) face geometrically finite ends, in other words, no component of $S \times \{0,1\} \setminus P$ facing a simply degenerate end touches $p$ at its frontier.

\subsection{Laminations}
A geodesic lamination on a hyperbolic surface $S$ is a closed subset of $S$ consisting of disjoint simple geodesics, which are called leaves.
For a geodesic lamination $\lambda$, each component of $S \setminus \lambda$ is called a complementary region of $\lambda$.
We say that a geodesic lamination is {\em arational} when every complementary region is either simply connected or in the case when $S$ is not closed, a topologically open annulus whose core curve is homotopic to a puncture or a boundary component of $S$.
A subset $l$ of a geodesic lamination consisting of leaves is called a {\em minimal component} when for each leaf $\ell$ of $l$, its closure $\bar \ell$ coincides with $l$.
Any geodesic lamination is decomposed into finitely many minimal components and finitely many non-compact isolated leaves.

A measured lamination is a (possibly empty) geodesic lamination endowed with a transverse measure which is invariant with respect to homotopies along leaves.
The support of the transverse measure of a measured lamination $\mu$, which is a geodesic lamination, is called the support of $\mu$ and is denoted by $|\mu|$.
When we consider a measured lamination, we always assume that its support is the entire lamination.
The measured lamination space $\ML(S)$ is the set of measured laminations on $S$ endowed with the weak topology.
Thurston proved that $\ML(S)$ is homeomorphic to $\reals^{6g-6+2p}$, where $g$ is the genus and $p$ is the number of punctures.
A weighted disjoint union of closed geodesics can be regarded as a measured lamination.
It was shown by Thurston that the set of weighted disjoint unions of closed geodesics is dense in $\ML(S)$.

The projective lamination space $\PML(S)$ is the space obtained by taking a quotient of $\ML(S) \setminus \{\emptyset\}$ identifying scalar multiples.
Thurston constructed a natural compactification of the Teichm\"{u}ller space whose boundary is $\PML(S)$ in such a way that the mapping class group acts continuously on the compactification.

We need to consider one more space, the unmeasured lamination space.
This space, denoted by $\mathcal{UML}(S)$, is defined to be the quotient space of $\ML(S)$, where two laminations with the same support are identified.
An element in $\mathcal{UML}(S)$ is called an unmeasured lamination.

For a minimal geodesic lamination $\lambda$ on $S$, its {\em minimal supporting surface} is defined to be an incompressible subsurface of $S$  containing $\lambda$ which is minimal up to isotopes among all such surfaces.
When $\lambda$ is a closed geodesic, we define its minimal supporting surface to be an annulus whose core curve is $\lambda$.
It is obvious the minimal supporting surface of $\lambda$ is uniquely determined up to isotopy. 

\subsection{Algebraic convergence and geometric convergence}
When $\{(G_i, \phi_i)\}$ converges to $(\Gamma, \psi) \in AH(S)$, we say that the sequence converges {\em algebraically} to $(\Gamma, \psi)$.
We can choose representatives $(G_i, \phi_i)$ so that $\phi_i$ converges to $\psi$ as representations.
As a convention, when we say that $\{(G_i, \phi_i)\}$ converges to $(\Gamma , \psi)$, we always take representatives so that $\{\phi_i\}$ converges to $\psi$.

We need to consider another kind of convergence: the geometric convergence.
A sequence of Kleinian groups $\{G_i\}$ is said to converge to a Kleinian group $G_\infty$ {\em geometrically} if (1) for any convergent sequence $\{\gamma_{i_j} \in G_{i_j}\}$, its limit lies in $G_\infty$, and (2) any element $\gamma \in G_\infty$ is a limit of some $\{g_i \in G_i\}$.
When $(G_i, \phi_i)$ converges to $(\Gamma, \psi)$ algebraically, the geometric limit $G_\infty$ contains $\Gamma$ as a subgroup.

When $\{G_i\}$ converges to $G_\infty$ geometrically, if we take a basepoint $x$ in $\hyperbolic^3$ and its projections $x_i \in \hyperbolic^3/G_i$ and $x_\infty \in \hyperbolic^3/G_\infty$, then $(\hyperbolic^3/G_i, x_i)$ converges to $(\hyperbolic^3/G_\infty, x_\infty)$ with respect to the pointed Gromov-Hausdorff topology:
that is, there exists a $(K_i, r_i)$-approximate isometry $B_{r_i}(\hyperbolic^3/G_i,x_i) \rightarrow B_{K_i r_i}(\hyperbolic^3/G_\infty, x_\infty)$ with $K_i \rightarrow 1$ and $r_i \rightarrow \infty$, where $B_{r_i}(\hyperbolic^3/G_i,x_i)$ denotes the $r_i$-metric ball centred at $x_i$.
Here a $(K_i, r_i)$-approximate isometry $\rho_i$ is a diffeomorphism from $B_{r_i}(\hyperbolic^3/G_i,x_i)$ to $B_{K_i r_i}(\hyperbolic^3/G_\infty, x_\infty)$ satisfying $K_i^{-1} d_{\hyperbolic^3/G_i}(y,z) \leq d_{\hyperbolic^3/G_\infty}(\rho_i(y), \rho_i(z)) \leq K_i d_{\hyperbolic^3/G_\infty}(y,z)$ for every $y,z \in B_{r_i}(\hyperbolic^3/G_i,x_i)$.

\subsection{Bers slices and b-groups}
We fix a point $m_0 \in \mathcal T(S)$ and consider a subspace $qf(\{m_0\} \times \mathcal T(\bar S))$ in $QF(S)$.
This space is called the {\em Bers slice} over $m_0$.
Kleinian groups lying on its frontier are called {\em b-groups} with lower conformal structure $m_0$.

Anderson-Canary \cite{AC} constructed a sequence of quasi-Fuchsian groups converging to a b-group whereas its coordinates in $QF(S)$ do not approach a Bers slice.

\begin{definition}
We say a sequence of quasi-Fuchsian groups $\{(G_i, \phi_i) =qf(m_i,n_i)\}$ converges {\em exotically} to a b-group $(\Gamma,\psi) $ if $\{(G_i, \phi_i)\}$ converges to $(\Gamma, \psi)$ algebraically and both $\{m_i\}$ and $\{n_i\}$ go out from any compact set in the Teichm\"{u}ller space.
\end{definition}

The existence of exotic convergence is related to  singularities of $AH(S)$ at its boundary (as a subspace of the representation space modulo conjugacy).
In fact, McMullen showed in \cite{Mc} that $AH(S)$ has a singular point at the boundary where $QF(S)$ bumps itself.
For a Kleinian surface group $(\Gamma, \psi) \in AH(S) \setminus QF(S)$, we say that $QF(S)$ {\em bumps itself} at $(\Gamma, \psi)$, when there is a neighbourhood $V$ of $(\Gamma, \psi)$ such that for any  smaller neighbourhood $U \subset V$, its intersection with the quasi-Fuchsian space, $U \cap QF(S)$ is disconnected.
McMullen showed the existence of points where $QF(S)$ bumps itself, making use of the construction of Anderson-Canary.

We say that a Bers slice $\mathcal B(m_0)=qf(\{m_0\} \times \mathcal T(\bar S))$ bumps itself at $(\Gamma, \psi)$ when there is a neighbourhood $V$ of $(\Gamma, \psi)$ such that for any  smaller neighbourhood $U \subset V$, its intersection with the Bers slice, $U \cap \mathcal B(m_0)$ is disconnected.
Up to today, it is not known whether a Bers slice can bump itself or not.

\subsection{Curve complexes and hierarchies}
\label{hierarchy}
In this subsection, we shall give an explanation of the work of Masur-Minsky on what they called {\em hierarchies of tight geodesics.}
Throughout this subsection, we fix an orientable hyperbolic surface $S$ of finite type, and consider its subsurfaces.
We say that a subsurface is {\em essential} when it is a proper subsurface and each of its frontier component is a non-contractible, non-peripheral curve on $S$.
When we talk about curve complexes, we usually regard essential subsurfaces as  open subsurfaces without boundary, and call them {\em domains} of $S$ following Masur-Minsky.
For a surface $S$ of genus $g$ with $p$ punctures, we define $\xi(S)$ to be $3g+p$.
We shall define curve complexes for orientable surfaces $S$ with $\xi(S) \geq 4$ or $\xi(S)=2$.
The {\em curve complex} $\CC(S)$ of $S$ with $\xi(S) \geq 5$ is defined as follows.
(This notion was first introduced by Harvey \cite{Har}.)
The vertices of $\CC (S)$, whose set we denote by $\CC_0(S)$, are the homotopy classes of essential simple closed curves on $S$.
A collection of $n+1$ vertices $\{v_0, \dots , v_n\}$ spans an $n$-simplex if and only if they can be represented as pairwise disjoint simple closed curves on $S$.

In the case when $\xi(S)=4$, the curve complex is a $1$-dimensional simplicial complex.
The vertices are the homotopy classes of essential simple closed curves as in the case when $\xi(S) \geq 5$.
Two vertices $v_0, v_1$ are connected by an edge if their intersection number is $1$ when $S$ is a once-punctured torus, and $2$ when $S$ is a four-times-punctured sphere.

In the case when $\xi(S)=2$, we consider a compactification of $S$ to an annulus.
The curve complex is a $1$-dimensional and the vertices are the homotopy classes (relative to the endpoints) of essential arcs with both endpoints on the boundary.
Two curves are connected by an edge if they can be made disjoint in their interiors.

A {\em tight sequence} in $\CC (S)$ with $\xi(S) \geq 5$ is a sequence of simplices $\{s_0, \dots, s_n \}$ with the first one and the last one being vertices
such that for any vertices $v_i \in s_i$ and $v_j \in s_j$, we have $d_{\CC(S)}(v_i, v_j)=|j-i|$, and $s_{j+1}$ is homotopic to the union of essential boundary components of a regular neighbourhood of $s_j \cup s_{j+2}$.
We also consider an infinite tight sequence such as $\{s_0, \dots\}$ or $\{\dots, s_0\}$ or $\{\dots, s_0, \dots\}$.
In the case of the surface with $\xi(S)=2\text{ or } 4$, we consider a sequence of vertices and ignore the second condition.

For an essential simple closed curve $c$ of $S$, we consider the covering $A_c$ of $S$ associated to the image of $\pi_1(c)$, which is an open annulus.
If we fix a hyperbolic metric on $S$, we can compactify the hyperbolic annulus $A_c$ to an annulus $\bar A_c$ by regarding  $\pi_1(c)$ as acting on $\hyperbolic^2$ and considering the quotient of $\hyperbolic^2 \cup \Omega_{\pi_1(c)}$ by $\pi_1(c)$, where $\Omega_{\pi_1(c)}$ denotes the region of discontinuity of the action of $\pi_1(c)$ on the circle at infinity of $\hyperbolic^2$.
We call an essential simple arc with endpoints on $\partial \bar A_c$ a {\em transversal} of $c$.
A simple closed curve $d$ intersecting $c$ essentially induces $i(c, d)$-many transversals of $c$.
Any two transversals induced from $d$ are within the distance $1$ in $\CC(A_c)$.

For a non-annular domain $\Sigma$ in $S$, we define $\pi_\Sigma : \CC_0(S) \rightarrow \frak P(\CC(\Sigma)) \cup \{\emptyset\}$ to be a map sending $c \in \CC_0(S)$ to the set of essential simple closed curves obtained by connecting the endpoints of each component of $c \cap \Sigma$ by arcs on $\Fr \Sigma$ in a consistent way if $c$ intersects $\Sigma$ essentially.
We define $\pi_\Sigma(c)$ to be $\emptyset$ when no essential simple closed curves are obtained from $c\cap \Sigma$.
When $\Sigma$ is an annulus we define $\pi_\Sigma(c)$ to be a transversal of the core curve of $\Sigma$ induced from  $c$  if $c$ intersects $\Sigma$ essentially.
(See \S2.3 of Masur-Minsky \cite{MaMi} for details.)

A {\em marking} $\mu$ on a surface $S$ consists of a simplex in $\CC(S)$ and transversals on some of its vertices (at most one for each).
The vertices of the simplex are called the {\em base curves} of $\mu$, and their union is denoted by $\base(\mu)$.
A marking $\mu$ is said to be {\em clean} if every component $c$ of  $\base(\mu)$ has a transversal and it is induced by a simple closed curve with intersection number $1$ when $c$ is non-separating and with intersection number $2$ if $c$ is separating, which is disjoint from the other components of $\base(\mu)$. 
A clean marking $\mu'$ is said to be compatible with a marking $\mu$ when $\base(\mu)=\base(\mu')$, and every transversal of a component $c$ in $\mu$ is within the distance $2$ from the transversal of $c$ in $\mu'$ as vertices in the curve complex of an annulus with core curve $c$.
A marking is called {\em complete} if its base curves constitute a pants decomposition of $S$ and every base curve has a transversal.

For a marking $\mu$ and a non-annular domain $\Sigma$ of $S$ whose frontier does not intersect $\base(\mu)$ transversely, we define $\mu|\Sigma$ to be a marking on $\Sigma$ whose base curves are those in $\base(\mu) \cap \Sigma$ and whose transversals are those induced from the transversals of $\base(\mu)$.
When $\Sigma$ is an annulus, there are two cases where $\mu|\Sigma$ is defined. 
One is when  $\mu$ intersects the core curve of $\Sigma$ transversely, in which case we define $\mu|\Sigma$ to be $\pi_\Sigma(\mu)$.
The other is when there is a component $b$ of $\base(\mu)$ which is a core curve of $\Sigma$, in which case we define $\mu|\Sigma$ to be the transversal of $b$. 

To deal with the case of geometrically infinite groups, we need a notion of generalised markings.
A {\em generalised marking} consists of an unmeasured lamination on $S$ and transversals on some of its components which are simple closed curves.
Also for a generalised marking $\mu$, we denote the unmeasured  lamination by $\base(\mu)$ and call it the {\em base lamination}.
We say that a generalised marking $\mu$ is {\em complete}, if its base lamination is maximal, \ie it is not a proper sublamination of another unmeasured lamination, and every simple closed curve in $\base(\mu)$ has a transversal.
From now on, we always assume {\em markings and generalised markings to be complete}.

A  finite {\em tight geodesic} on a surface $\Sigma$ is a triple $(g, I(g), T(g))$, where $g$ is a tight sequence, and $I(g)$ and $T(g)$ are generalised markings on $\Sigma$ whose base laminations have at least one simple closed curve component, such that the first vertex is a simple closed curve component of the $\base (I(g))$ and the last vertex is that of $\base (T(g))$.
The surface $\Sigma$ is called the {\em support} of $g$ and we write $\Sigma =D(g)$.
An infinite tight geodesic is defined similarly just letting $T(g)$ be  an arational unmeasured lamination to which $s_i$ converges as $i \rightarrow \infty$ (in the quotient topology of the unmeasured lamination space induced from the measured lamination space) when $g$ is in the form of $\{s_0, \dots\}$.
Similarly, by letting $I(g)$ be an arational unmeasured lamination to which $s_i$ converges as $i \rightarrow -\infty$ when $g$ is in the form of $\{\dots , s_0\}$, and by letting both  $I(g)$ and $T(g)$ be arational unmeasured laminations which are limits of $\{s_i\}$ as $i \rightarrow -\infty$ and as $i \rightarrow \infty$ respectively when $g$ is in the form of $\{\dots, s_0, \dots\}$. 
Refer to \S\ref{at infinity} for more explanations on the boundary of $\CC(\Sigma)$.

Let $\Sigma$ be a domain of $S$.
For a simplex $s$ in $\CC(\Sigma)$, a {\em component domain } of $s$ is defined to be either a component  of $\Sigma \setminus s$ or an annulus whose core curve is a  component of $s$.
We consider only one annulus for each component of $s$.
Let $g$ be a tight geodesic in $\CC(\Sigma)$, and suppose that $s$ is a simplex on $g$.
Let $\Sigma'$ be a component domain of $s$.
(Such a domain is also said to be a component domain of $g$.)
Then, following Masur-Minsky \cite{MaMi}, we define $T(\Sigma', g)$ to be $\suc(s)|\Sigma'$ if $s$ is not the last vertex of $g$, and to be $T(g)|\Sigma'$ if $s$ is, where $\suc(s)$ denotes the simplex of $g$ following $s$.
Similarly, we define $I(\Sigma', g)$ to be $\pre(s)|\Sigma'$ if $s$ is not the first vertex of $g$, and to be $I(g)|\Sigma'$ if $s$ is, where $\pre(s)$ denotes the simplex of $g$ preceding $s$.
We write $\Sigma' \subord g$ or $\Sigma' \subord (g, s)$, if $T(\Sigma', g)$ is non-empty, and $g \supord \Sigma'$ or $(g,s) \supord \Sigma'$ if $I(\Sigma', g)$ is non-empty.
If a geodesic $k$ is a tight geodesic supported on $\Sigma'$ with $(g,s) \supord \Sigma'$ and $I(k)=I(\Sigma', g)$, then we write $g \supord k$ or $(g,s) \supord k$, and say that $k$ is {\em directly backward subordinate} to $g$ at $s$.
Similarly, if $\Sigma' \subord (g,s)$ and $T(k)=T(\Sigma', g)$, we write $k \subord g$ or $k \subord (g,s)$, and say that $k$ is {\em directly forward subordinate} to $g$ at $s$.

A {\em hierarchy} $h$ on $S$, which was introduced by Masur-Minsky \cite{MaMi}, is a family of tight geodesics supported on domains in $S$ having the following properties.
\begin{enumerate}
\item
There is a unique geodesic $g_h$ supported on $S$.
\item
For any $g \in h$ other than $g_h$, there are geodesics $b,f \in h$ with $b \supord g \subord f$.
\item
For any $b,f \in h$ and a component domain $\Sigma$ of $b, f$ with $b \supord \Sigma \subord f$ ($b$ and $f$ may coincide), there is a unique geodesic $k$ supported on $\Sigma$ with $b \supord k \subord f$.
\end{enumerate}
A hierarchy $h$  is said to be {\em complete} if every component domain of geodesics in $h$ supports a geodesic in $h$, and {\em $4$-complete} if every non-annular component domain of geodesics in $h$ supports a geodesic in $h$.


We write $g \searrow (f,v)$ if there is a sequence of geodesics in $h$ such that $g=f_0 \subord f_1 \subord \dots \subord (f_n,v)=(f,v)$, and say that $g$ is forward subordinate to $f$.
Similarly, we write $(b,u) \swarrow g$ if there is a sequence in $h$ such that $(b,u)=(b_m,u) \supord \dots \supord b_1 \supord b_0=g$, and say that $g$ is backward subordinate to $b$.
We use symbol $\seeq$ to mean either $\searrow$ or $=$ and $\sweq$ to mean either $\swarrow$ or $=$.

In \S\ref{sec:self-bumping}, we shall use the notions of slices and resolutions of hierarchies invented by Masur-Minsky \cite{MaMi}.
We shall review them briefly here.
Let $h$ be a complete or $4$-complete hierarchy.
A {\em slice} $\sigma$ of $h$ is a set of pairs $(g,v)$, where $g$ is a geodesic in $h$ and $v$ is a simplex on $g$ satisfying the following conditions.
(Masur and Minsky call  $\sigma$ satisfying the first three conditions a slice, and call it a complete slice  if it also satisfies the fourth condition.)
\begin{enumerate}
\item A geodesic can appear at most in one pair of $\sigma$.
\item There is a pair whose first entry is the main geodesic of $h$.
\item For each pair $(g,v)$ in $\sigma$ such that $g$ is not the main geodesic, $D(g)$ is a component domain of   a simplex $v'$ for some $(g',v') \in \sigma$.
\item For each component domain $D$ of $v$ for $(g,v) \in \sigma$ with $\xi(D) \neq 3$ if $h$ is complete and $\xi(D)>3$ if $h$ is $4$-complete, there is a pair $(g',v')\in \sigma$ with $D(g')=D$.
\end{enumerate}

Masur and Minsky introduced two kinds of order, $\prec_p$ between pairs of geodesics and simplices in $h$  and $\prec_s$ between slices.
For two pairs $(g,v)$ and $(g',v')$ of a $4$-complete hierarchy, we write $(g,v) \prec_p (g',v')$ if either $g=g'$ and $v'$ comes after $v$, or there is a geodesic $g''$ with $(g,v) \seeq (g'',w)$ and $(g'', w') \sweq (g',v')$ such that $w'$ is a simplex coming after $w$.
For two distinct slices $\sigma$ and $\tau$, we write $\sigma \prec_s \tau$ if for any $(g,v) \in \sigma$, either $(g,v) \in \tau$ or there is $(g',v') \in \tau$ with $(g,v) \prec_p (g',v')$.

A {\em resolution} $\tau=\{\sigma_i\}$ of a $4$-complete hierarchy $h$ is an ordered sequence of slices of $h$ such that $\sigma_{i+1}$ is obtained from $\sigma_i$ by an elementary forward move.
Here an elementary forward move is a change of pairs in $\sigma_i$ as follows:
We advance $(g,v) \in \sigma_i$ to $(g, \suc(v))$ under the condition that for every pair $(g',v')$ supported on a component domain of $v$ into which $\suc(v)$ is projected to an essential curve, the simplex $v'$ is the last vertex, and after removing all such $(g',v')$ we add pairs $(g'',v'')$ such that  $g''$ is supported on a component domain of $\suc(v)$ into which $v$ is projected to an essential curve and $v''$ is the first vertex of $g''$.

\subsection{Model manifolds}
\label{model}
A model manifold for a Kleinian surface group was constructed in Minsky \cite{Mi} as follows.
Let $G$ be a Kleinian surface group with $M=\hyperbolic^3/G$.
From an end invariant of $G$, we shall construct a hierarchy of tight geodesics, which we denote by $h_G$.
In the case when $G$ is quasi-Fuchsian, we construct a hierarchy by defining the initial and terminal markings to be the shortest clean markings with respect to the upper and lower conformal structures at infinity.
When $M$ has a totally degenerate end without accidental parabolics, we define the initial or terminal generalised marking  to be its ending lamination.
In general, we consider the union of ending laminations of $M_0$, parabolic loci for upper or lower ends, and shortest clean markings on the remaining geometrically finite upper or lower ends, and let them be terminal or initial generalised markings.
Here, we say that  a clean marking $\mu$ on a hyperbolic surface $S$ is shortest if the base curves of $\mu$ form a shortest pants decomposition of $S$, and transversals are chosen so that their total length is the smallest among the transversals obtained from them by performing Dehn twists around the base curves.
Note that we are {\em not} assuming the transversals of $\mu$ are really shortest among all transversals.

Having defined the hierarchy $h_G$, we construct a resolution $\{\tau_i(G)\}$ of $h_G$.\label{boundary block}
In the resolution, we look at each step $\tau_i(G) \rightarrow \tau_{i+1}(G)$ that advances a vertex on a $4$-geodesic, from $w_i$ to $w_{i+1}$.
For such a step, we provide an {\em internal block}, which is topologically homeomorphic to $\Sigma \times [0,1]$, where $\Sigma$ is either a sphere with four holes or a torus with one hole.
The block has two ditches, one on the top and the other on the bottom, corresponding to the two vertices $w_i$ and $w_{i+1}$.
To be more precise, we take annular neighbourhoods $N_i, N_{i+1}$ of $w_i, w_{i+1}$ on $\Sigma$, and set
 $B=\Sigma \times [0,1] \setminus (N_i \times [0,1/4] \cup N_{i+1} \times [3/4, 1])$.
The top and bottom boundary of a block consists of pairs of sphere with three holes.
We fix some constant $\epsilon_1$ less then the Bers constant for $S$ throughout the construction.
We put a hyperbolic metric on $\Sigma$ so that the lengths of the boundary components, $w_i$ and $w_{i+1}$ are all equal to  $\epsilon_0$.
We also assume that  $N_i$ and $N_{i+1}$ are regular neighbourhoods whose boundaries have length $\epsilon_1$ and deform their metrics to flat ones keeping their boundaries fixed.
(Here $\epsilon_2$ is chosen so that for pants decomposition of $S$ by simple closed geodesics with length less than $\epsilon_1$,  their annular neighbourhoods whose boundaries have length $\epsilon_2$ are pairwise disjoint.)
We define the metric on $B$ to be the one induced by the product of the hyperbolic metric on $\Sigma$ as above and the ordinary metric on $[0,1]$.
We note that the isometry type of $B$ depends only on whether $\Sigma$ is a sphere with four holes or a torus with one hole.

The model manifold for $G$ is constructed by piling up such blocks by pasting a top component of one block to a bottom component of another, according to the information given by the resolution $\{\tau_i(G)\}$,  attaching {\em boundary blocks} to the top and the bottom of the piled up blocks if there are geometrically finite ends of $M_0$, which have special forms and are constructed according to conformal structures at infinity of $G$, and then finally filling in \lq Margulis tubes'.
(Here we abuse the term \lq Margulis tube', which was defined before using the Margulis constant $\epsilon_0$.
Our tubes here may have larger injectivity radii, but still are isometric to tubular neighbourhoods of closed geodesics.)
In this paper, we define a boundary block to have topologically a form $\Sigma \times [s,t)$ or $\Sigma \times (t,s]$ for some incompressible subsurface $\Sigma$ of $S$ (\ie a subsurface $\Sigma$ such that every component of $\Fr \Sigma$ is non-contractible in $S$), and do not put extra-annuli as in Minsky's definition.

To be more precise, a boundary block has the following form.
We describe it here only when the block corresponds to an upper end.
We can deal with the case when the end is lower just by turning everything upside down.
Let $n_0$ be a point in $\teich(\Sigma)$, where each component of $\Fr \Sigma$ is assumed to be a puncture, and regard it as a hyperbolic metric on $\Int \Sigma$ with each component of $\Fr \Sigma$ assumed to be the boundary of an $\epsilon_0$-cusp neighbourhood.
Take a shortest pants decomposition of $(\Sigma, n_0)$ and denote its components by $c_1, \dots , c_p$.
A boundary block $B$ is  constructed  by defining $B=\Sigma \times [-1,\infty) \setminus (\cup_{k=1}^p A(c_k) \times [-1,0))$, where $A(c_k)$ is an open annular neighbourhood  of $c_k$ whose boundaries have length $\epsilon_1$.
We now put a Riemannian metric on $B$ as follows.
Since each component of $\Sigma \times \{0\} \setminus (\cup_{k=1}^p A(c_k) \times \{0\})$ is a pair of pants, we put a standard hyperbolic metric metric $n_0$ so that each boundary component becomes a closed geodesic of length $\epsilon_0$.
Now, as was shown in \S 3.4 of Minsky \cite{Mi}, there is a metric $n_0'$ on $\Sigma$ conformal to $n_0$ in which the $A(c_k)$ are flat annuli and whose restriction to each component of $\Sigma \setminus (\cup_{k=1}^p A(c_i))$ coincides with the hyperbolic metric $n_0$.
We put this metric $n_0'$ on $\Sigma \times \{0\}$, and $n_0|(\Sigma\setminus (\cup_{k=1}^p A(c_i)))$  on $\Sigma \times \{-1\} \setminus (\cup_{k=1}^p A(c_k) \times \{-1\})$.
On $\Sigma \times [-1,0) \setminus (\cup_{k=1}^p A(c_k) \times [-1,0))$, we put the product of the metric of $dn_0$ and $dt$.
Now, on $\Sigma \times [0,1]$, we put a metric defined by $e^{2t} d(n_0')^2+ dt^2$.

We do not put extra-annuli which appeared in Minsky's construction because we are constructing a model manifold of the {\em non-cuspidal part}, not of the entire manifold.
By the same reason, in contrast to Minsky's original construction, we  do not fill in cusp neighbourhoods.
Each slice in $\{\tau_i(G)\}$ corresponds to a {\em split level surface} in the model manifold which is a disjoint union of horizontal surfaces in blocks which are spheres with three holes.
Taking split level surfaces into pleated surfaces and extending them over Margulis tubes, a homotopy equivalent map from the model manifold to $M_0$ is constructed.
This can be modified to a uniform bi-Lipschitz map which is called {\em a model map} to $M_0$.
See  Minsky \cite{Mi} and Brock-Canary-Minsky \cite{BCM} for more details.
There is an alternative approach to constructing model manifolds by Bowditch \cite{Bow1}, \cite{Bow2} and \cite{Bow3}.
\subsection{The boundaries at infinity of curve complexes}
\label{at infinity}
It was proved by Masur-Minsky \cite{MaMi1} that $\CC(S)$ is a Gromov hyperbolic space with respect to the path metric defined by setting every edge to have the unit length.
For a Gromov hyperbolic space, its boundary at infinity can be defined as a topological space.
(Refer for instance to Coornaert-Delzant-Papadopoulos \cite{CDP}.)
Klarreich in \cite{Kl} showed that the  boundary at infinity of $\CC(S)$ is the space of ending laminations: that is, the space of arational unmeasured laminations with topology induced from $\mathcal{UML}(S)$.
This space is denoted by $\EL(S)$.

We shall show the following lemma, which is an easy consequence of the definition of the topology of $\CC(\Sigma) \cup \EL(\Sigma)$.
\begin{lemma}
\label{endpoints converge}
Let $\{g_i\}$ be a sequence of geodesics in $\CC(\Sigma)$ converging to a geodesic ray $g_\infty$ uniformly on every compact set.
Then the last vertex of $g_i$ converges to the endpoint at infinity of $g_\infty$ with respect to the topology of $\CC(\Sigma) \cup \EL(\Sigma)$.
\end{lemma}
\begin{proof}
Let $\lambda$ be a measured lamination whose support is the endpoint at infinity of $g_\infty$.
We can assume that all the $g_i$ have the same initial vertex, which we denote by $v$.
Let $w_i$ be the last vertex of $g_i$.
Since the length of $g_i$ goes to infinity, the distance between $v$ and $w_i$ goes to infinity.
On the other hand, since $\{g_i\}$ converges to $g_\infty$ on every compact set, there is a number $n_{i,j}$ going to $\infty$ such that the first $n_{i,j}$ simplices of $g_i$ and $g_j$ are the same.
Since $(w_i|w_j)_v \geq n_{i,j}$, we see that $(w_i|w_j)_v$ goes to $\infty$ as $i, j \rightarrow \infty$.
Therefore,  $\{w_i\}$ converges to some ending lamination after passing to a subsequence.
By the definition of the topology on $\CC(\Sigma) \cup \EL(\Sigma)$, there is a measured lamination $\mu$ and 
 positive real numbers $r_i$ such that $\{r_i w_i\}$ converges to $\mu$.

We need to show that $|\mu|=|\lambda|$.
Suppose not.
Since $\{g_i\}$ converges to $g_\infty$ uniformly on every compact set, we can take a simplex $v_i \in g_\infty$ which is also contained in $g_i$ tending to $|\lambda|$ in $\CC(\Sigma) \cup \EL(\Sigma)$.
Since $|\lambda|$ and $|\mu|$ are distinct points on the boundary at infinity, we have $\limsup_{i \rightarrow \infty} (v_i|w_i)_v < \infty$.
This contradicts the facts that both $v_i$ and $w_i$ lie on the same geodesic $g_i$ and that both $d(v,v_i)$ and $d(v,w_i)$ go to $\infty$.
\end{proof}

\section{The  main results}
\label{main results}
In this section, we shall state our main theorems.
\subsection{End invariants of limit groups}
We shall first state a theorem showing that for a limit of quasi-Fuchsian groups, the limit laminations of upper conformal structures at infinity appear as  ending laminations of upper ends whereas the limit of lower ones appear as ending laminations of lower ends.

\begin{theorem}
\label{+-side --side}
Let $\{(m_i,n_i)\}$ be a sequence in $\mathcal T(S) \times \mathcal T(\bar S)$ such that $\{qf(m_i,n_i)\}$ converges to $(\Gamma, \psi)$ in $AH(S)$.
Let $[\mu^+]$ and $[\mu^-]$ be projective laminations which are limits of $\{m_i\}$ and $\{n_i\}$ in the Thurston compactification of the Teichm\"{u}ller space passing to subsequences.
Then every component of $|\mu^+|$ that is not a simple closed curve  is the ending lamination of an upper end of $(\hyperbolic^3/\Gamma)_0$  whereas every component of $|\mu^-|$ that is not a simple closed curve is the ending lamination of a lower end.
%
\end{theorem}

Theorem \ref{+-side --side} will be obtained by combining the following theorem with a simple lemma regarding the Thurston compactification of Teichm\"{u}ller space.

\medskip

\begin{theorem}
\label{main limit laminations}
In the same setting as in Theorem \ref{+-side --side},
let $c_{m_i}$ and $c_{n_i}$ be shortest pants decompositions of $(S,m_i)$ and $(S,n_i)$ respectively.
Let $\nu^-, \nu^+$ be  the Hausdorff limits of $\{c_{m_i}\}$ and $\{c_{n_i}\}$ respectively after passing to subsequences.
Then the minimal components of $\nu^+$ that are not simple closed curves coincide with the ending laminations of upper  ends of $(\hyperbolic^3/\Gamma)_0$.
Moreover, every upper  parabolic curve is contained in  $\nu^+$ .
Similarly the minimal components of  $\nu^-$ that are not simple closed curves coincide with the ending laminations of   lower  ends of $(\hyperbolic/\Gamma)_0$, and every lower  parabolic curve is contained in  $\nu^-$.

Conversely every simple closed curve contained in either $\nu^-$ or $\nu^+$  that has isolated leaves spiralling around it is a parabolic curve.
Every such simple closed curve in $\nu^+$ that is not contained in $\nu^-$ is an upper parabolic curve whereas every such simple closed curve in $\nu^-$ that is not contained in $\nu^+$ is a lower parabolic curve.
\end{theorem}

\subsection{Divergence theorems}
We shall next state our theorems on divergence of quasi-Fuchsian groups, where we shall give sufficient conditions for sequences of quasi-Fuchsian groups to diverge.
\begin{theorem}
\label{main}
Let $\{(m_i, n_i)\}$ be a sequence in $\mathcal{T}(S) \times \mathcal{T}(\bar S)$ satisfying the following conditions.
\begin{enumerate}
\item
 $\{m_i\}$ converges to a projective lamination $[\mu^-] \in \PML(S)$ whereas $\{n_i\}$ converges to $[\mu^+] \in \PML(S)$ in the Thurston compactification of the Teichm\"{u}ller space.
 \item
 There are components $\mu_0^+$ of   $\mu^-$ and $\mu_0^-$ of $\mu^+$ which are not weighted simple closed curves and have the  minimal supporting surfaces sharing at least one boundary component up to isotopy.
  \end{enumerate}
 Then the sequence  $\{qf(m_i, n_i)\} \subset QF(S)$ diverges in $AH(S)$.
 \end{theorem}
 
Theorem \ref{main} follows rather easily from Theorem \ref{+-side --side}.
If we use Theorem \ref{main limit laminations} instead of Theorem \ref{+-side --side}, we get the following.

\begin{theorem}
\label{modified main}
Let $\{m_i\}$ and $\{n_i\}$ be sequences  in $\mathcal{T}(S)$ and $\mathcal{T}(\bar S)$ without convergent subsequences, and let $c_{m_i}$ and $c_{n_i}$ be shortest pants decomposition of the hyperbolic surfaces $(S,m_i)$ and $(S,n_i)$ respectively.
Suppose that $c_{m_i}$ and $c_{n_i}$ converge to geodesic laminations $\mu^-$ and $\mu^+$ in the Hausdorff topology respectively.
Suppose that there are minimal components $\mu^-_0$ of $\mu^-$ and $\mu_0^+$ of $\mu^+$ which are not simple closed curves and have minimal supporting surfaces sharing at least one boundary component up to isotopy.
Then $\{qf(m_i, n_i)\} \subset QF(S)$ diverges in $AH(S)$.
\end{theorem}

In the setting of these two theorems above, the case when $\mu_0^-$ and $\mu_0^+$ have the same support is most interesting.
In fact, if they do not, it is much easier to prove the theorems just by using the continuity of length function in hyperbolic manifolds and the fact that on the ending lamination of an end $e$ of the non-cuspidal part  is maximal on a frontier component of a relative compact core facing $e$. 
Also the assumption  that $\mu^-_0$ and $\mu^+_0$ are not simple closed curves is essential.
  In the case when $\mu^-$ and $\mu^+$ share  simple closed curves, the construction of Anderson-Canary \cite{AC} gives an example of convergent sequence.
  Still, we can show the following theorem. 
  \begin{theorem}
 \label{scc case}
 Let $\mu^-$ and $\mu^+$ be two measured laminations on $S$ such that the  components shared by  $|\mu^-|$ and $|\mu^+|$ are all simple closed curves or there are no components shared by $|\mu^-|$ and $|\mu^+|$.
Suppose that 
either
\begin{enumerate}
\item
there is a boundary component of the minimal supporting surface of a non-simple closed curve component of $\mu^+$ which is contained in $|\mu^-|$ up to isotopy, or
\item
there is a boundary  component of the minimal supporting surface of a non-simple closed curve component of $\mu^-$ which is contained in $|\mu^+|$ up to isotopy.
\end{enumerate}
Then for every $\{m_i \in \teich(S)\}$ converging to $[\mu^-]$ and $\{n_i \in \teich(S)\}$ converging to $[\mu^+]$ in the Thurston compactification of the Teichm\"{u}ller space, the sequence  $\{qf(m_i,n_i)\} \subset QF(S)$ diverges in $AH(S)$.
 \end{theorem}
 

In the case when  a simple closed curve component of $|\mu^+|$ which does not lie on the boundary of minimal supporting surface of non-simple closed curve component, up to isotopy, is shared by $|\mu^-|$ as the same kind of component, we need to take into accounts the weights on $c_1, \dots , c_r$, as was done in Ito \cite{Ito} in the case of once-punctured torus groups.
 
 \begin{theorem}
 \label{generalised Ito}
 Consider, as in Theorem \ref{scc case}, sequences $\{m_i\}$ and $\{n_i\}$ converging to $[\mu^-]$ and $[\mu^+]$ respectively, and suppose that their supports share only  simple closed curves $c_1, \dots , c_r$.
 Suppose that none of $c_1, \dots, c_r$ is isotopic into the boundary of the minimal supporting surface  of a component of $\mu^-$ or $\mu^+$.
 Then $\{qf(m_i, n_i)\}$ converges after taking a subsequence only if the following conditions are satisfied.
 \begin{enumerate}
 \item For each $c_j$ among $c_1, \dots, c_r$, neither $\len_{m_i}(c_j)$ nor $\len_{n_i}(c_j)$ goes to $0$.
 \item There are sequences of integers $\{p_i^1\}, \dots , \{p_i^r\}, \{q_i^1\}, \dots , \{q_i^r\}$
 such that the following hold after passing to a subsequence:
 \begin{enumerate}
 \item If $|\mu^-| \setminus \cup_{j=1}^r c_j $ is non-empty, then  $(\tau_{c_1}^{p_i^1} \circ \dots \circ \tau_{c_r}^{p_i^r})^*(m_i)$ converges to $[\mu^- \setminus \cup_{j=1}^r w_j c_j ]$ in the Thurston compactification, where $w_j$ is the transverse measure on $c_j$ which  $\mu^-$ defines and $\tau_{c_j}$ denotes the Dehn twist around $c_j$.
 Otherwise, $(\tau_{c_1}^{p_i^1} \circ \dots \circ \tau_{c_r}^{p_i^r})^*(m_i)$ either stays in a compact set of the Teichm\"{u}ller space or converges to a projective lamination $[\nu^-]$ which contains none of  $c_1, \dots, c_r$ as leaves.
 \item In the same way, if $|\mu^+| \setminus \cup_{j=1}^r c_j$ is non-empty, $(\tau_{c_1}^{q_i^1} \circ \dots \circ \tau_{c_r}^{q_i^r})^*(n_i)$ converges to $[\mu^+ \setminus \cup_{j=1}^r v_j c_j]$ in the Thurston compactification, where $v_j$ is the transverse measure on $c_j$ which $\mu^+$ defines.
Otherwise,  $(\tau_{c_1}^{q_i^1} \circ \dots \circ \tau_{c_r}^{q_i^r})^*(n_i)$ either stays in a compact set of the Teichm\"{u}ller space or converges in the Thurston compactification to a projective lamination  $[\nu^+]$  which contains none of $c_1, \dots , c_r$ as a leaf.
 \item There exist $a_j \in \integers\ (j=1, \dots , r)$ 
 with $a_j \neq 0,-1$ and $k_i^j \in \integers$ with $|k_i^j| \rightarrow \infty$
 such that $p_i^j=k_i^ja_j$ and $q_i^j=k_i^j(a_j+1)$ for every $j=1, \dots , r$ and large $i$.
 \end{enumerate}
 \end{enumerate}
 If the sequence really converges, then $c_j$ is an upper parabolic curve of the algebraic limit if $a_i>0$, and  a lower parabolic curve otherwise.
 
Conversely, let $a_j\in \integers \, (j=1, \dots , r)$ be any number, and $\mu^-,\mu^+$  measured laminations whose supports share only (possibly empty) simple closed curves $c_1, \dots , c_r$ and which satisfy the following conditions.
\begin{enumerate}[\indent\  (1*)]
\item
The laminations $\mu^-$ and $\mu^+$ do not have non-simple-closed-curve components  whose minimal supporting surfaces share a boundary component up to isotopy.
\item
In the case when both $\mu^-$ and $\mu^+$ are connected and the minimal supporting surfaces of $\mu^-$ and $\mu^+$ are the entire surface $S$, the supports $|\mu^-|$ and $|\mu^+|$ do not coincide.
\item No simple closed curve in  $|\mu^-|$  is isotopic into the boundary of the minimal supporting surface of a non-simple-closed-curve component of $\mu^+$.
In the same way, no simple closed curve in $|\mu^+|$ is isotopic into the boundary of the minimal supporting surface of a non-simple-closed-curve component of $\mu^-$.
\end{enumerate}
Then, there is a sequence of $\{(m_i,n_i)\}$ in $\teich(S) \times \teich(\bar S)$ with algebraically convergent $\{qf(m_i,n_i)\}$  such that $\{m_i\}$ converges to $[\mu^-]$ and $\{n_i\}$ converges to $[\mu^+]$  in the Thurston compactification, and the two conditions (1) and (2) above are satisfied.
 \end{theorem}
 
 The latter half of this theorem  shows that {\em as sufficient conditions for divergence expressed in term of the limits of the conformal structures in the Thurston compactification, Theorems \ref{main} and \ref{scc case} together with the main theorem of \cite{OhG} are best possible}.
 
%

 \subsection{Non-existence of exotic convergence}
The assumptions in  Theorem \ref{scc case}  is related to the fact that such a sequence as in the statement cannot converge exotically to a b-group.
 The condition that none of $c_1, \dots , c_r$ lies on the boundary of the supporting surface of a component of $\mu^-$ or $\mu^+$ is essential for the exotic convergence.
We can prove the following related results.
(Recall that a non-peripheral parabolic locus is said to be isolated if it does not touch a simply degenerate end.)
 
 \begin{theorem}
 \label{no exotic convergence}
 Let $G$ be a b-group without isolated parabolic locus.
 Then there is no sequence of quasi-Fuchsian groups exotically converging to $G$.
 \end{theorem}
 \subsection{Self-bumping}
As the results of McMullen \cite{Mc}, Bromberg \cite{Br} and Magid \cite{Mg} suggest, the singularities of $AH(S)$ which are found thus far are all related to the construction of Kerckhoff-Thurston \cite{KT} and Anderson-Canary \cite{AC}.
The following results show that convergence to geometrically infinite groups in $AH(S)$ without isolated parabolic loci is quite different from the situation for regular b-groups where $QF(S)$ bumps itself.
 \begin{theorem}
 \label{bumping theorem}
 Let $\Gamma$ be a  geometrically infinite group  with isomorphism $\psi : \pi_1(S) \rightarrow \Gamma$ in $AH(S)$.
 Suppose that $\Gamma$ does not have an isolated parabolic locus.
Let $\{(m_i, n_i)\}$ and $\{(m_i',n_i')\}$ be two sequences in $\teich(S) \times \teich(\bar S)$ such that both $\{qf(m_i, n_i)\}$ and $\{qf(m_i',n_i')\}$ converge to $(\Gamma, \psi)$.
Then for any neighbourhood $U$ of the quasi-conformal deformation space $QH(\Gamma, \psi)$ of $(\Gamma, \psi)$ in $AH(S)$, we can take  $i_0$ so that if $i > i_0$, then there is an arc $\alpha_i$ in $U \cap QF(S)$ connecting $qf(m_i,n_i)$ with $qf(m_i',n_i')$.
In the case when $\Gamma$ is a b-group whose lower conformal structure at infinity is $m_0$, we can take $\alpha_i$ satisfying further the following condition.
For any neighbourhood $V$ of $m_0$ in $\teich(S)$, we can take $i_0$ so that for any $i > i_0$, the arc $\alpha_i$ is also contained in $qf(V \times \mathcal T(\bar S))$.
\end{theorem}

We shall then get a corollary as follows.
\begin{mcorollary}
\label{no self-bumping}
 In the setting of Theorem \ref{bumping theorem}, suppose furthermore  that each component of $\Omega_\Gamma/\Gamma$ that is not homeomorphic to $S$ is a thrice-punctured sphere.
Then $QF(S)$ does not bump itself at $(\Gamma,\psi)$, and in particular $AH(S)$ is locally connected at $(\Gamma,\psi)$.
  \end{mcorollary}
 
 We can generalise this corollary by dropping the assumption that there are no isolated parabolic loci.
The same result has been obtained by  substantially different methods in Brock-Bromberg-Canary-Minsky \cite{BBCM}; see also Canary \cite{CaInd}.
Also, a related result has been obtained by  Anderson-Lecuire \cite{AL}.
 
\begin{mcorollary}
\label{no self-bumping with isolated}
Let $\Gamma$ be a group on the boundary of $QF(S)$ with isomorphism $\psi: \pi_1(S) \rightarrow \Gamma$.
This time we allow $\Gamma$ to have isolated parabolic loci.
\begin{enumerate}
\item
If every component of $\Omega_\Gamma/\Gamma$ is a thrice-punctured sphere, then $QF(S)$ does not bump itself at $(\Gamma, \psi)$.
\item
If $\Gamma$ is a b-group and every component of $\Omega_\Gamma/\Gamma$ corresponding to upper ends of $(\hyperbolic^3/\Gamma)_0$ is a thrice-punctured sphere, then the Bers slice containing $(\Gamma, \psi)$ on the boundary does not bump itself.
\end{enumerate}
\end{mcorollary}
 
%
%
%

\subsection{General Kleinian surface groups}
We can generalise Theorems \ref{main limit laminations} and  \ref{main} to sequences of Kleinian surface groups which may not be quasi-Fuchsian.

Let $G$ be a Kleinian surface group with marking $\phi: \pi_1(S) \rightarrow G$, and set $M=\hyperbolic^3/G$.
The marking $\phi$ determines a homeomorphism $h: S \times \reals \rightarrow M$.
Let $P_+$ be the upper parabolic locus on $S$.
We consider all the upper ends of the non-cuspidal part $M_0$.
For a geometrically finite end, we consider its minimal pants decomposition, and for a simply degenerate end, we consider its ending lamination.
Take the union of all these curves and laminations together with  core curves of $P_+$, and denote it by $e_+$.
In the same way, we define $e_-$ for the lower ends.
We call $e_+$ and $e_-$ the upper and the lower {\em generalised shortest pants decompositions} respectively.

We now state a generalisation of Theorem \ref{main limit laminations}
\begin{theorem}
\label{generalised limit laminations}
Let $\{(G_i, \phi_i)\}$ be a sequence of Kleinian surface groups which have upper and lower generalised  shortest pants decompositions  $e(i)_+$ and $e(i)_-$.
Suppose that $\{(G_i,\phi_i)\}$ converges algebraically to $(\Gamma, \psi)$.
Consider the Hausdorff limit $e(\infty)_+$ of $\{e(i)_+\}$ and $e(\infty)_-$ of $\{e(i)_-\}$ after passing to  subsequences.
Then every minimal component of $e(\infty)_+$ (resp. $e(\infty)_-$)  that is not a simple closed curve is the ending lamination of an upper end (resp. a lower end) of $(\hyperbolic^3/\Gamma)_0$.
Conversely any ending lamination of an upper end (resp. a lower end) of $(\hyperbolic^3/\Gamma)_0$ is a minimal component of $e(\infty)_+$ (resp. $e(\infty)_-$).
Moreover, every upper (resp. lower) parabolic curve is contained in $e(\infty)_+$ (resp. $e(\infty)_-$). 
\end{theorem}

Next we shall state a generalisation of Theorem \ref{main}.
\begin{theorem}
\label{generalised main}
Let $\{(G_i, \phi_i)\}$ be a sequence of Kleinian surface groups which have upper and lower generalised  shortest pants decompositions $e(i)_+$ and $e(i)_-$.
Let $e(\infty)_+$ and $e(\infty)_-$ be the Hausdorff limits of $\{e_+(i)\}$ and $\{e_-(i)\}$ respectively, after passing to a subsequence.
If there are minimal components $\lambda$ of $e(\infty)_-$ and $\mu$ of $e(\infty)_+$ which are not simple closed curves and whose minimal supporting surfaces share a boundary component up to isotopy, then $\{(G_i, \phi_i)\}$ diverges in $AH(S)$.
\end{theorem}

\subsection{Application}
We shall briefly explain here an application of the main results, in particular of Theorem 11 and Theorem 5.2, which appears in Ohshika \cite{OhF}.
For a point $m_0 \in \teich(S)$, the subspace of $QF(S)$ in the form of $qf(\{m_0\} \times \teich(\bar S))$ is called the Bers slice based on $m_0$, and is denoted by $B_{m_0}$.
Bers proved in \cite{Be} that the closure of $B_{m_0}$ in $AH(S)$ is compact for any $m_0 \in \teich(S)$.
The closure is called the Bers compactification of the Teichm\"{u}ller space $\teich(S)$ (identified with $\teich(\bar S)$) based on $m_0$.
We denote its boundary by $\partial^B_{m_0} \teich(S)$.
Kerckhoff and Thurston proved in \cite{KT} that there are two points $m_0, m_1 \in \teich(S)$ such that the natural homeomorphism between $B_{m_0}$ and $B_{m_1}$, which is obtained by identifying them with $\teich(S)$, cannot extend continuously to a homeomorphism between their boundaries $\partial^B_{m_0} \teich(S)$ and $\partial_{m_1}^B \teich(S)$.
This implies that the action of the mapping class group $MCG(S)$ on $\teich(S)$ does not extend continuously to the Bers compactification (based on any point).

In \cite{OhF}, we considered a quotient space of the Bers boundary by collapsing each quasi-conformal deformation space lying there into a point, and showed that the mapping class group acts on this quotient space.
(According to McMullen, Thurston already considered this space and conjectured this result.)
We denote this quotient space obtained from $\partial_{m_0}^B \teich(S)$ by $\partial_{m_0}^{RB} \teich(S)$ and call it the reduced Bers boundary based on $m_0$.
Applying Theorem 11 in the present paper, we also showed that conversely every automorphism of $\partial_{m_0}^{RB}\teich(S)$ is induced from an extended mapping class.
The same kind of result was obtained for the unmeasured lamination space by Papadopoulos \cite{Papa} and Ohshika \cite{OhP},  and for the geodesic lamination space with the Thurston topology by Charitos-Papadoperakis-Papadopoulos \cite{CPP}.

\section{Models of geometric limits}
\subsection{Brick decompositions of geometric limits}
\label{brick decomposition}
In this section, we shall review the results in Ohshika-Soma \cite{OhSo} and show some facts derived from them, which are essential in our discussion.
We shall give sketches of proofs for all necessary results  so that the reader can understand them without consulting \cite{OhSo}.

Throughout this section, we assume that we have a sequence $\{(G_i, \phi_i)\}$ in $AH(S)$ converging to $(\Gamma, \psi)$, and that $\{G_i\}$ converges geometrically to $G_\infty$, which contains $\Gamma$ as a subgroup.
We do not assume that $G_i$ is quasi-Fuchsian, to make our argument work also for the proofs of Theorems \ref{generalised limit laminations} and \ref{generalised main}.
Recall that  $M_\infty=\hyperbolic^3/G_\infty$ is a Gromov-Hausdorff limit of $M_i=\hyperbolic^3/G_i$  with basepoints $y_i$ which are projections of some  point fixed in $\hyperbolic^3$.
We denote $\hyperbolic^3/\Gamma$ by $M'$.
Let $p: M' \rightarrow M_\infty$ denote the covering associated to the inclusion of $\Gamma$ to $G_\infty$.
Let $\rho_i: B_{r_i}(M_i, y_i) \rightarrow B_{K_ir_i}(M_\infty, y_\infty)$ denote an approximate isometry corresponding to the pointed Gromov convergence of $(M_i,y_i)$ to $(M_\infty, y_\infty)$.

In \cite{OhSo}, we introduced the notion of brick manifolds.
A brick manifold is a $3$-manifold constructed from \lq bricks' which are defined as follows.
We note that a brick is an entity different from a block introduced by Minsky which we explained in Preliminaries.
Still, as we shall see, they are closely related, and actually, in our settings every brick is decomposed into blocks.

\begin{definition}
\label{brick manifold}
A {\em brick} is a product interval bundle of the form $\Sigma \times J$, where $\Sigma$ is an incompressible subsurface of $S$ with $\chi(\Sigma) <0$ and $J$ is a closed or half-open interval in $[0,1]$.
We assume that $\Sigma$ is a closed subset of $S$, \ie $\Fr \Sigma$ is contained in $\Sigma$.
(Recall that we say that a (not necessarily proper) subsurface is incompressible if none of its boundary components is null-homotopic and at least one boundary component is non-peripheral, \ie if it is either essential or $S$ itself.)
Note that it may be possible that two boundary components of $\Sigma$ are parallel in $S$.
For a brick $B=\Sigma \times J$, its {\em lower front}, denoted by $\partial_-B$, is defined to be  $\Sigma \times \inf J$, and its {\em upper front}, denoted by $\partial_+B$, is defined to be  $\Sigma \times \sup J$.
When $J$ is an half-open interval, one of them may not really exist, but corresponds to an end.
In this case, it is called an {\em ideal front}.
A brick naturally admits two foliations: one is a codimension-$1$ horizontal foliation whose leaves are horizontal surfaces $\Sigma \times \{t\}$, and the other is a codimension-$2$  vertical foliation whose leaves are vertical lines $\{p\} \times J$.
We define $\xi(B)$ to be $\xi(\Sigma)$.

A brick manifold is a manifold consisting of countably many bricks, whose boundary consists of tori and open annuli.
Two bricks can intersect only at their (non-ideal) fronts in such a way that an essential  subsurface, which is possibly disconnected but none of whose components is an annulus, in the upper front of one brick  is pasted to an essential  subsurface in the lower front of the other brick.
In the case when the manifold is homeomorphic to $S \times (0,1)$, we allow two bricks intersect at their entire non-ideal fronts.
We also assume that an infinite sequence of bricks cannot accumulate inside the manifold, \ie an infinite sequence of bricks must correspond to an end of the manifold after passing to a subsequence.
\end{definition}

Since $G_i$ is a Kleinian surface group, $\hyperbolic^3/G_i$ has a bi-Lipschitz model which was constructed by Minsky \cite{Mi} and was proved to be bi-Lipschitz by Brock-Canary-Minsky \cite{BCM}.
We ignore cusp neighbourhoods in the model manifolds of Minsky to make them models for the non-cuspidal parts.
Let $\mathbf M_i$ be a  model manifold for $(M_i)_0=(\hyperbolic^3/G_i)_0$ in the sense of Minsky with a bi-Lipschitz model map $f_i : \mathbf M_i \rightarrow (M_i)_0$.
Minsky's construction is based on complete hierarchies of tight geodesics which are determined by the end invariants of $M_i$, as we explained in Preliminaries.
The model manifold has decomposition into blocks and Margulis tubes, which corresponds to a resolution of a complete hierarchy.
When we talk about model manifolds $\mathbf M_i$, {\em we always assume the existence of complete hierarchies $h_i$ beforehand}, and that the manifolds are decomposed into blocks and Margulis tubes using resolutions.
The metric of model manifolds are defined as the union of  metrics on internal blocks and metrics determined by conformal structures at infinity on boundary blocks.
We should note that as was shown in \cite{Mi}, the decomposition of $\mathbf M_i$ into blocks and the metric on $\mathbf M_i$ depend only on $h_i$ and end invariants, and are independent of choices of resolutions.

We shall now see that geometric limits of algebraically convergent quasi-Fuchsian groups have also model manifolds.
The following is one of the main theorems of \cite{OhSo} which is fitted into our present situation.

\begin{othertheorem}[Ohshika-Soma \cite{OhSo}]
\label{Ohshika-Soma}
Let $\{(G_i, \phi_i)\}$ be a sequence in $AH(S)$ converging to $(\Gamma, \psi)$ with $M'=\hyperbolic^3/\Gamma$, and $M_\infty$ a geometric limit of $M_i=\hyperbolic^3/G_i$ with basepoint at $y_i$.
Then, there are a  model manifold $\mathbf{M}$ of $(M_\infty)_0$, which has a structure of brick manifold and is a geometric limit of $\mathbf M_i$ as a metric space, and a model map $f: \mathbf{M} \rightarrow (M_\infty)_0$ which is a $K$-bi-Lipschitz homeomorphism for a constant $K$ depending only on $\chi(S)$.
The model manifold $\mathbf M$ has the following properties.
\begin{enumerate}
\setcounter{enumi}{-1}
\item Each brick of $\mathbf M$ is decomposed into blocks and Margulis tubes as in Minsky's model manifolds, each of which is a limit of blocks and Margulis tubes in $\mathbf M_i$.
\item $\mathbf M$ can be embedded in $S \times [0,1]$ (with its image in $S\times (0,1)$) in such a way that the vertical and horizontal foliations of the bricks are mapped into horizontal surfaces and vertical lines of $S\times [0,1]$ respectively.
\item There is no essential properly embedded annulus in $\mathbf M$.
\item An end contained in a brick is defined to be either geometrically finite or simply degenerate. 
The model map takes  geometrically finite ends to geometrically finite ends of $(M_\infty)_0$, and simply degenerate ends  to  simply degenerate ends of $(M_\infty)_0$.
\item Every geometrically finite end of $\mathbf M$ corresponds to an incompressible subsurface of either $S \times \{0\}$ or $S \times \{1\}$.
\item An end not contained in a brick is neither geometrically finite nor simply degenerate.
For such an end, there is no half-open incompressible annulus tending to the end which is not properly homotopic into a boundary component.
We call such an end {\em wild}.
\item There are only countably many ends.
\item There is a $\pi_1$-injective map $g: S \rightarrow \mathbf M$ which is $\pi_1$-injective also as a map to $S \times [0,1]$, such that $(f \circ g)_\# $ coincides with $\iota \circ \psi$, where $\iota$ is the monomorphism from $\Gamma=\pi_1(M')$ to $\pi_1(M_\infty)$ induced by the inclusion of the algebraic limit $\Gamma$ into the geometric limit $\pi_1(M_\infty)$.
\item For any sequence of points $\{x_i \in \mathbf M\}$ tending to  an end of $\mathbf M$, its image in $S \times [0,1]$ converges, after passing to a subsequence, to a point in $S \times [0,1] \setminus \mathbf M$.
There are no  two sequences  in $\mathbf M$ tending to distinct ends of $\mathbf M$ whose images in $S \times [0,1]$ converge to the same point.
\end{enumerate}
\end{othertheorem}

Before starting the proof, we shall illustrate how a wild end as was described in (5) looks.
Suppose that an end corresponding to $\Sigma \times \{t_0\}$ with an incompressible subsurface $\Sigma$ of $S$ is wild.
Then there is a sequence of either torus cusps or simply degenerate ends (or both) whose images in $S \times [0,1]$ accumulate to $\Sigma \times \{t_0\}$.
In the case when torus cusps accumulate to $\Sigma \times \{t_0\}$, the vertical projections of its core curves to $S$ converge to an arational geodesic lamination in the Hausdorff topology.
In the case when simply degenerate ends $e_i$ accumulate to $\Sigma \times \{t_0\}$, the end $e_i$ corresponds to $\Sigma_i \times \{t_i\}$ with $t_i \rightarrow t_0$ such that $\Sigma_i$ is a subsurface of $\Sigma$ for large $i$, and each boundary component  of $\Sigma_i$ converges to the same arational lamination on $\Sigma$ in the Hausdorff topology.
The conditions that the Hausdorff limits are arational prevents the existence of an essential open annulus as described in the part (5) above.
Figure \ref{fig:wild} illustrates the case when simply degenerate ends accumulate to a wild end from below.

\begin{figure}
\scalebox{0.5}{\includegraphics{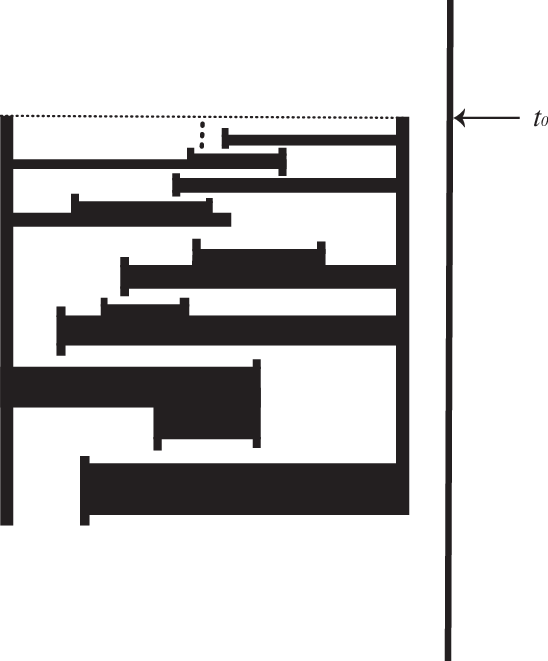}}
\caption{$\Sigma \times \{t_0\}$ is a wild end to which simply degenerate ends accumulate from below.
The vertical line at the right end denotes the coordinate $[0,1]$. Two vertically long rectangles at the left and right sides correspond to either frontier components or punctures of $\Sigma$. 
Small rectangles attached to black regions are sent into cusp neighbourhoods.
Each long horizontal side of black regions corresponds to a simply degenerate end.}
\label{fig:wild}
\end{figure}

\subsection{Proof of Theorem \ref{Ohshika-Soma}}
Although this theorem was  already proved in \cite{OhSo}, we shall give an abridged version of its proof here, using sometimes an explanation  a bit different from the one given in \cite{OhSo}.

This model manifold $\mathbf M$ is obtained as the non-cuspidal part of a geometric limit  of Minsky's model manifolds for $\hyperbolic^3/G_i$ with suitable basepoints chosen.
%
Let $\mathbf M_i$ be Minsky's model manifold for $(M_i)_0$ with a bi-Lipschitz model map $f_i$ as explained above.
We put a basepoint $x_i$ in $\mathbf M_i$ which is mapped to $y_i$ by $f_i$.
Recall that Minsky's model manifold is composed of blocks and tubes.
As was explained in \S \ref{hierarchy}, there are two kinds of blocks: internal blocks and boundary blocks.
An internal block is topologically homeomorphic to $\Sigma \times [0,1]$, where $\Sigma$ is either a sphere with four holes or a torus with one hole.
Geometrically it has two ditches, and has  form of $B=\Sigma \times [0,1] \setminus (N_\alpha \times [0,1/4] \cup N_\beta \times [3/4, 1])$, where $N_\alpha, N_\beta$ are annular neighbourhoods of simple closed curves $\alpha, \beta$ on $\Sigma$ with $i(\alpha, \beta)=2$ if $\Sigma$ is a sphere with four holes and $i(\alpha, \beta)=1$ if $\Sigma$ is a torus with one hole.
We should also recall that the isometry type of $B$ is unique once we fix $\Sigma$ to be a sphere with four holes or a torus with one hole.
A boundary block corresponds to a geometrically finite end of $(\hyperbolic^3/G_i)_0$ as was explained in \S \ref{model}.
We also have an embedding of $\mathbf M_i$ into $S \times [0,1]$ such that the product structure of each block coincides with that of $S \times [0,1]$, in particular each horizontal surface in a block is mapped into a horizontal surface of the form $S \times \{t\}$.
Recall each $\Sigma$ is an incompressible subsurface of $S$.
For each $\Sigma$, we fix an embedding of $\Sigma$ into $S$ by fixing some hyperbolic metric on $S$ and making each boundary  component of $\Sigma$  geodesic if no two boundary components of $\Sigma$ are parallel on $S$, and a simple closed curve with small constant geodesic curvature otherwise.
By this way, the vertical projection of $B$ into $S$ is determined.
We note that this hyperbolic metric on $S$ has nothing to do with the metric which we endow on model manifolds.

After gluing blocks in accordance with the information given by a hierarchy of tight geodesics associated to $G_i$, we get a 3-manifold whose boundary consists of two surfaces homeomorphic to $S$ and countably many tori.
We denote by $\mathbf M_i[0]$ this manifold before filling Margulis tubes, which is a subset of $\mathbf M_i$.
For each torus boundary, we fill in a Margulis tube, which is a solid torus whose isometry type is determined by the flat structure and the marking on the boundary torus which is determined by the hierarchy.

Now, we shall see that a geometric limit of these model manifolds serves as a model manifold of $(\hyperbolic^3/G_\infty)_0$ once cusp neighbourhoods are removed.
\begin{lemma}
\label{model limit}
Put a basepoint $x_i$ in $\mathbf M_i[0]$.
Then $\{(\mathbf M_i, x_i)\}$ converges geometrically to a $3$-manifold $\hat{\mathbf M}$ consisting of internal and boundary blocks, Margulis tubes, and cusp neighbourhoods, after passing to a subsequence.
By removing the cusp neighbourhoods, we get a brick manifold $\mathbf M$.
The bi-Lipschitz model map $f_i :\mathbf M_i \rightarrow (\hyperbolic^3/G_i)_0$ also converges geometrically to a bi-Lipschitz map whose restriction to $\mathbf M$ is a bi-Lipschitz model map to $(\hyperbolic^3/G_\infty)_0$.
\end{lemma}
\begin{proof}
For internal blocks, their isometry types depend only on whether $\Sigma$ is a sphere with four holes or a torus with one hole.
Therefore, their geometric limits are also internal blocks.
We now turn to boundary blocks.
Since topologically a boundary block is homeomorphic to $\Sigma \times [0,\infty)$ or $\Sigma \times (-\infty, 0]$ for an incompressible subsurface $\Sigma$ of $S$, we can assume, passing to a subsequence, $\Sigma$ does not depend on $i$.
A boundary block as we explained in \S\ref{model} has a metric uniquely determined by the conformal structure at infinity of the corresponding geometrically finite end.
In the case when the conformal structures at infinity are bounded (in the moduli space) as $i$ varies, by taking a subsequence, we can assume that the conformal structures converge in the moduli space, and  it is evident that the geometric limit of boundary blocks is again a boundary block corresponding to the limit conformal structure at infinity. 

Now we consider the situation where the hyperbolic metric $m_0^i$ on the geometrically finite end corresponding to $\Sigma \times \{\infty\}$ varies with $i$ and goes to an end in the moduli space.
Let $c_1, \dots , c_p$ be a shortest pants decomposition of $(\Sigma, m_0^i)$.
Since $\{m_0^i\}$ is unbounded in the moduli space,  the lengths of some of $c_1, \dots , c_p$ with respect $m_0^i$ go to $0$ as $i \rightarrow \infty$.
By renumbering $c_1, \dots , c_p$, we can assume $c_1, \dots , c_q$ to be the curves whose lengths go to $0$.
If the length of $c_k$ with respect to $m_0^i$ goes to $0$, then the modulus of the annulus $A(c_k)$, whose boundaries we assumed to have length $\epsilon_2$ in \cref{boundary block}, goes to $\infty$, and hence in a geometric limit, the surface $\Sigma$ is cut along $c_k$.
If we put a basepoint on $\Sigma \times \{-1\} \setminus (\cup_{k=1}^p A(c_k) \times \{-1\})$, then the boundary blocks have a geometric limit after passing to a subsequence, and the limit is identified with $(\Sigma \setminus \cup_{k=1}^q c_k) \times [-1, \infty) \setminus (\cup_{k=q+1}^p A(c_k)  \times [-1, 0)$.
By Lemma 3.4 in Minsky \cite{Mi}, there is a uniform bi-Lipschitz map (\ie with bi-Lipschitz constants independent of $i$) from this boundary block to a component of the complement of the convex core of any hyperbolic structure on $\Sigma \times \reals$ whose conformal structure at infinity is $m_0^i$, with tubular neighbourhoods of closed geodesics corresponding to the pants decomposition $c_1, \dots , c_p$ removed.
This implies that the geometric limit of  boundary blocks is again a boundary block after its intersection with $\epsilon_0$-cusp neighbourhoods are removed.

Next we turn to Margulis tubes.
Let $V_i$ be a Margulis tube in the model manifold $\mathbf M_i$.
We can parametrise $\mathbf M_i$ by $S \times [0,1]$ in such a way that $V_i$ corresponds to $A_i \times [s,t]$ for some essential annulus $A_i$ on $S$.
Each Margulis tube $V_i$ has a coefficient $\omega_{\mathbf M_i}(V_i)$ which is defined as follows.
The boundary of the tube $\partial V_i$ has a flat metric induced from the metric of $\mathbf M_i[0]$ determined by the blocks.
We can give a marking (longitude-meridian system) $(\alpha_i, \beta_i)$ to $\partial V_i$ by defining $\alpha_i$ to be a horizontal curve in a block and $\beta_i$ to be  $\partial(a_i \times [s,t])$ for some essential arc $a_i$ connecting the two boundary components of $A_i$.
The flat metric and the marking determine a point in the Teichm\"{u}ller space of a torus identified with $\{z \in \complexes \mid \Im z >0\}$, which we define to be $\omega_{\mathbf M_i}(V_i)$.
More concretely, we define the point $z_i=\omega_{\mathbf M_i}(V_i)$ in the upper half plane to be the point corresponding to $\partial V_i$ if $\partial V_i$ with marking $(\alpha_i, \beta_i)$ is conformal to $\complexes/(\integers+z_i\integers)$ taking $\alpha_i$ to the curve coming from the first $\integers$ and $\beta_i$ to  the curve coming from $z_i \integers$.
We note that this parametrisation is continuous with respect to the geometric topology.

By our construction of metrics on blocks  $\Im \omega_{\mathbf M_i}(V_i)$ is bounded away from $0$.
If $|\omega_{\mathbf M_i}(V_i)|$ is bounded, then passing to a subsequence, we can assume that $\{\omega_{\mathbf M_i}(V_i)\}$ converges to some number $w$, hence that $\{V_i\}$ converges to some Margulis tube whose coefficient is $w$.
If $\Re \omega_{\mathbf M_i}(V_i) \rightarrow \infty$ whereas $\Im \omega_{\mathbf M_i}(V_i)$ is bounded, then the conformal structure on a torus corresponding to $\omega_{\mathbf M_i}$ is bounded in the moduli space as $i \rightarrow \infty$.
Therefore, $\partial V_i$ converges some flat torus passing to a subsequence.
On the other hand, since $\Re \omega_{\mathbf M_i}(V_i)$ diverges, the length of the meridian for $V_i$ diverges.
This means that passing to a subsequence, $\{V_i\}$ converges to a torus cusp neighbourhood.
If $\Im \omega_{\mathbf M_i}(V_i) \rightarrow \infty$, passing to a subsequence, the conformal structure on a torus corresponding to $\omega_{\mathbf M_i}(V_i)$ diverges in the moduli space.
Since the length of longitude of $\partial V_i$ is defined to be constant, this means that $\partial V_i$ converges to an open annulus geometrically. 
This is possible only when $\{V_i\}$ converges geometrically to one or two rank-1 cusp neighbourhoods.
Since we are considering the non-cuspidal part of the geometric limit, these cusp neighbourhoods are removed and what remain are torus or annulus boundaries.
(What should be removed is $\epsilon_0$-cusp neighbourhood, whereas the limit is $\epsilon_2$-cusp neighbourhood.
This does not matter since they are constant only depending on $S$ and we can map the $\epsilon_2$-cusp neighbourhood to the $\epsilon_0$-cusp neighbourhood by a uniform bi-Lipschitz map.)

Thus we have seen that the non-cuspidal part of the geometric limit  of $\mathbf M_i$ with basepoint at $x_i$, which is defined to be  $\mathbf M$, is a union of (internal and boundary) blocks and Margulis tubes.
We denote by $\mathbf M[0]$ the part of $\mathbf M$ which is the geometric limit of $\mathbf M_i[0]$ inside $\mathbf M$.
Now, we shall see how a structure of brick decomposition of the geometric limit appears.
Since the blocks of $\mathbf M_i[0]$ are embedded and they are pasted together along horizontal surfaces, we see that the notion of horizontal surfaces (and of horizontal foliation) makes sense in $\mathbf M_i[0]$.
Taking their geometric limit, we see that $\mathbf M[0]$ also have well-defined horizontal surfaces.
Each Margulis tube of $\mathbf M_i$ is foliated by horizontal annuli.
Since Margulis tubes  filled into $\mathbf M[0]$ is a geometric limit of Margulis tubes in $\mathbf M_i$, they are also foliated by horizontal annuli.
Therefore, horizontal surfaces intersecting a Margulis tube extend to annuli in the tube to form horizontal surfaces corresponding to subsurfaces of $S$.
Taking geometric limits of these surfaces, we see that also in the geometric limit, horizontal surfaces intersecting Margulis tubes can be extended to annuli in the Margulis tubes to form subsurfaces of $S$, which we regard as horizontal surfaces.
We consider a maximal union of vertically parallel horizontal surfaces, and define its closure to be a brick of $\mathbf M$.
It is easy to see that two bricks can intersect each other only along (a part of) horizontal boundaries, and that the intersection consists of essential subsurfaces by the fact that it is decomposed into horizontal surfaces lying in blocks and horizontal annuli in Margulis tubes.
We can also check that there is no annulus component among the intersection, since every boundary component corresponds to a cusp limit of Margulis tubes, and no two such components are homotopic.
Thus we can see that the geometric limit admits a structure of a brick manifold.
As the restriction of a geometric limit of the model maps $f_i: \mathbf M_i \rightarrow (M_i)_0$ to the non-cuspidal part, we get a model map $f_\infty : \mathbf M \rightarrow (M_\infty)_0$.
\end{proof}

Thus we have obtained $\mathbf M$ and $f_\infty$ and by construction the condition (0) holds.
We now check that the conditions (1)-(6) of the statement hold.
Since $f_i$ is a homeomorphism and takes Margulis tubes whose cores are sufficiently short to the same kind of Margulis tubes in $\hyperbolic^3/G_i$, there is a one-to-one correspondence between the Margulis tubes with short cores of $\mathbf M_i$ and those of $(\hyperbolic^3/G_i)_0$.
Since there are no two distinct Margulis tubes in $(\hyperbolic^3/G_i)_0$ whose cores are homotopic, the same holds for Margulis tubes with short cores in $\mathbf M_i$.
The condition (2) is derived from this property since two homotopic longitudes on $\partial \mathbf M$ give rise to two Margulis tubes with homotopic very short core curves for large $i$.

Next we turn to the condition (3).
Let $B$ be a brick in $\mathbf M$, and suppose that $B$ has an end $e$.
If the end $e$ is contained in a boundary block of $\mathbf M$, we define $e$ to be geometrically finite.
Since $f_i$ takes the of every boundary block of $\mathbf M_i$ to a geometrically finite end of $M_i$ and a boundary block of $\mathbf M$ is a geometric limit of boundary blocks of $\mathbf M_i$, as was seen in the proof of Lemma \ref{model limit} above, which must also correspond to a geometrically finite end of $(M_\infty)_0$, the limit $f$ also takes $e$ to a geometrically finite end of $(M_\infty)_0$.

If $e$ in $B$ is not contained in a boundary block, we define $e$ to be simply degenerate.
Such an end appears only when there are infinitely many blocks constituting $B \cong \Sigma \times J$. 
This implies that there are infinitely many Margulis tubes tending to the end $e$.
Since each core curve of Margulis tube is homotopic to a simple closed curve on $\Sigma$, and $f$ takes such a core curve to a closed geodesic in $M_\infty$, we see that the end of $(M_\infty)_0$ corresponding to $e$ is simply degenerate.

Now we check the condition (5).
Let $e$ be an end of $\mathbf M$ which is not contained in a brick.
We need to show the following.
\begin{lemma}
\label{no open annulus}
There is no essential half-open annulus whose end tends to the end $e$.
\end{lemma}
\begin{proof}
We prove this by contradiction.
Suppose that $A$ is an essential half-open annulus tending to $e$.
Then $A$ intersects infinitely many bricks $B_n$ tending to $e$.
Since the $B_n$ are distinct, by using the condition (2), we can choose a vertical boundary of $B_n$ with core curve $c_n$ so that the $f(c_n)$ represent all distinct homotopy classes in $M_\infty$.
Moreover the distance between the basepoint $y_\infty$ and the closed geodesic homotopic to $f(c_n)$ goes to $\infty$ as $n \rightarrow \infty$.
Let $a$ be a core curve of $A$.
Then $a$ and $c_n$ can be realised as disjoint curves on a horizontal surface on which $c_n$ lies.
Pulling back this situation to $\mathbf M_i$ and $M_i$, we see that there is a pleated surface $k_n^i : S\rightarrow M_n$ homotopic to $f_i|S\times \{t\}$ which realises both $a$ and $c_n$ at the same time.
This is impossible since the distance modulo the thin part between $k_n^i(a)$ and the one representing $k_n^i(c_n)$ goes to $\infty$, which contradicts the compactness of pleated surfaces.
(When $f(a)$ represents a parabolic class, we need some more care, but essentially the same kind of argument works, for $f(a)$ is homotoped into a cusp which does not touch the end $e$ in this case.)
\end{proof}

What now remain are (1), (4), (6) and (7).

Before starting to prove them, we shall describe general settings.
First, we consider a monotone increasing exhaustion of $\mathbf M$ by connected submanifolds consisting of finitely many bricks; that is, a sequence of submanifolds $\mathbf N_1 \subset \mathbf N_2 \subset \dots$ such that $\mathbf M=\cup_n \mathbf N_n$, each $\mathbf N_n$ is a connected union of finitely many bricks in $\mathbf M$, and every brick of $\mathbf N_n$ is attached to $\mathbf N_{n-1}$ at either its upper front or lower front or both.
(See Definition \ref{brick manifold} for definitions of upper and lower fronts.)
Each compact brick of $\mathbf N_n$ is contained in the range of the approximate isometry from $\mathbf M_i$ to $\mathbf M$ for large $i$, and hence we can embed it preserving the vertical and horizontal foliations.
In the same way, for each non-compact brick $B$ of $\mathbf N_n$, its real front, which we denote by $\Sigma$, is contained in the range of the approximate isometry for large $i$.
Since $B$ is homeomorphic to either $\Sigma \times [0,1)$ or $\Sigma \times (0,1]$, by embedding its real front using the embedding of $\mathbf M_i$ and the approximate isometry, we can embed $B$ into $S\times [0,1]$, again preserving the vertical and horizontal foliations.
We can also adjust the image of its ideal front so that the closures of the images of two bricks do not intersect, \ie so that even two ideal fronts do not intersect.
Since $\mathbf N_n$ has only finitely many bricks, by taking sufficiently large $i$, we can embed $\mathbf N_n$ into $S \times [0,1]$ using these embeddings of its bricks induced by the  embedding of  $\mathbf M_i$ into $S \times [0,1]$, keeping the horizontal and vertical foliations.
We denote this embedding of $\mathbf N_n$ by $\eta_n$.

Passing to a subsequence and isotoping the $\eta_n$ vertically, we can assume that for each brick $B$ of $\mathbf M$, the horizontal level of $\eta_n(B)$ is independent of $n$ if it is defined, without changing the condition that the upper ends of the upper boundary blocks   lie on $S \times \{1\}$ and the lower ends of the lower boundary blocks lie on $S \times \{0\}$.
Here we call a boundary block of the form $\Sigma \times [s,t)$ a upper and that of the form $\Sigma \times (s,t]$ lower, and its end is called an upper end and a lower end respectively.
Now for each brick $B$ of $\mathbf M$, let $\sup B$ and $\inf B$ be the levels of $\eta_n(\partial_+B)$ and $\eta_n(\partial_- B)$ with respect to the second factor of $S \times [0,1]$ (which we call {\em horizontal levels} from now on) under an embedding $\eta_n$ for which $\eta_n(B)$ is defined.
As noted above, these are independent of $n$.
We consider the set $C=\{\sup B,\inf B \mid B \text{ are bricks in } \mathbf M\} \subset [0,1]$ and its closure $\bar C$.
We can further perturb $\eta_n$ (for all $n$ simultaneously) so that two distinct bricks share $\inf$ or $\sup$ only at $0$ and $1$.

%
%

For each $t \in [0,1]$, we consider the surface $\Int \eta_n(\mathbf N_n) \cap S \times \{t\}$, which we regard as a subsurface of $S$ identifying $S \times \{t\}$ with $S$, and we denote it by $D_{\eta_n}(t)$.
Since bricks are embedded with  horizontal foliations preserved, $D_{\eta_n}(t)$ is an incompressible subsurface of $S$ and there are no components of $D_{\eta_n}(t)$ which are discs or annuli.
For convenience, we fix a hyperbolic structure on $S$, and we can assume that each  boundary component of $D_{\eta_n}(t)$ either  is  either is geodesic or  has a small geodesic, so that if two of such surfaces are homotopic, then they are vertically parallel to each other.
Since there are neither annuli nor discs, for any $t \in [0,1]$, the Euler characteristic of $D_{\eta_n}(t)$ is monotone non-increasing with respect to $n$.
Therefore, the homeomorphism type of $D_{\eta_n}(t)$ does not change for large $n$.
Furthermore, since there are only finitely many ways to embed a surface essentially into $S$ up to automorphisms of $S$, we can assume that there is $n_0$ such that the homeomorphism type of $(S, D_{\eta_n}(t))$ as a pair does not change for $n \geq n_0$.
We say that $D_{\eta_n}(t)$ is {\em stable} in this situation.
For stable $D_{\eta_n}(t)$, we denote its Euler characteristic, which is independent of $n$, by $\chi_\mathrm{stab}(D(t))$.

Now we start to prove (6).
Since an end either lies in a brick or is a limit of bricks, it always corresponds to a number in $\bar C$.
(Although we have not assumed that $\mathbf M$ is embedded in $S \times [0,1]$ at this stage, the horizontal level of an end makes sense because the horizontal levels of each brick are well defined.)
Moreover, since there is a uniform bound for the number of components of $D_{\eta_n}(t)$, and one can approach $S \times \{t\}$ only from two sides, from above and from below, there is a uniform bound for the number of ends corresponding to $t \in \bar C$.
Therefore, to show (6), what we need to show is the following.
\begin{lemma}
\label{countable}
The set $\bar C$ is countable.
\end{lemma}

Let $t_\infty$ be an accumulation point of $C$.
Then we can easily see the following.
We define two subsurfaces $\Sigma^-_{\eta_n}(t_\infty)$ and $\Sigma^+_{\eta_n}(t_\infty)$ of $S$ to be the sets of limit points  of $S \setminus D_{\eta_n}(x-\epsilon)$ and $S \setminus D_{\eta_n}(x+\epsilon)$ as $\epsilon \searrow 0$ respectively.
We denote the complements of $\Sigma^-_{\eta_n}(t_\infty)$ and $\Sigma^+_{\eta_n}(t_\infty)$ by $D^+_{\eta_n}(t_\infty)$ and $D^-_{\eta_n}(t_\infty)$ respectively.
\begin{claim}
\label{region}
There is $\delta >0$ such that for any $t \in [t_\infty-\delta, t_\infty) \cup (t_\infty, t_\infty+\delta]$ we have $\chi_\mathrm{stab}(D(t)) \leq \chi_\mathrm{stab}(D(t_\infty))$ and $D_{\eta_n}(t_\infty)$ is vertically isotopic into $D_{\eta_n}(t)$ in $S \times [0,1]$ for sufficiently large $n$.
Moreover, for any $t \in [t_\infty -\delta, t_\infty)$, the surface $D_{\eta_n}^-(t_\infty)$ is vertically homotopic into either $D_{\eta_n}(t)$, and for any $t \in (t_\infty, t_\infty+\delta]$, the surface $D_{\eta_n}^+(t_\infty)$  is vertically homotopic into $D_{\eta_n}(t)$.
\end{claim}
\begin{proof}
By the same reason as the argument just before the lemma,  there are only finitely many bricks of $\mathbf M$ the closures of whose images under $\eta_n$ intersect the level surface $S \times\{t_\infty\}$, which we name $\{B^{t_\infty}_j\}$.
(Again this property does not depend on $n$, and that both $\inf \eta_n(B^{t_\infty}_j)$ and $\sup \eta_n(B^{t_\infty}_j)$ are independent of $n$.)
We set  $t^+=\min_j\{ \sup B^{t_\infty}_j \mid  \sup B^{t_\infty}_j  \geq t_\infty\}$ and $t^-=\max_j\{ \inf B^{t_\infty}_j\mid  \inf B^{t_\infty}_j \leq t_\infty\}$.
Then  $\delta$ which is defined to be $\min\{t_+-t_\infty, t_\infty-t_-\}$ has the desired property.
\end{proof}
%
\begin{proof}[Proof of Lemma \ref{countable}]
By Claim \ref{region}, for any $t \in [t_\infty-\delta, t_\infty)]$, the surface $D_{\eta_n}^-(t_\infty)$ is vertically homotopic into $D_{\eta_n}(t)$.
Moreover, if $t_\infty$ is an accumulation point from below, since $\eta_n(\mathbf N_n)$ is connected and bricks intersect only along fronts,  there is no $t \in [t_\infty -\delta, t_\infty)$ such that  $\chi_\mathrm{stab} (D_{\eta_n}(t))=\chi(D^-_{\eta_n}(t_\infty))$, and hence for any $t \in [t_\infty -\delta, t_\infty)$, we have $\chi_\mathrm{stab} (D_{\eta_n}(t)) < \chi(D^-_{\eta_n}(t_\infty))$.
Therefore if there is an accumulation point $t$ of $C$ in $[t_\infty-\delta, t_\infty)$, then $\max\{\chi(D^-_{\eta_n}(t)), \chi(D^+_{\eta_n}(t))\} < \chi(D^-_{\eta_n}(t_\infty))$ for large $n$.
The same holds when $t_\infty$ is an accumulation point from above just changing $-$ to $+$.
Repeating the same argument using \cref{region}, we can take a neighbourhood $[t-\delta_t, t+\delta_t]$ such that for any $t' \in [t-\delta_t, t+\delta_t]$, we have $\max\{\chi(D^-_{\eta_n}(t')), \chi(D^+_{\eta_n}(t\))\} < $\max\{\chi(D^-_{\eta_n}(t)), \chi(D^+_{\eta_n}(t))\}$.
Since $\max\{\chi(D^-_{\eta_n}(t)), \chi(D^+_{\eta_n}(t))\}$ is between $\chi(S)$ and $0$, in finite steps, we reach a situation where there are no accumulation points  in a neighbourhood.
Thus we have shown that there are only countably many accumulation points of $C$, 
and hence $\bar C$ is countable.
\end{proof}
Thus we have proved (6).

For later use, using (6), we modify $\eta_n$ further as follows.
If there is a brick $B$ whose  front (real or ideal) is mapped by $\eta_n$ into $S \times \{t\}$ for an accumulation point $t$ of $C$, we can perturb $\eta_n|B$ off from $S \times \{t\}$ for all $n$ at the same time since there are only countably many accumulation points
In the same way,  if there are two sequence of  bricks $\{B_j\}$ and $\{B_j'\}$ such that $\lim_j \sup B_j=t=\lim_j \inf B_j'$, and if $\sup B_i$ converges to $t$ from below whereas $\inf B_i'$ converges from above,
then by using \cref{region}, we can perturb $\eta_n|B_n'$ so that $\lim_i \sup B_j' >t$ holds.

Thus we can assume the following.
\begin{assump}
\label{simple accumulation}
No point of $C$ is  is an accumulation point of $C$.
To each accumulation point of $C$ points of $C$ accumulate either only from below or only from above and cannot accumulate from both sides at the same time.
\end{assump}

Next we turn to the most difficult one, the condition (1), whose proof was given in \S 3.1 of \cite{OhSo}.
The subtle point is that we cannot construct an embedding of $\mathbf M$ into $S \times [0,1]$ simply as a limit of the  embeddings $\eta_n$ of $\mathbf N_n$.
This is because for distinct $m$ and $n$, the images of a brick by $\eta_m$ and $\eta_n$ may be different even homotopically.
Therefore, instead of taking a limit of $\eta_n$, we construct embeddings $h_n$ of $\mathbf N_n$ inductively which stabilise if we restrict them  to each brick.
The embedding $h_n$ is not an extension of the previous $h_{n-1}$, but an extension of an embedding obtained by \lq twisting' $h_{n-1}$.
We shall use the original embeddings $\eta_n$ in a step of induction to define \lq twisting maps'.
Although $h_n$ and $\eta_n$ send each brick to the same horizontal levels, there is no simple direct relation between the two.

First we shall explain intuitively how to construct $h_n$ provided $h_{n-1}$ is already defined.
We consider a decomposition of $[0,1]$ into subintervals by setting the set of the sup and the inf of the bricks of $\mathbf N_n$ to be the dividing points.
Then we subdivide the bricks of $\mathbf N_n$ so that each brick is contained in the product of $S$ and one of the above subintervals.
We give an order, defined using the horizontal levels, to the set $\mathcal B_n$ of bricks (after the subdivision) which are contained in $\mathbf N_n$ but not in $\mathbf N_{n-1}$.
Let $B$ be a brick in this set, and suppose that we have already defined $h_n$ for bricks in $\mathcal B_n$ which are smaller with respect to the given order, and denote this embedding by $h_n^B$ and the union of bricks in $\mathcal B_n$ smaller than $B$ by $\mathbf B(B)$.
We cannot always extend $h_n^B$ to $B$ since even  if both $\partial_- B$ and $\partial_+ B$ are contained in $\mathbf N_{n-1} \cup \mathbf B(B)$,  their images under $h_n^B$ may not be vertically isotopic.
(If they are vertically isotopic, then we can define $h_n|B$ simply by extending $h_n^B$ so that the horizontal level of $h_n(B)$ is the same as that of $\eta_n(B)$.)
We shall avoid these difficulties by composing to $h_n^B$ what we call a solid twisting map, which is a map obtained by cutting $S \times [0,1]$ at two levels $\Sigma \times \{a\}$ and $\Sigma \times \{c\}$ with an essential subsurface $\Sigma$ of $S$, and twisting the part $\Sigma \times [a,c]$ by using a homeomorphism from $\Sigma$ to $\Sigma$ fixing the boundary.

Now we start a formal discussion.
We let $C_n $  be a subset of $C$ defined by  $C_n=\{\sup B', \inf B' \mid  \text{the }  B' \text{ are  bricks in }  \mathbf N_{n}\}$, and number the elements of $C_n$ as $a_1, \dots ,a_s$.
We subdivide the brick decomposition of $\mathbf N_n$ by cutting them along the level surfaces corresponding to  $S \times \{a_j\}$ for $j=1, \dots, s$ via $\eta_n$.
By this operation, the image of each brick in $\mathbf N_n$ under $\eta_n$ is contained in $S\times [a_j, a_{j+1}]$ for some $j =1, \dots, s$, where we set $a_0$ to be $0$ and $a_{s+1}$ to be $1$.
For each $a_j$ among $a_1, \dots, a_s$, we shall define a twisting map $\varphi_{a_j}: S \rightarrow S$, which is a homeomorphism, and a solid twisting map $\hat\varphi_{a_j}: S \times [0,1] \rightarrow S \times [0,1]$ for $j=1,\dots , s$, which is discontinuous only along at most two horizontal subsurfaces embedded in $S \times [0,1]$.
Depending on the location of $a_j$, the map $\hat\varphi_{a_j}$ twists either a compact submanifold of $S \times [0,1]$ above $S \times \{a_j\}$ or below $S \times \{a_j\}$ and leaves the remaining part of $S \times [0,1]$ unchanged.
To determine which region we twist, we first need to define an accumulation point of $C$ to which $a_j$ \lq belongs'.

Recall that a subsurface $D_{\eta_n}(x)$ of $S$ was defined for any $x \in [0,1]$ and $n \in \naturals$.
Suppose that $d$ is an accumulation point of $C$.
By Assumption \ref{simple accumulation}, we see that $d$ is not contained in $C$ and that $C$ accumulates to $d$ either from above or from below, not from both at the same time.
We set $\Sigma_{\eta_n}(d)$ to be the complement of $D_{\eta_n}(d)$ on $S$.
By Claim \ref{region},
 after passing to a subsequence of $\{\mathbf N_n\}$, we can find $\delta >0$  such that for any $n$ and $x \in [d-\delta, d+\delta]$, the subsurface $\Sigma_{\eta_m}(x)$ is vertically isotopic into $\Sigma_{\eta_n}(d)$.
Furthermore, by Assumption \ref{simple accumulation}, by taking smaller $\delta$ if necessary, we can assume the following.
\begin{assump}
\label{one side empty}
On the side from which $C$ does not accumulate, $\Sigma_{\eta_n}(x)$ is vertically parallel to $\Sigma_{\eta_n}(d)$ for every $x \in [d-\delta, d+\delta]$ and large $n$.
\end{assump}
We note that  $\delta$ above may depend on $d$ but not on $n$ if it is large enough.
Now since there are only finitely many bricks whose images under the $\eta_n$ intersect $S \times \{d\}$, by taking $\delta$ small enough, changing the order of appearance of finitely many bricks and passing to a subsequence, we can assume the following.
\begin{assump}
\label{change order}
For every accumulation point $d$ of $C$ and $\delta>0$ as above, the first brick put in $S \times (d-\delta, d)$ or $S \times (d, d+\delta)$ appears after all bricks intersecting $S \times \{d\}$, that is, if a brick $B$ with $\eta_n(B) \subset S \times (d-\delta, d) \cup S \times (d, d+\delta)$ is contained in $\mathbf N_n$, for any $B'$ with $\eta_m(B') \cap S \times \{d\} \neq \emptyset$, there is $j < n$ with $B' \subset \mathbf N_j$.
\end{assump}

Let $\{d_i\}$ be the set of accumulation points of $C$, and for each $d_i$, we take $\delta_i$ so that 
$[d_i-\delta_i, d_i)\cup(d_i, d_i+\delta_i]$ has the properties of Claim \ref{region} and Assumptions \ref{one side empty} and \ref{change order}.
We  set $I_i$ to be $(d_i-\delta_i, d_i) \cup (d_i, d_i+\delta_i)$.
We note that it may be possible that two distinct $I_{i_1}$ and $I_{i_2}$ intersect, but we can choose smaller $\delta_i$ if necessary so that if $I_{i_1} \cap I_{i_2} \neq \emptyset$, then one of the two contains the other.
By cutting bricks along $S \times \{d_i-\delta_i\}$ and $S \times \{d_i+\delta_i\}$, we can assume the following:
\begin{assump}
\label{only one}
For a brick $B$, if one of its sup and inf is contained in $I_i$, then the other is contained in its closure $\bar I_i$.
\end{assump}
We say that a point $a \in C$ {\em belongs to} an accumulation point $d_i$ if $a$ is contained in $I_i$ and if $I_i$ is the smallest with respect to the inclusion among the intervals $I_i$ containing $a$.
We say that a brick $B$ belongs to $d_i$ if $\bar I_i$ is the smallest among the intervals $\{\bar I_i\}$ that contains both $\sup B$ and $\inf B$.

Let $a$ be a point in $C_n \setminus\{0,1\}$.
We say that a point $a \in C_n$ belonging to $d_i$ is an {\em lower twisting point }when $a$ lies in $(d_i-\delta_i, d_i)$, and a {\em upper twisting point} otherwise.
Since there are only finitely many points of $C$ which are contained in none of the $I_i$, by passing to a subsequence, we can assume the following.
\begin{assump}
If a brick $B$ has $\inf$ and $\sup$ neither of which contained in any $I_i$, then $B$ is contained in $\mathbf N_1$.
\end{assump}

Twisting maps (or solid twisting maps) at lower twisting points, which we shall define below, will be called lower (solid) twisting maps, and those at upper twisting points upper (solid) twisting maps. 
We shall  define a new embedding $h_n$ for  $\mathbf N_n$ for every $n \in \naturals$ inductively, starting from $h_1$, which is set to be $\eta_1$.
The embedding $h_n$ will be defined to each brick with the same $\sup$ and $\inf$ as the original embedding $\eta_n$, but our construction is based on the induction, and $\eta_n$ will be used only to define twisting maps.
Suppose that the embedding $h_{n-1}$ is defined on $\mathbf N_{n-1}$ in such a way that $h_{n-1}$ sends each brick of $\mathbf N_{n-1}$  to the same horizontal levels as $\eta_{n-1}$ (and hence also $\eta_n$) does.
We now start to define $h_n$ on $\mathbf N_n$.

Let $\mathbf B_n$ be the set of bricks contained in $\mathbf N_n$ but not in $\mathbf N_{n-1}$.
We let $C_n'$ be  a subset of $C_n$ defined by $C_n' = \{\sup B, \inf B \mid B \in \mathbf B_n \} \cap (\cup_j I_j)$.
We consider accumulation points of $C$ to which some point of $C_n'$ belongs.
Since $\mathbf N_n$ has only finitely many bricks, there are only finitely many such accumulation points.
We renumber such accumulation points and the corresponding intervals as $d_1, \dots, d_r$  and $I_1, \dots , I_r$ in such a way that if $I_{j_1}$ is (properly) contained in $I_{j_2}$ for $d_{j_1}, d_{j_2} \in \{d_1, \dots , d_r\}$, then $j_1 > j_2$.

Starting from $I_1$, and in the order according to the subscripts,  we shall define twisting maps for each brick $B$ in $\mathbf B_n$ which belongs to $d_j$,  and  $h_n$ on $B$ after twisting the embedding defined up to that point.
We let $\mathbf N_n^{j-1}$  be the union of bricks of $\mathbf N_n$ and bricks in $\mathbf B_n$ belonging to $d_1, \dots , d_{j-1}$.
Now, assuming that we have already defined the embedding $h_n^{j-1}$ for  $\mathbf N_n^{j-1}$, we shall define an embedding $h_n^j$ for $\mathbf N_n^j$, which is not necessarily an extension of $h_n^{j-1}$.
By Assumption \ref{one side empty}, either all points of $C'_n$ belonging to $d_j$ are upper twisting points or all of them are lower twisting points.
Since the argument is quite the same for both cases, we here assume that they are lower twisting points.
Let $\{a_1, \dots , a_p\}$ be the points in $C'_n $ belonging to $d_j$, aligned in the increasing order.
We shall proceed again inductively to define a lower twisting map $\varphi_{a_k}$ for each $a_k$ among $a_1, \dots , a_p$ and the corresponding lower solid twisting map $\hat \varphi_{a_k}: S\times [0,1] \to S \times [0,1]$

Let $a_k$ be one among $\{a_1, \dots , a_p\}$, and suppose that we have already constructed solid twisting maps $\hat \varphi_{a_1}, \dots , \hat \varphi_{a_k}$ in such a way that $\hat \varphi_{a_k} \circ \dots \circ \hat \varphi_{a_1} \circ h_{n}^{j-1}$ extends to an embedding $h_n^{j-1}(k-1)$ of the union of $\mathbf N_n^{j-1}$ and the bricks in $\mathbf B_n$ whose suprema are in $\{a_1, \dots , a_{k-1}\}$, which we denote by $\mathbf  N_n^{j-1}(k-1)$.
We let  $B_{l_1}, \dots , B_{l_t}$ be the bricks with suprema at $a_k$ which are contained in either $\mathbf B_n$ or  $\mathbf N_{n-1}$ in such a way that their images under $h_{n}^{j-1}(k-1)$ are contained in  $\Sigma_{h_n^{j-1}}(d_j) \times [0, a_k]$.
We now define a lower twisting map  $\varphi_{a_k}: S \rightarrow S$ at $a_k$, which is supported on $\Sigma_{h_n^{j-1}(k-1)}(d_j)$ and the corresponding lower solid twisting map $\hat \varphi_{a_k}$ which is supported in $\Sigma_{h_n^{j-1}(k-1)}(d_j) \times [a_k, d_j)$.

Since both $\eta_n$ and $h_n^{j-1}(k-1)$ embed $\mathbf N_{n-1}^j(k-1)$ into $S \times [0,1]$ and they send each brick at the same horizontal level, there is a homeomorphism $\varsigma_n^j(k): h_n^{j-1}(k-1)(\mathbf N_{n-1}^j(k-1)) \to \eta_n(\mathbf N_{n-1}^j(k-1))$ preserving the horizontal levels.
Since $\eta_n$ also embeds $\mathbf N_{n-1}^j(k-1) \cup (B_{l_1}, \dots , B_{l_t})$, for each $l$ among $l_1, \dots , l_t$, its bottom $\eta_n(\partial_-B_l)$ is vertically homotopic to its top $\eta_n(\partial_+ B_l)$ if $\partial_+ B_l$ is contained in $\mathbf N_{n-1}^j(k-1)$.
Therefore there is a homeomorphism $\varphi_{a_k} : S \rightarrow S$  which is supported on $\Sigma_{h_n^{j-1}(k-1)}(d_i)$ such that $h_n^{j-1}(k-1)(\partial_- B_l)$ is vertically homotopic to $h_n^{j-1}(k-1)(\partial_+ B_l)$ if $\partial_+ B_l$ is contained in $\mathbf N_{n-1}^j(k-1)$.
We let this $\varphi_{a_k}$ be the lower twisting map at $a_k$, and define the corresponding lower solid twisting map $\hat \varphi_{a_k} : S\times [0,1] \to S \times [0,1]$  to be $(\varphi_{a_k}(x), t)$  for $(x,t) \in \Sigma_{h_n^{j-1}(k-1)}(d_i) \times [a_k, d_j)$ and the identity elsewhere.
Then $\hat \varphi_{a_k} \circ h_n^{j-1}(k-1)$ extends to an embedding of $\mathbf N_{n-1}^j(k-1) \cup (B_{l_1}, \dots , B_{l_t})$.
Repeating this construction inductively on $a_1, \dots , a_p$ first then $d_1, \dots, d_r$ next, we get an embedding $h_n$ of $\mathbf N_n$.

We shall now show that for each brick $B$ of $\mathbf M$, the restriction $h_n|B$ stabilises for large $n$.
(Here we are considering the original brick decomposition, not the subdivided one as the argument above.)

For a given brick $B$, take the smallest $m$ such that $B$ is contained in $\mathbf M_m$.
We consider twisting maps  at points  $a_i \in C_{n_i}$  for some $n_i >m$ belonging to an accumulation point $d_i$, which appear in the construction of the embeddings above.
We shall show that there are finitely many $i$ and $n_i$ for which the solid twisting map affects the embedding of $B$.
We first consider the case where  $h_m(B) \cap (S \times (a_i, d_i)) =\emptyset$ ($a_i < d_i$) or $h_m(B) \cap (S \times (d_i, a_i)) = \emptyset$ ($a_i > d_i$).
For such $a_i$, the solid twisting map does not affect the embedding of $B$ because the solid twitting map only moves $S \times (a_i, d_i)$ or $S \times (d_i, a_i)$ in this case.
Next suppose that either $h_m(B) \cap (S \times (a_i, d_i))$ or  $h_m(B) \cap (S \times (d_i, a_i))$ is non-empty.
Since the argument is the same for both cases, we only consider the case when $a_i$ lies in $(d_i-\delta_i, d_i)$ and $h_m(B) \cap (S \times (a_i, d_i)) \neq \emptyset$.
If $d_i$ is contained in $(\inf B, \sup B]$, then $h_{n-1}(B) \cap S \times \{d_i\}$ is contained in $D_{h_{n-1}}(d_i)$, which implies that $h_{n-1}(B)$ is outside the support of the solid twisting map at $a_i$.
Therefore, we can assume that $d_i > \sup B$.
Since $a_i < \sup B < d_i$, if there are infinitely many such $a_i \in C_{n_i}$ with $n_i > m$, the sequence $\{a_i\}$ consisting of these points has an accumulation point $a_\infty \leq \sup B$.
Since $a_\infty$ is an accumulation point of $C$, there is $j$ such that $a_\infty=d_j$.
For sufficiently large $i$, the point $a_i$ is contained in $I_j$.
Since $a_i$ belongs to $d_i$, this implies that $I_i$ is contained in $I_j$.
Since $d_i-\delta_i < a_i \leq d_j=a_\infty \leq \sup B < d_i$, we see that $I_i$, which contains $(d_i-\delta_i, d_i)$ cannot be contained in $I_j$ which does not contain $d_j$.
This is a contradiction.
Thus, we have shown that there are only finitely many solid twisting maps which affect the embeddings of $B$.
Therefore, $\{h_n\}$ stabilises on each brick for large $n$, and the limit of $\{h_n\}$ is a well-defined embedding of $\mathbf M$ into $S \times [0,1]$.

The condition (4) is much easier to see.
As was shown in the proofs of (3) and (5),  each geometrically finite end appears as a geometric limit of  geometrically finite ends in  boundary blocks of the $\mathbf M_i$.
Since they can be assumed to lie on $S \times \{0,1\}$ by the embedding $h_i$ of $\mathbf M_i$ constructed above, we can see the limit also lies on $S\times \{0,1\}$.

We shall next check the condition (7).
We fix a generator system of $\pi_1(S)$ so that none of the generators has image under $\phi_i$ which converges to a parabolic element in $\Gamma$ as $i \rightarrow \infty$.
Since $\{(G_i,\phi_i)\}$ converges algebraically,  the geodesic loops representing the generators based at $y_i$ have bounded lengths as $i \rightarrow \infty$.
Therefore, we can construct a map $g_i : S \rightarrow M_i$ triangulated by ideal triangles whose image contains these geodesic loops so that $\{g_i\}$ converges geometrically to a map $\hat  g : S \rightarrow M_\infty$ which can be lifted to a map from $S$ to $M'$.
Then $\{f_i^{-1} \circ g_i\}$ also converges to $f^{-1} \circ \hat g$, which we set to be $g$.
By construction, $g$ is $\pi_1$-injective also as  a map to $S \times [0,1]$.
Moreover, since $\hat g=f\circ g$ can be lifted to $M'$ as a homotopy equivalence, we see that $(f\circ g)_\#\pi_1(S)$ corresponds to the image of $\pi_1(M')$.
%
%
%
%

Finally, we shall check the condition (8).
Since we chose $\eta_n$ so that no ideal front intersects another front in $S \times [0,1]$, the embedding of $\mathbf M$ which we constructed has the same property.
Since we isotoped $\eta_n$ so that there is no accumulation point of $C$ contained in $C$ itself and there is no accumulation point of $C$ to which points of $C$ accumulate from both below and above as in \cref{simple accumulation}, there is no wild end which intersects a front of a brick or another wild end.

\subsection{Block decomposition}
\label{label}
Regard $\mathbf M$ as a subset of $S \times [0,1]$ which is embedded preserving the horizontal and vertical foliations as in Theorem \ref{Ohshika-Soma}.
Then a brick of $B$ has a form $\Sigma \times J$ with respect to the parametrisation of $S\times [0,1]$.
We denote $\sup J$ by $\sup B$ and $\inf J$ by $\inf B$ as in the previous subsection.
Note that $\sup B$ is the level of the horizontal leaf on which the upper front of $B$ lies and $\inf B$ that on which the lower front of $B$ lies.
Each end of $\mathbf M$, even if it is wild, corresponds to $\Sigma \times \{t\}$ for some incompressible subsurface $\Sigma$ of $S$.
By the condition (2), every geometrically finite end is contained in either $S \times \{0\}$ or $S \times \{1\}$.
We call those contained in $S\times \{0\}$  {\em lower geometrically finite ends} and those in $S\times \{1\}$ {\em upper geometrically finite ends}.
By moving the embedding vertically if necessary, we can assume that $S \times \{0\}$ and $S \times \{1\}$ consist of a union of ends, annuli homotopic  in $S \times [0,1] \setminus \Int \mathbf M$ to the closure of open annulus boundary components of $\mathbf M$, and open annuli corresponding to punctures of $S$.

Occasionally, it is  convenient to consider the complement of $\mathbf M$ in $S \times (0,1)$.
Let $C$ be a component of $S\times (0,1) \setminus \mathbf M$.
Then $\Fr C \cap C$ consists of (countably many) horizontal surfaces, each corresponding to an end of $\mathbf M$ which is either simply degenerate or wild.
(Note that by the condition (8), each component of $\Fr C \cap C$ corresponds to only one end of $\mathbf M$.)
On the other hand, $\Fr C \setminus C$ consists of either annuli or a single torus, which are boundary components of $\mathbf M$.



We can associate to each geometrically finite end of $\mathbf M$  a marked conformal structure at infinity of the corresponding geometrically finite end of $(M_\infty)_0$.
As was shown in the proof of Lemma \ref{model limit}, each geometrically finite end is contained in a geometric limit of boundary blocks of $\mathbf M_i$.
Therefore, this conformal structure at infinity given on each geometrically finite end coincides with the conformal structure at infinity of the corresponding geometrically finite block.
Similarly, we can associate to each simply degenerate end of $\mathbf M$ the ending lamination of the corresponding simply degenerate end of $(M_\infty)_0$.
We call these conformal structures and ending laminations {\em labels}.

In general a brick manifold each of whose non-wild ends has a label, an arational lamination for a simply degenerate end, and a conformal structure at infinity for a geometrically finite end, is called  a {\em labelled brick manifold}.
We showed in \cite{OhSo} that any labelled brick manifold is decomposed into blocks in the sense of Minsky \cite{Mi} and  tubes.
This is just a generalisation of Minsky's construction of model manifolds in \cite{Mi} based on hierarchies of tight geodesics defined in \cite{MaMi}.
In particular in the case of Kleinian surface groups, our decomposition into blocks coincides with the construction of Minsky's model manifolds completely.

In our present situation, we do not need this general theory since $\mathbf M[0]$, which is a geometric limit of $\mathbf M_[0]$, has a decomposition into blocks which is  a limit  block decomposition of $\mathbf M_i[0]$ as was stated in the part (0) of Theorem \ref{Ohshika-Soma}.
We put a metric of a Margulis tube into each tube so that the flat metric induced on the boundary coincides with that induced from the metric on $\mathbf M[0]$ determined by blocks.
As was explained in Proof of Theorem \ref{Ohshika-Soma}, each Margulis tube $V=A \times [s,t]$ has a coefficient $\omega_{\mathbf M}(V)$ lying in $\{z \in \complexes \mid \Im z >0\}$.
We define $\mathbf M[k]$ to be the complement of the tubes whose $\omega_{\mathbf M}$ have absolute value greater than or equal to $k$.
By defining a brick to be the closure of a maximal union of parallel horizontal leaves, we can define a brick decomposition of $\mathbf M_i[k]$ and $\mathbf M[k]$.
Such a brick decomposition is called the {\em standard brick decomposition}.

The following proposition was first shown in  the proof of Theorem A in \cite{OhSo} (\S 5.2).
The claims in this proposition were already proved in Theorem \ref{Ohshika-Soma} above except for the existence of labels.
The labels can be given by pulling back those on $(M_\infty)_0$ as we have just explained above.
\begin{proposition}
\label{limits of models}
In the situation of \cref{Ohshika-Soma},
let $x_i \in \mathbf M_i$ be a point in $\mathbf M_i$ such that $f_i(x_i)=y_i$.
Then $(\mathbf M_i[0], x_i)$ converges to $\mathbf M[0]$  after passing to a subsequence.
The model manifolds $\mathbf M$ and $\mathbf M_i$ have  structures of labelled brick manifolds admitting  block decompositions with the following conditions.

Let  $\rho_i^{\mathbf M}$ be an approximate isometry  between $\mathbf M_i$ and the union of $\mathbf M$ and cusp neighbourhoods, which corresponds to the geometric convergence.
Then we can arrange $\rho_i^{\mathbf M}$ so that the following hold.
\begin{enumerate}
\item
For any compact set $K$ in $\mathbf M$, the restriction $\rho_i \circ f_i \circ (\rho_i^{\mathbf M})^{-1}|K$ converges to $f|K$ uniformly as $i \rightarrow \infty$.
\item
 For any block $b$ of $\mathbf M[0]$, its pull-back $(\rho_i^{\mathbf M})^{-1}(b)$ is a block in $\mathbf M_i[0]$ for large $i$.
 
\item $\rho_i^{\mathbf M}$ preserves the 
horizontal foliations.
\end{enumerate}
\end{proposition}
\subsection{Algebraic limits in the models}
\label{algebraic limit}
By Theorem \ref{Ohshika-Soma}-(7), there is an inclusion of $\pi_1(S)$ in $\pi_1(\mathbf M)$ corresponding to the inclusion of $\Gamma$ into $G_\infty$.
We realise this inclusion by a $\pi_1$-injective immersion $g: S \rightarrow \mathbf M$ so that $(f \circ g)_\# \pi_1(S)$ is equal to the image of $\pi_1(M')$ in $\pi_1(M_\infty)$ under the covering projection as in Theorem \ref{Ohshika-Soma}-(7).
We call such $g$ an {\em algebraic locus}.

\begin{lemma}
\label{position of S}
An algebraic locus $g$ can be homotoped to a map $g'$ as follows.
\begin{enumerate}
\item The surface $S$  is decomposed into incompressible subsurfaces $\Sigma_1, \dots , \Sigma_m$, none of which is an annulus, and (possibly empty) annuli $A_1, \dots , A_\mu$.
\item The restriction of $g'$ to $\Sigma_j$ is a horizontal embedding into $\Sigma_j \times \{t_j\} \subset \mathbf M$.
\item Each annulus $g'(A_j)$ is composed of $2n-1$ horizontal annuli and $2n$ vertical annuli for some $n \in \naturals$ and goes around a torus boundary of $\mathbf M$ $n$-times. See Figure \ref{fig:go-around}.
\end{enumerate}
\end{lemma}
\begin{proof}
Recall that $g$ is $\pi_1$-injective even as a map to $S\times [0,1]$.
Since every $\pi_1$-injective map from $S$ to $S \times [0,1]$ is homotopic to a horizontal surface (see Proposition 3.1 of Waldhausen \cite{Wa}), $g$ is homotopic to a horizontal surface $S \times \{t\}$ in $S \times [0,1]$.

Since $\mathbf M$ is a brick manifold, we can homotope $g$ within $\mathbf M$ so that $g(S)$ consists of horizontal leaves in bricks and vertical annuli. 
By the additivity of  Euler characteristics, we see that the sum of the Euler characteristics of the horizontal leaves is equal to $\chi(S)$.
We consider the projection of horizontal leaves to $S$.
Since $g$ is homotopic to $S \times \{t\}$ in $S \times I$, we see by the invariance of the algebraic intersection number that for each point $x \in S$ the surface $g(S)$ contains $x \times \{s\}$ for some $s \in I$. 
This implies that the horizontal leaves cannot overlap along a surface with negative Euler characteristic.
It follows that only compact regions that $g(S)$ can bound in $S \times I$ are solid tori.
If such a solid torus is contained in $\mathbf M$, we can eliminate it by a homotopy.
The only remaining possibility is that such a solid torus contains  components of $\partial \mathbf M$.
By (2) of Theorem \ref{Ohshika-Soma}, there is only one boundary component contained in each solid torus.
Thus we have reached the situation is as in our statement.
\end{proof}

We call a map $g': S \rightarrow \mathbf M$ as in Lemma \ref{position of S} {\em a standard algebraic immersion}.
Recall that there is a homotopy equivalence $\Phi_i: S \rightarrow M_i$ realising $\phi_i$.
By composing the inverse of the model map, we have a homotopy equivalence from $S$ to $\mathbf M_i$, which {\em we denote by $\Phi_i^{\mathbf M}$}.

\begin{lemma}
\label{marking by g'}
Let $\rho_i^{\mathbf M}$ be an approximate isometry between $\mathbf M_i$ and $\mathbf M$ with basepoints at the thick parts as in Theorem \ref{Ohshika-Soma} and Proposition \ref{limits of models}.
For sufficiently large $i$, the immersion $(\rho_i^{\mathbf M})^{-1} \circ g'$ is homotopic to $\Phi_i^{\mathbf M}$ as a map to $\mathbf M_i$.
\end{lemma}
\begin{proof}
By the definition of $g'$, we see that $f \circ g'$ is homotopic to $\Psi$.
Since $\rho_i^{-1} \circ \Psi$ is homotopic to $\Phi_i$ for large $i$, our lemma follows from the condition (1) of Proposition \ref{limits of models}.
\end{proof}

\begin{figure}
\scalebox{0.5}{\includegraphics{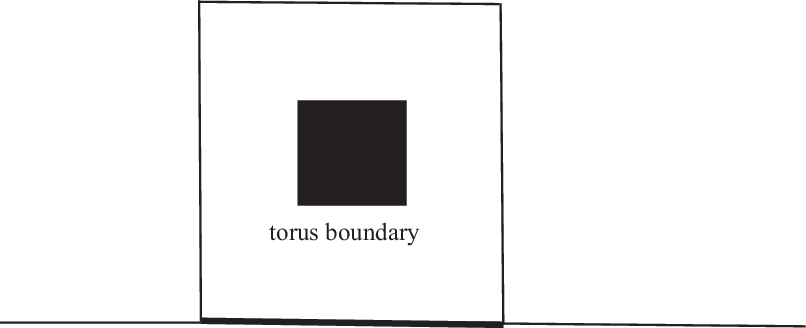}}
\caption{$g'$ going around a torus boundary component (the case when $n=1$).}
\label{fig:go-around}
\end{figure}

\begin{definition}
\label{upper/lower}
Let $e$ be a simply degenerate end of $\mathbf M$ contained in a brick $B=\Sigma \times J$.
We say that $e$ is {\em an upper algebraic end} if $J=[s, s')$ and $\Sigma \times \{s'-\epsilon\}$ is freely homotopic  in $\mathbf M$ to $g'(\Sigma)$ which lies in  a horizontal part of $g'(S)$ for sufficiently small $\epsilon >0$.
In the same way, we say that $e$ is {\em a lower algebraic end} if $J=(s, s']$ and $\Sigma \times \{s+\epsilon\}$ is freely homotopic   in $\mathbf M$ to $g'(\Sigma)$ which lies in a horizontal part $g'(S)$ for sufficiently small $\epsilon >0$.
We also call the ending lamination of an algebraic simply degenerate end algebraic, and say that an ending lamination is upper or lower depending on whether the end is upper or lower.

Similarly, a core curve of an open annulus boundary component or a longitude (\ie a horizontal curve) of a torus boundary component of $\mathbf M$ is said to be an {\em algebraic parabolic curve} if it is homotopic   in $\mathbf M$ to a simple closed curve lying on a horizontal part of $g'(S)$.
We also call its image of the vertical projection to $S$ an algebraic parabolic curve.
A parabolic curve is said to be upper or lower  in the same way as simply degenerate ends, but we should note if it lies on a torus boundary component around which $g'(S)$ goes (by this we mean that $g'(S)$ goes around the torus boundary component $k$-times with $k \neq 0, -1$), then {\em the curve is defined to be both upper and lower }at the same time.
An algebraic parabolic curve is said to be {\em isolated} if it is a core curve of an isolated parabolic locus of $M'_0$. 
\end{definition}

\begin{lemma}
\label{isolated}
Any algebraic parabolic curve lying on a torus boundary component of $\mathbf M$ is isolated.
\end{lemma}
\begin{proof}
Let $c$ be an algebraic parabolic curve lying on a torus boundary component of $\mathbf M$.
Let $P$ be a parabolic locus of a relative core $C$ of $(M')_0$ into which the lift of $f(c)$ to $M'$ is homotopic.
Suppose, seeking a contradiction, that $P$ is not isolated.
Then, there is a simply degenerate end $e$ touching the $\integers$-cusp corresponding to $P$.
By Thurston's covering theorem (see \cite{Th} and Canary \cite{CaT}) together with the argument in the proof of  Lemma 2.3 in  \cite{OhQ}, there is a neighbourhood $U$ of $e$ which is projected in to $M_\infty$ homeomorphically by the covering projection.
This implies that $\mathbf M$ has a simply degenerate brick which touches an open annulus boundary of $\mathbf M$ into which $c$ is homotopic.
Since no two distinct boundary components of $\mathbf M$ have homotopic essential closed curves by Theorem \ref{Ohshika-Soma}-(2), this contradicts the assumption that $c$ lies on a torus boundary.
\end{proof}

\begin{lemma}
\label{ends of M'}
The algebraic simply degenerate ends of $\mathbf M$  correspond one-to-one to the simply degenerate ends of $M_0'$ by mapping them by $f$ and lifting them to $M'_0$.
The upper (resp. lower) ends correspond  to upper (resp. lower) ends of $M'_0$.
\end{lemma}
\begin{proof}
Consider a simply degenerate end corresponding to the upper ideal front of a brick $B=\Sigma \times [s,s')$ of $\mathbf M$.
By the definition of the model map $f$,  there is an infinite sequence of horizontal surfaces $\Sigma \times \{t_j\}$ in $B$ tending to $\Sigma \times \{s'\}$ which are mapped to a sequence of pleated surfaces $f(\Sigma \times \{t_j\})$ tending to  the corresponding simply degenerate end $e$ of $(M_\infty)_0$.
Since $\Sigma \times \{t\} \in B$ is freely homotopic into $g'(S)$, the pleated surfaces $f(\Sigma \times \{t_j\})$ lift to pleated surfaces $\tilde f_i$ tending to an end of $M'_0$.

Since the model map $f$ has degree $1$ with respect to the orientation of $(M_\infty)_0$, the end $e$ is situated above $f \circ g'(S)$.
Lifting this to $M'_0$, we see that the end to which the $\tilde f_i$ tend  is an upper end.
Similarly, we can show that if the simply degenerate brick $B$ has the form $\Sigma \times (s,s']$, then the corresponding end of $M'_0$ is a lower end.

Conversely, suppose that $e'$ is a simply degenerate end of $M'_0$.
By Thurston's covering theorem and some argument applying it
(\cite{Th}, \cite{CaT} and the argument in the proof of Lemma 2.3 of \cite{OhQ}), there is a neighbourhood $E$ of $e'$ such that $p|E$ is a proper embedding into $(M_\infty)_0$.
Let $\bar e$ denote the simply degenerate end of $(M_\infty)_0$ contained in $p(E)$.
Then, there is a simply degenerate end $\hat e$ of $\mathbf M$ which is sent to $\bar e$ by $f$.
Since $e'$ is simply degenerate, there is an essential subsurface $\Sigma$ of $S$ and a sequence of pleated surfaces $h_i: \Sigma \rightarrow M'_0$ taking $\partial \Sigma$ into $\partial M'_0$ which tend to $e'$.
Their projections $p \circ h_i$ are pleated surfaces tending to $\bar e$.
This implies that the end $\hat e$ is contained in a simply degenerate brick $B_e \cong \Sigma \times J$, where $J$ is a half-open interval.
Since $f(\Sigma \times \{t\})$ is homotopic to $p \circ h_i$ and $p \circ h_i$ is homotopic into $f \circ g'(S)$, we see that $\Sigma \times \{t\}$ is freely homotopic to $g'(\Sigma)$.
By Lemma \ref{position of S}, this is possible only when $\Sigma \times \{t\}$ is homotopic into a horizontal part of $g'(S)$, and we see that the end $\hat e$ is algebraic.
\end{proof}

\subsection{Simply degenerate ends in limit models}
\label{proof of prop}
As was shown in Theorem \ref{Ohshika-Soma}, except for the geometrically finite  ends lying on $S \times \{0,1\}$, all the tame ends of $\mathbf M$ are simply degenerate.
Next we shall see how simply degenerate ends in the model manifold $\mathbf M$ are approximated in $\mathbf M_i$.
Recall that the model manifold $\mathbf M_i$ corresponds to a hierarchy $h_i$ of tight geodesics.
Recall also that we have a homeomorphism $\hat \Phi_i : S \times (0,1) \rightarrow M_i$ inducing $\phi_i$ between the fundamental groups.
This determines an embedding $\iota_i=f_i^{-1} \circ \hat \Phi_i$ of the standard $S \times (0,1)$ into $S \times [0,1]$ in which $\mathbf M_i$ is embedded.
We identify the standard $S \times [0,1]$ and $S \times [0,1]$ in which $\mathbf M_i$ is embedded so that this $\iota_i$ becomes an inclusion.
In other words, by this identification, the model map $f_i$ is homotopic to $\hat \Phi_i$ if regarded as a map from $S \times (0,1)$.
We identify two $S \times [0,1]$ in which  $\mathbf M_{i_1}$ and $\mathbf M_{i_2}$ are embedded respectively for every pair $i_1, i_2$ using $\iota_{i_1}$ and $\iota_{i_2}$.

We fix a complete marking $\mu$ on $S$ once and for all.
For a domain $X$ in $S$, by considering a component of $\pi_X(\base (\mu))$, we can define a basepoint in $\CC(X)$.
We call this basepoint the {\em basepoint determined by  $\mu$}.

\begin{proposition}
\label{simply degenerate brick}
Let $B=\Sigma \times J$ be a simply degenerate brick in $\mathbf M$ whose end $e$ is algebraic.
Then passing to a subsequence there
is a tight geodesic $\gamma_i$  contained in the hierarchy  $h_i$ (which was used to construct $\mathbf M_i$) as follows.
\begin{enumerate}
\item
The support of $\gamma_i$ is $\Int \Sigma$.
\item
Either  all $\gamma_i$ are geodesic rays, or they are finite geodesics whose lengths go to $\infty$ as $i \rightarrow \infty$.
\item
In the case when the $\gamma_i$ are geodesic rays, their endpoints at infinity converge to the ending lamination of $e$ in $\EL(\Int \Sigma)$ as $i \rightarrow \infty$, and $\gamma_i$ contains a simplex whose distance from the basepoint determined by $\mu$ is bounded as $i \rightarrow \infty$.
\item
Suppose that the $\gamma_i$ are finite geodesics.
In the case when the end $e$ is upper, 
the last vertex  of $\gamma_i$ converges to the ending lamination of $e$ as $i \rightarrow \infty$.
In the case when the end is lower, the first vertex of $\gamma_i$ converges to the ending lamination as $i \rightarrow \infty$.
In both cases, $\gamma_i$ contains a simplex whose distance from the basepoint determined by $\mu$ is bounded as $i \rightarrow \infty$.
\end{enumerate}
\end{proposition}

Before starting the proof of Proposition \ref{simply degenerate brick}, we shall show the following lemma which is similar to Lemma 6.2 in Masur-Minsky \cite{MaMi}.
\begin{lemma}
\label{distance implies support}
There are constants $M$ and $P$ depending only on $\xi(S)$ with the following property.
Let $h$ be a  hierarchy of tight geodesics on $S$, and $D$ a domain (\ie an open incompressible subsurface) of $S$.
Suppose that there are two vertices $v, w$ of $\CC(D)$ which are contained in simplices  of tight geodesics constituting $h$ such that $d_{\CC(D)}(v,w) > M$.
Then there is a tight geodesic in $h$ supported on $D$ which contains simplices $s_v$ and $s_w$ such that $d_{\CC(D)}(v, s_v) \leq P$ and $d_{\CC(D)}(w, s_w) \leq P$. 
\end{lemma}
\begin{proof}
Lemma 6.2 in Masur-Minsky \cite{MaMi} says that there is a constant $M'$ depending only on $\xi(S)$ such that for if $d_{\CC(D)}(I(H), T(H)) >M'$ for some hierarchy $H$ on $S$, then $D$ supports a geodesic in $H$.
Our lemma can be proved by repeating the argument of the proof Masur-Minsky's lemma or modifying the hierarchy $h$ so that we can apply Masur-Minsky's lemma.
We shall explain the latter here.

We set $M=M'+2$ and let $v,w$ be vertices in $\CC(D)$ as are given in the statement.
Let $\sigma_v$ and $\sigma_w$ be simplices on geodesics of $h$ containing $v$ and $w$ respectively.
We construct a resolution of $h$, which we denote by $\tau=\{\tau_j\}$, where $\tau_j$ is a slice and $j$ ranges in an interval in $\integers$.
By the definition of resolutions, there are slices $\tau_{j_v}$ containing $\sigma_v$ and $\tau_{j_w}$ containing $\sigma_w$.
Since $d_{\CC(D)}(v,w) >2$, they have non-zero intersection number, and hence cannot appear in the same slice, which implies that $j_v \neq j_w$.
By interchanging $v$ and $w$ if necessary, we can assume that $j_v < j_w$.

Now, we consider a subsequence of the resolution defined to be $\tau'=\{\tau_j\}_{j_v \leq j \leq j_w}$.
By the definition of elementary moves in a resolution (see \S 5 in \cite{MaMi}), for any geodesic $g$ in $h$, the simplices $v$ on $g$ such that $(g,v)$ is contained in slices of $\tau'$ form a contiguous subset of the set of simplices on $g$.
We denote this subset by $V(g)$, and a subgeodesic of $g$ consisting of simplices contained in $V(g)$ by $g'$, which might be empty.
Moreover, if $g_1 \subord (g,v)$ or $(g,v) \supord g_1$ and $V(g_1) \neq \emptyset$, then $(g,v)$ appears in some slice in $\tau'$ by the definition of slices.
Therefore, the set of geodesics $\{g'\}_{g \in h}$ with the relation of subordination inherited from $h$ forms a hierarchy on $S$ by setting its initial marking to be $\tau_{j_v}$ and its terminal marking to be $\tau_{j_w}$.
(Strictly speaking, a geodesic of this hierarchy may have the first simplex or the last simplex which may not be vertices.
This does not affect the argument for proving Lemma 6.2.)
Let $h'$ denote this hierarchy.
Then we have $d_{\CC(D)}(I(h'), T(h'))=d_{\CC(D)}(\tau_{j_v}, \tau_{j_w}) \geq d_{\CC(D)}(v,w)-2 > M'$.
Now, by applying Masur-Minsky's lemma for $h'$ and $D$, we see that  $D$ supports a geodesic $g'_D$ in $h'$.
Since $g'_D$ is a subgeodesic of a geodesic in $h$, we see that $D$ supports a geodesic in $h$.

Now we shall show the existence of simplices $s_v, s_w$ in $g_D$ with the condition given in the statement.
Lemma 6.1 of \cite{MaMi} implies that there is a constant $M_1$ depending only $\xi(S)$ which bounds both $d_{\CC(D)}(I(h'), I(g_D))$ and $d_{\CC(D)}(T(h'), T(g_D))$.
We set $P$ in the statement to be $M_1+1$.
Let $v_D$ be the first simplex of $g'_D$ and $w_D$ the last simplex of $g_D$.
Since $\sigma_v$ is contained in $I(h')$, we have $d_{\CC(D)}(\sigma_v, v_D) \leq M_1+1$.
Since $v_D$ is a simplex on $g_D$, this gives a bound as we wanted for $v$.
In the same way, we have $d_{\CC(D)}(\sigma_w, w_D) \leq M_1+1$.
This complete the proof.
\end{proof}
\begin{proof}[Proof of Proposition \ref{simply degenerate brick}]
Recall that $\{(\mathbf M_i, x_i)\}$ converges to the union of $(\mathbf M, x_\infty)$ and cusp neighbourhoods.
We denote by $\rho_i^{\mathbf M}$ a $(K_i, r_i)$-approximate isometry associated to this convergence with domain $B_{r_i}(\mathbf M_i, x_i)$ as before.
The intersection of the range of $\rho_i$ and $\mathbf M$ is the $K_ir_i$-ball centred at $x_\infty$, which we denote by $B_{K_ir_i}(\mathbf M, x_\infty)$.

Since $B$ is assumed to be simply degenerate brick, there is a sequence of  Margulis tubes $T_1, T_2, \dots $ appearing in the block decomposition of $\mathbf M$, which tend to the end of $B$ and whose core curves projected into $\Sigma$, which we denote by $c_1, c_2, \dots $, converge to the ending lamination for the end.
Passing to a subsequence, we can assume that $d_{\CC(\Sigma)}(c_{j_1}, c_{j_2}) \geq |j_1-j_2|$ for any $j_1, j_2 \in \naturals$.
For any $n \in \naturals$, we can take $i_0 \in \naturals$ such that for any $i \geq i_0$, the ball $B_{K_i r_i}(\mathbf M, x_\infty)$ contains all tubes $T_1, \dots , T_n$.
Since the end of $B$ is algebraic, the core curve of $(\rho_i^{\mathbf M})^{-1}(T_j)$ is homotopic in $\mathbf M_i$ to $c_j$ regarded as lying on $S \times \{1/2\}$.
Since $(\rho_i^{\mathbf M})^{-1}(T_1), \dots , (\rho_i^{\mathbf M})^{-1}(T_n)$ are Margulis tubes of $\mathbf M_i$, the curves $c_1, \dots , c_n$ are contained in simplices of geodesics constituting the hierarchy $h_i$.
Since $d_{\CC(\Sigma)}(c_1, c_n) \geq n$, by letting $n$ be greater than $M$ given in Lemma \ref{distance implies support}, we see that $\Sigma$ supports a geodesic in $h_i$, which we define to be $\gamma_i$.
Then the part (1) holds automatically, and moreover by passing to a subsequence, we can assume either all the $\gamma_i$ are finite geodesics or all of them are geodesic rays.

By \cref{distance implies support}, passing to a subsequence again,  we can assume that $\gamma_i$ contains a simplex $s_i$ with $d_{\CC(\Sigma)}(c_i, s_i) \leq P$.
Since $\{c_i\}$ tends to the ending lamination of $e$ as $i \rightarrow \infty$, if the $\gamma_i$ are finite geodesics then $\{\gamma_i\}$ converges to a geodesic ray which tends to the ending lamination of $e$, which implies that the last vertex of $\gamma_i$ converges to the ending lamination of $e$.
Suppose that passing to a subsequence all the $\gamma_i$ are geodesic rays.
Then again since $\gamma_i$ contains $s_i$ tending to the ending lamination, by the hyperbolicity of $\CC(\Sigma)$ the endpoint of $\gamma_i$ converges to the ending lamination of $e$.
In both cases, since $\gamma_i$ contains $s_1$ for large $i$ with $d_{\CC(\Sigma)}(c_1, s_1) \leq P$ by \cref{distance implies support}, we see that it contains a simplex with its distance from the basepoint determined by $\mu$ bounded as $i \rightarrow \infty$.
\end{proof}
In the proof of Proposition \ref{simply degenerate brick}, we used the assumption that the end in $B$ is algebraic only to show that the support of $\gamma_i$ is $\Sigma$.
Even in the case when the end in $B$ may not be algebraic, the argument above shows that we still have a geodesic as $\gamma_i$ in $h_i$, and its support is the preimage of $\Sigma$, which may depend on $i$.
Thus we get the following corollary.

\begin{corollary}
\label{non-algebraic brick}
Let $B = \Sigma \times J$ be a simply degenerate brick of $\mathbf M$.
Let $\mathcal V$ be the union of all boundary components of $\mathbf M$ touching the vertical boundary of $B$, and $\mathcal V_i$ the union of Margulis tubes corresponding to $(\rho_i^{\mathbf M})^{-1}(\mathcal V \cap B_{K_i r_i}(\mathbf M, x_\infty))$.
Then, there is a geodesic $\gamma_i$ in $h_i$ satisfying the following conditions.
\begin{enumerate}
\item
For sufficiently large $i$, the preimage $(\rho_i^{\mathbf M})^{-1}(B)$ is contained in a brick $B_i=\Sigma_i \times J_i$ in the standard brick decomposition of $\mathbf M_i \setminus \mathcal V_i$.
\item
The geodesic $\gamma_i$ is supported on $\Sigma_i$.
Passing to a subsequence, we can assume that all $\gamma_i$ either have finite lengths or are geodesic rays.
\item
If the $\gamma_i$ have finite lengths, their lengths go to $\infty$ as $i \rightarrow \infty$.
\item
Let $\parreal B$ be the real front of $B$.
Let $k_i: \Sigma \rightarrow \Sigma_i$ be a homeomorphism induced from  $(\rho_i^{\mathbf M})^{-1}|\parreal B$.
If $\gamma_i$ has finite length, for the last vertex $v_i$ of $\gamma_i$, its image $k_i^{-1}(v_i)$ on $\Sigma$ converges to the ending lamination of the simply degenerate end of $B$.
If $\gamma_i$ is a ray, then for the endpoint $e_i$ of $\gamma_i$ at infinity, $k_i^{-1}(e_i)$ converges to the ending lamination of the end in $B$ in $\EL(\Sigma)$.
\end{enumerate}
\end{corollary}
\section{Limits of end invariants and ends of models}
In this section, we consider the situation where $\{(G_i, \phi_i)=qf(m_i,n_i)\}$ converges to $(\Gamma, \psi)$ in $AH(S)$ and $\{G_i\}$ converges geometrically to $G_\infty$.
We assume that $\{m_i\}$ converges to $[\mu^-]$  and $\{n_i\}$ converges to $[\mu^+]$ in the Thurston compactification of the Teichm\"{u}ller space.
Let $\Sigma^-$ and $\Sigma^+$ be the boundary components of the convex core of $M_i=\hyperbolic^3/G_i$ facing the upper and lower ends respectively.
We shall first recall the following fact, which follows from the continuity of length function.

 \begin{lemma}
\label{mu is ending}
Let $\nu$ be a component of either $\mu^-$ or $\mu^+$.
If $\nu$ is a weighted simple closed curve, then $\Psi(|\nu|)$ represents a parabolic class of $\Gamma$.
Otherwise, its image $\Psi(\nu)$ represents the ending lamination for an end of $M'_0$.
\end{lemma}
\begin{proof}
This is just a combination of Thurston's theorem and the continuity of the length function proved by Brock \cite{Br} in general form.
We can assume that $\nu$ is a component of $\mu^-$ since the argument for the case of $\mu^+$ is exactly the same.
Thurston's Theorem 2.2  in \cite{Th2} (whose proof can be found in \cite{FLP} and  \cite{Ths}) shows that there is a sequence of simple closed curve $r_i \gamma_i$ converging to $\mu^+$ such that $\len_{m_i}(r_i \gamma_i)$ goes to $0$.
By Bers' inequality \cite{Be}, this implies that $\len_{\Sigma_i^-}(r_i \gamma_i)$ also goes to $0$.
By the continuity of the length function with respect to the algebraic topology (see Brock \cite{Br}), we have $\len_{M'} \Psi(\mu^-)=0$, which means every component of $\Psi(\mu^-)$  represents either a parabolic class or an ending lamination.
\end{proof}

We shall refine this lemma in Corollary \ref{Thurston limit} to show that in $\mathbf M$ the components of  the limit of $\{m_i\}$ appear as lower algebraic parabolic curves or  lower algebraic ending laminations,  whereas those of $\{n_i\}$ appear as upper algebraic parabolic curves or upper algebraic ending laminations.
(See Definition \ref{upper/lower} for the definitions of these terms.)
%
\begin{othertheorem}
\label{limit laminations}
Let $c_{m_i}$ and $c_{n_i}$ constitute shortest (hyperbolic) pants decompositions of $(S,m_i)$ and $(S,n_i)$ respectively.
Let $\nu^-$ and $\nu^+$ be  the Hausdorff limits of $\{c_{m_i}\}$ and $\{c_{n_i}\}$ respectively.
Then the minimal components of $\nu^+$ that are not simple closed curves coincide with the upper algebraic ending laminations of $\mathbf M$.
Moreover, every upper algebraic parabolic curve of $\mathbf M$, regarded as a curve on $S$,  is contained in  $\nu^+$.
Similarly the minimal components of  $\nu^-$ that are not simple closed curves coincide with the lower algebraic ending laminations of $\mathbf M$, and every lower algebraic parabolic curve is contained in  $\nu^-$.
\end{othertheorem}

\begin{proof}
%
%
We shall only deal with $\nu^+$.
The argument for $\nu^-$ is obtained only by turning $\mathbf M$ upside down.
Let $h_i$ be a hierarchy corresponding to $qf(m_i,n_i)$, and consider the model manifold $\mathbf M_i$ such that $\mathbf M_i[k]$ converges geometrically to $\mathbf M[k]$ as before.
We regard $\mathbf M$ as being embedded in $S \times [0,1]$ as usual.

We shall first show that any upper algebraic ending lamination of $\mathbf M$ is a minimal component of $\nu^+$.
Let $B=\Sigma \times [s,t)$ be a algebraic simply degenerate brick of $\mathbf M$ containing an end $e$.
By Proposition \ref{simply degenerate brick}, the hierarchy $h_i$ contains a geodesic $\gamma_i$ supported on $\Sigma$ whose last vertex converges to the ending lamination $\lambda_e$ of $e$ in $\UML(\Sigma)$.
Now, as was shown in \S 6 in Masur-Minsky \cite{MaMi} using Theorem 3.1 in the same paper, the distance between the last vertex of $\gamma_i$ and the projection of the terminal marking $T(h_i)$ of $h_i$ to $\Sigma$ is uniformly bounded.
In particular, for the shortest pants decomposition $c_{n_i}$, which consists of the base curves of $T(h_i)$, its projection to $\Sigma$ converges to $\lambda_e$ in $\UML(\Sigma)$.
Since the Hausdorff limit of $c_{n_i}|\Sigma$ contains the limit of the projection of $c_{n_i}$ in $\UML(\Sigma)$, this shows that any upper algebraic ending lamination  is contained in $\nu^+$.

Secondly, we shall show that every upper algebraic parabolic curve is contained in $\nu^+$.
Let $c$ be an upper algebraic parabolic curve on $S$, and denote a standard algebraic immersion by $g' : S \to \mathbf M$.
There are three cases which we have to consider.
The first is the case (a) when $g'(c)$ is homotopic to a curve on a torus boundary of $\mathbf M$; the second is the case (b) when $g'(c)$ is homotopic to a core curve of  an open annulus boundary component of $\mathbf M$ at least one of whose ends touches a geometrically finite end; and the third is the case (c) when $g'(c)$ is homotopic to a core curve of an open annulus boundary component whose ends touch only simply degenerate or wild ends.

(a) We first consider the case when the curve $g'(c)$ is homotopic into a torus component $T$ of $\partial \mathbf M$.
For later use, we state here the result of  the case (a) as a claim, taking into account also the case of $\nu^-$.
\begin{claim}
\label{case a}
Let $c$ be a simple closed curve on $S$ such that $g'(c)$ is homotopic into a torus boundary component of $\mathbf M$.
Then $c$ is a minimal component of $\nu^+$ if $c$ is an upper parabolic curve and is a minimal component of $\nu_-$ if $c$ is a lower parabolic curve.
\end{claim}

\begin{proof}
Let $V_i$ be the Margulis tube bounded by $(\rho_i^{\mathbf M})^{-1}(T)$ in $\mathbf M_i$ for large $i$.
Its boundary $\partial V_i$ has a marked flat structure which is parametrised by $\omega_{\mathbf M_i}(V_i)$ as was explained in the proof of Lemma \ref{model limit}.
The real part of $\omega_{\mathbf M_i}(V_i)$ corresponds to the difference of the marking on the top horizontal annulus and that of the bottom horizonal annulus, hence to the length of the tight geodesic of $h_i$ supported on an annulus on $S$ homotopic to the vertical projection of  the horizontal annulus of $V_i$.
Since $\partial V_i$ converges to the boundary of a torus cusp neighbourhood,  $\Re \omega_{\mathbf M_i}(V_i)$ goes to $\infty$ whereas the imaginary part is bounded as $i \rightarrow \infty$.
Let $c_i$ be a simple closed curve on $S$ whose image by $\Phi^{\mathbf M}_i$ is homotopic to the longitude of $V_i$.
Since the longitude of $\partial V_i$ converges to that of $T$ which is homotopic to the image of a simple closed curve under $g'$, the homotopy class of $c_i$ is independent of $i$ for large $i$.
Therefore, by taking a subsequence, we can assume that $c_i$ is constantly $c$.
Let $A$ be an annulus on $S$ which is a regular neighbourhood of $c$.
Since $\Re \omega_{\mathbf M_i}(V_i) \rightarrow \infty$, as was explained above, there is a geodesic $\gamma_i$ in $h_i$ supported on $A$ whose length goes to $\infty$ as $i \rightarrow \infty$.
Let $a(i)$ and $b(i)$ be the first and  last vertices of $\gamma_i$, and let $n(a)_i$ and $n(b)_i$ be the (signed) numbers of times  $a(i)$ and $b(i)$ respectively go around $c$ compared to the transversal of the marking determined by $\Phi^{\mathbf M}_i$.
Then $|n(b)_i-n(a)_i|$ goes to $\infty$ as $i \rightarrow \infty$.
We set $n(i)$ to be $n(b)_i-n(a)_i$.

By the definition of hierarchy, there is a vertex  of a geodesic $g_i$  in $h_i$  with $\xi(D(g_i))=4$ which represents $c$ (and is denoted also by $c$), satisfying $\pi_A(\pre(c))=a(i)$ and $\pi_A(\suc(c))=b(i)$.
As was shown above,  the distance between $\pi_A(\pre(c))$ and $\pi_A(\suc(c))$ goes to $\infty$.
Since these $\pre(c)$ and $\suc(c)$ may depend on $i$, we denote $\pre(c)$ in $g_i$ by $v_i$ and $\suc(c)$ in $g_i$ by $w_i$.


Since there is an elementary move  changing $v_i$ to $c$, there is a block $b_i$ in $\mathbf M_i$ realising this elementary move by the definition of model manifold by Minsky \cite{Mi}.
Let $U_i$ be the Margulis tube in $\mathbf M_i$ whose core curve represents $v_i$, and $u_i$ a  horizontal longitude on $\partial U_i$.
Recall that the block decomposition of $\mathbf M_i$ converges geometrically to that of $\mathbf M$ as $i \rightarrow \infty$.
Therefore, the block $b_i$ can be pushed forward to a block $b_\infty$ in $\mathbf M$ for large $i$, and hence there is either a Margulis tube or a torus  boundary in $\mathbf M$ whose core curve or longitude, which we denote by $u_\infty$, is homotopic to $\rho_i^{\mathbf M}(u_i)$ for every large $i$.
First suppose that $g'$ does not go around $T$.
Then, since $c$ is upper parabolic,  $g'$ can be homotoped so that it passes   $b_\infty$ horizontally, and consequently, there is a simple closed curve $u$ on $S$ such that $g'(u)$ is homotopic to $u_\infty$.
By pulling back this to $\mathbf M_i$, we see $\Phi^{\mathbf M}_i(u)$ is homotopic to $\pre(c)$ for large $i$.
This means that $n(a)_i$ is bounded as $i \rightarrow \infty$, and hence $|n(b)_i|$ goes to $\infty$.

We next consider the  case when $g'$ goes $k$-times around $T$ for a positive integer $k$.
(Since $c$ is an upper parabolic locus, $g'$ cannot go around $T$ in the negative direction.)
In this case, to homotope $(\rho^{\mathbf M}_i)^{-1} \circ g'$ to an immersion $g''_i$ which does not go around $V_i$, we have to make it pass $k$ times through $V_i$ in the upward direction.
This $g''_i$ converges geometrically to a surface which passes $b_\infty$ horizontally, and therefore, in the same argument as the previous paragraph,
 there is a simple closed curve $u$ on $S$ as above such that $g''_i(u)$ is homotopic to $u_i$.
Recall, as we explained in the proof of \cref{model limit}, the meridian $\beta_i$ of $V_i$ realises a homotopy between the markings below and above, and hence as the surface passes $V_i$ in the positive direction, the transverse of $c$ is twisted $n(i)$-times.
This implies that the $kn(i)$-time Dehn twist of $u_i$ around $c$ represents the constant homotopy class $u$ for large $i$.
Therefore, $|n(a)_i|$ grows in the order of $|kn(i)|$, and $|n(b)_i|$ does in the order of $|(k+1)n(i)|$ in this case.

In either case, we see that $|n(b)_i|$ goes to $\infty$.
Therefore, by \S6 of \cite{MaMi} again, the projection of the shortest pants decomposition $c_{n_i}$ to $\CC(A)$ also goes around $n'(i)$ times around $c$ with $n'(i) \rightarrow \infty$.
This shows that the Hausdorff limit $\nu^+$ of $\{c_{n_i}\}$ contains $c$ as a minimal component.
\end{proof}

(b) Next we consider the second case when $g'(c)$ is homotopic to a core curve of an open annulus boundary component $T$ which touches a geometrically finite end $e$ of $\mathbf M$.
Then we shall prove the following claim, which will also be used later.
\begin{claim}
\label{homotopic into gf}
Let $\Sigma$ be an incompressible subsurface of $S$ such that $\Sigma \times \{1\}$ corresponds to a geometrically finite end of $\mathbf M$ with conformal structure $n_\Sigma$.
Let $a$ be a simple closed curve on $S$ such that $g'(a)$ is homotopic to a curve $a'$ on $\Sigma \times \{1\}$.
Then the hyperbolic length $\length_{n_i}(a)$ converges to that of $a'$ with respect to $n_\Sigma$.
In particular, if $a'$ is peripheral, then $\length_{n_i}(a)$ goes to $0$.
\end{claim}
\begin{proof}
By the definition of conformal structures on the geometrically finite bricks of $\mathbf M$ in  \S \ref{label}, there is a subsurface $\Sigma_i$ in $S$ such that $(\Sigma_i,n_i)$ converges to $(\Sigma, n_\Sigma)$ geometrically.
Since $g'$ is an algebraic locus, for a simple closed curve $a'$ of $\Sigma$ homotopic to $g'(a)$ is pulled back to a curve  on $\Sigma_i$ homotopic to  $\Phi^{\mathbf M}_i(a)$.
This implies the hyperbolic length of $a$ with respect to $n_i$ converges to that with respect to $n_\Sigma$.
\end{proof}

This implies that the length of $c$ with respect to $n_i$ goes to $0$, and the pants decomposition $c_{n_i}$ must contain $c$.
Therefore $c$ is contained in the Hausdorff limit of $\{c_{n_i}\}$ also in this case.

(c) Now, we consider the third case when $g'(c)$ is homotopic to a core curve of an open annulus boundary component $T$ whose ends touch only  simply degenerate or wild ends.
Suppose, seeking a contradiction, that there is a minimal component $d$ of $\nu^+$ intersecting $c$ transversely.
By Claim \ref{case a}, we see that $d$ cannot intersect an upper algebraic parabolic curve on $S$ whose image by $g'$ is homotopic into a torus boundary component.
We shall first show the following claim.
\begin{claim}
\label{obstructing subsurface}
There is either a simply degenerate end or a horizontal annulus on a torus boundary component, corresponding to $\Sigma \times \{t\}$  lying above $g'(S)$ and an incompressible subsurface $F$ of $S$ containing $d$ such that $g'(F \cap \Sigma)$ is vertically parallel into $\Sigma \times \{t-\epsilon\}$ for any small $\epsilon >0$, and $d$ intersects $\Sigma$ essentially on $F$.
(Here $F \cap \Sigma$ is assumed to have no inessential intersection.)
\end{claim}
\begin{proof}
Suppose first that there is an upper algebraic simply degenerate end corresponding to $\Sigma \times \{t\}$ such that $\Sigma$ intersects $d$ essentially.
(This is equivalent to saying that $d$ intersects the minimal supporting surface of an upper ending lamination essentially.)
Then, by letting $F$ be the entire $S$, the condition above holds.

Suppose next that $d$ can be homotoped so as to be disjoint from any minimal supporting surface of the upper ending lamination (\ie from the surface  $\Sigma$ for any upper algebraic simply degenerate ends as described above).
Let $F$ be a component of the complement in $S$ of the union of the minimal supporting surfaces of the upper algebraic ending laminations and annular neighbourhoods of parabolic curves of the case (a) above (\ie those homotopic into a torus boundary component), which contains $d$.
Since $c$, as well as $d$,  cannot intersect the minimal supporting surface of an upper algebraic ending lamination and a parabolic curve of (a) above, $c$ is also contained in $F$.
Consider  a simply degenerate end or a lower horizontal annulus of a torus boundary corresponding to $\Sigma' \times \{t\}$ above $g'(S)$ such that $\Sigma'$ intersects $d$ essentially.
We can see that such an end or a torus exists since the boundary component into which $g'(c)$ is homotopic touches either a simply degenerate end or a wild end and there is no essential half-open annulus tending to a wild end.
We take a lowest one among such $\Sigma' \times \{t\}$,  and denote it by $\Sigma \times \{t\}$.
Then $(\Sigma \cap F) \times \{t-\epsilon\}$ is homotopic to $g'(\Sigma \cap F)$ in $\mathbf M$ for any small positive $\epsilon$.
(See Figure \ref{fig:lowest}.)
Thus we get  subsurfaces $F$ and $\Sigma$ as we desired.
\end{proof}

Suppose now that  $\Sigma \times \{t\}$ corresponds to a simply degenerate end, and denote it by $e$.
Let $\lambda$ be the ending lamination of the end $e$, and $B$ the brick of $\mathbf M$ containing $e$.
By Corollary \ref{non-algebraic brick}, there are bricks $B_i\cong \Sigma_i \times J_i$  containing $(\rho_i^{\mathbf M})^{-1}(B)$  and geodesics $\gamma_i$  supported on $\Sigma_i$  whose lengths go to $\infty$ as $i \rightarrow \infty$.
The approximate isometry induces a homeomorphism $f_i: \Sigma \rightarrow \Sigma_i$.
 Also, we know that for the last vertex $t_i$ of $\gamma_i$, its image $(f_i)^{-1}(t_i)$ converge to the ending lamination $\lambda$.
By the same argument as before using \S 6 of \cite{MaMi},  $\{(f_i)^{-1}\pi_{\Sigma_i}(c_{m_i})\}$  converges to a geodesic lamination containing $\lambda$.
We state this as a claim for later use.
\begin{claim}
\label{simply degenerate case}
Suppose that $\Sigma \times \{t\}$ corresponds to a simply degenerate end.
Then there is a homeomorphism $f_i : \Sigma \to \Sigma_i$ induced from an approximate isometry between $\mathbf M$ and $\mathbf M_i$, and 
the sequence of simple closed curves $\{(f_i)^{-1}\pi_{\Sigma_i}(c_{m_i})\}$  converges to a geodesic lamination containing $\lambda$.
\end{claim}

Since $g'(F\cap \Sigma)$ is vertically homotopic into $\Sigma \times \{t-\epsilon\}$ in $\mathbf M$ for any small positive $\epsilon$, we see that $f_i|(F \cap \Sigma)$ can be isotoped so that it is  the identity map if restricted to $F \cap \Sigma$.
Since $\lambda$ is arational in $\Sigma$, and $\Sigma$ intersects $d$ essentially,  we see that $\lambda$ intersects $d$ essentially.
On the other hand, by \cref{simply degenerate case}, $\{f_i^{-1}(c_{m_i})\}$ converges to a geodesic lamination containing $\lambda$ whereas $\{c_{m_i}\}$ converges to $\nu$ containing $d$, both in the Hausdorff topology.
Considering the fact that $f_i|F\cap \Sigma$ can be assumed to be the identity, we see that this is impossible.

Next suppose that $\Sigma \times \{t\}$ is a horizontal annulus lying  on a torus boundary component $T$.
Let $\Sigma_i$ be the lower horizonal annulus on $\rho_i^{-1}(T)$.
Then  by the same argument as in the proof of \cref{case a}, we see that   $\pi_{\Sigma_i}(c_{n_i})$ spirals around the core curve of $\Sigma_i$ more and more as $i \rightarrow \infty$.
Since $(\Sigma \cap F) \times \{t-\epsilon\}$ is homotopic to $g'(\Sigma \cap F)$ and $\Sigma \cap F$ intersects $d$ essentially, this shows that the Hausdorff limit of $\{c_{n_i}\}$ intersects $d$ transversely.
This is a contradiction.
Thus we have shown that an upper algebraic parabolic curve cannot intersect a minimal component of $\nu^+$ transversely.
Since $\nu^+$ is a Hausdorff limit of pants decomposition, every closed geodesic either intersects $\nu^+$ transversely or is contained in $\nu^+$. 
Therefore implies that any upper algebraic curve is contained in $\nu^+$.

\begin{figure}
\scalebox{0.6}{\includegraphics{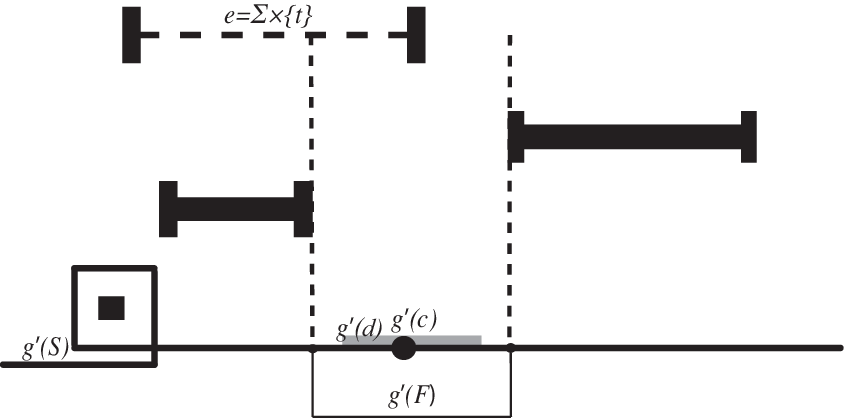}}
\caption{The definition of $F$ and the lowest end}
\label{fig:lowest}

\end{figure}

\medskip
Thus we have shown that both the upper algebraic ending laminations and the upper algebraic parabolic curves are contained in $\nu^+$.
To complete the proof, it remains to show that every minimal component of $\nu^+$ that is not a simple closed curve is an upper algebraic ending lamination.
For that we have only to show that $\nu^+$ has no minimal component that is not a compact leaf and is disjoint from the minimal supporting surfaces of the upper algebraic ending laminations up to isotopy.
We recall that no minimal component of $\nu^+$ intersects an upper parabolic locus regarded as lying on $S$ by the results of the cases (a)-(c) above.
 Let  $\Sigma_1, \dots , \Sigma_{j_0}$ be the minimal supporting surfaces of the upper algebraic ending laminations.
  Let $F$ be a component of  the complement of the union of $\cup_{j=1}^{j_0} \Sigma_j$ and all upper parabolic loci.
 What we have to show is that every minimal component $d$ of $\nu^+$ contained in $F$ is a simple closed curve.
 
The argument is quite similar to the proof of case (c).
Before dealing with the general situation, we begin with considering the special case when $g'(F)$ is homotopic into an end lying above $g'(S)$.
The end cannot be wild since a wild end has no essential open annulus tending to the end.
Hence the end is either geometrically finite or simply degenerate.
Then there is an incompressible subsurface $\Sigma$ of $S$ containing $F$ such that  $g'(F)$ is homotopic into an end corresponding to $\Sigma \times \{s\}$ (\ie into $\Sigma \times \{s-\epsilon\}$ for any small positive $\epsilon$), where $s=1$ if the end is geometrically finite.
We first consider the case when $s=1$ and $S' \times \{1\}$ is geometrically finite.
As shown in \cref{homotopic into gf},  the surface $\Sigma \times \{1\}$ has a hyperbolic metric $n_\infty$ which is a geometric limit of $(S, n_i)$ with base point lying in the thick part of $(\Sigma_i, n_i|\Sigma_i)$ for some subsurface $\Sigma_i$ homeomorphic to $\Sigma$.
Since $F \times \{1-\epsilon\}$ in $\mathbf M$ is homotopic to $g'(F)$, we see that $\Sigma_i$ contains $F$ up to isotopy for large $i$ and that $n_i$ induces a hyperbolic structure $n_i|F$ on $F$ with geodesic boundary, which converges to $n_\infty|F$ preserving the markings.
Note that  $\nu^+|F$ is a Hausdorff limit of $c_{n_i}|F$ and $d$ is contained in it.
Since $n_i|F$ converges to $n_\infty|F$, we see that  any minimal component  of the Hausdorff limit $\nu^+$ of $\{c_{n_i}\}$  contained in $F$ must be a compact leaf.

Next suppose that there is a  simply degenerate end $e$ of $\mathbf M$ of the form $\Sigma \times \{s\}$ with ending lamination $\lambda$, lying above $g'(S)$ into which $g'(F)$ is homotopic.
Then this end cannot be algebraic since $F$ lies in the complement of the minimal supporting surfaces of upper algebraic ending laminations up to isotopy.
This implies that $F$ is a proper subsurface of $\Sigma$ up to isotopy.
By \cref{simply degenerate case}, there is a homeomorphism $f_i : \Sigma \to \Sigma_i$ for a subsurface $\Sigma_i$ of $S$, and $\{f_i^{-1}(\pi_{S'_i}(c_{n_i}))\}$ converges to  a geodesic lamination containing $\lambda$.
Now, since $g'(F)$ is homotopic into $\Sigma \times \{s\}$ and $\lambda$ is arational in $\Sigma$, this shows that the Hausdorff limit of $\{c_{n_i}|F\}$ consists only of arcs.
Therefore, $F$ cannot contain $d$, contradicting our assumption.
In general, as shown in \cref{obstructing subsurface}, there is a simply degenerate end or a horizontal annulus on a torus boundary component, corresponding to $\Sigma \times \{t\}$ situated above $g'(S)$, such that $g'(F\cap \Sigma)$ is homotopic into $\Sigma \times \{t-\epsilon\}$ in $\mathbf M$ for any small positive $\epsilon$, and such that $F \cap \Sigma$ intersects $d$ essentially.
(See Figure \ref{fig:lowest} again.)
Suppose first that $\Sigma \times \{t\}$ corresponds to a simply degenerate end $e$.
Then as in \cref{simply degenerate case}, the Hausdorff limit of $c_{n_i}|(F \cap \Sigma)$ must contain the restriction of the ending lamination $\lambda$ of $e$ to $F$.
Since $F \cap \Sigma$ intersects $d$ essentially, $\lambda$ is arational in $\Sigma$, and $\Sigma \times \{t\}$ is not algebraic, this shows that the Hausdorff limit of $c_{n_i}|(F \cap \Sigma)$ intersects $d$ transversely.
This is a contradiction.

Next suppose that $\Sigma \times \{t\}$ lies on a torus boundary component $T$.
Then each component of $\Sigma \cap F$ is a strip.
As in the proof of  \cref{case a}, for each component $\Delta$ of $\Sigma \cap F$, the Hausdorff limit of $c_{n_i}|F$ has a leaf running to join the two components of $\Fr_\Sigma \Delta$.
Since $F \cap \Sigma$ intersects $d$ essentially on $F$, each component of $d \cap \Sigma \cap F$ joint two boundary components of $\Sigma$.
This implies that the Hausdorff limit of $c_{n_i}$ intersects $d$ transversely, which is a contradiction again.
Thus we have shown that every minimal component of the Hausdorff limit of $\{c_{n_i}\}$ contained in $F$ is a compact leaf.
This completes the proof of Theorem \ref{limit laminations}.
\end{proof}

Theorem \ref{main limit laminations} in \S\ref{main results} is obtained as a corollary of Theorem \ref{limit laminations}, as follows.

\begin{proof}[Proof of Theorem \ref{main limit laminations}]
Each simply degenerate end of $\mathbf M$ is mapped to that of $(M_\infty)_0$.
Let $p: M' \rightarrow M_\infty$ be a covering associated to the inclusion of the algebraic limit $\Gamma$ into the geometric limit $G_\infty$.
By the covering theorem of Thurston \cite{Th} and Canary \cite{CaT} with the argument in Lemma 2.3 of \cite{OhQ}, each simply degenerate end of $(M')_0$ has a neighbourhood which is mapped homeomorphically to a neighbourhood of a simply degenerate end of $(M_\infty)_0$.
Furthermore the ending lamination of an end of $(M')_0$ is identified with that of the corresponding end of $(M_\infty)_0$ by $p$, which follows immediately by the definition of ending laminations.
Therefore the algebraic simply degenerate ends of $\mathbf M$ correspond to simply degenerate ends of $(M')_0$ one-to-one preserving the ending laminations.
It is also obvious that upper (resp. lower) ends of $\mathbf M$ correspond to upper (resp. lower) ends of $(M')_0$.
Similarly the algebraic parabolic curves correspond to the core curves of the parabolic loci of $M'$.
Therefore, Theorem \ref{limit laminations} implies the statement of \cref{main limit laminations} except for the last paragraph.

Suppose that $c$ is a simple closed curve in $\nu^-$ or $\nu^+$ that has isolated leaves around it  is a parabolic curve.
We assume that $c$ is contained in $\nu_-$ for simplicity.
Then there is a component $d_i$ of $c_{m_i}$ spiralling around $c$ more and more as $i \to \infty$.
Since the length of $d_i$ is bounded, by the continuity of length function this shows that the geodesic length of $\psi(c)$ is $0$, and hence $c$ is a parabolic curve.
Suppose further that $c$ is not contained in $\nu_-$.
Then by \cref{limit lamination}, $c$ cannot be a lower parabolic curve of $\mathbf M$, which means in particular that $c$ is an upper parabolic curve of $\mathbf M$.
Furthermore this implies that even if $g'(c)$ is homotopic into a torus boundary component $T$ of $\mathbf M$, the algebraic locus $g'(S)$ cannot go around $T$.
Therefore, $c$ is an upper parabolic curve of $M'$. 
\end{proof}

\begin{lemma}
\label{limit of pants}
Let $\{g_i\}$ be a sequence in the Teichm\"{u}ller space $\mathcal T(S)$ which converges to a projective lamination $[\mu]$ in the Thurston compactification.
Let $c_i$ be a  pants decomposition on $(S,g_i)$ whose total length is uniformly bounded independently of $i$.
Then the Hausdorff limit of any subsequence of  $\{c_i\}$ contains all the components of $|\mu|$  as minimal components.
\end{lemma}

\begin{proof}
Let $\lambda$ be the Hausdorff limit of a convergent subsequence of $\{c_i\}$.
Since $c_i$ is pants decomposition of $S$, every measured lamination  on $S$ except for the components of $c_i$ intersects $c_i$ essentially.
Therefore every measured lamination on $S$ other than those contained in  $\lambda$ intersects $\lambda$ essentially.
Let $\mu_0$ be a component of $\mu$, and suppose that its support is not a minimal component of $\lambda$.
Then $\mu_0$ must intersect $\lambda$ transversely.

We shall first consider the case when $\mu_0$ is not a simple closed curve.
Let $\Sigma$ be the minimal supporting surface of $\mu_0$.
We consider a sequence of essential arcs and simple closed curves $c_i \cap \Sigma$ on $\Sigma$.
Note that $c_i \cap \Sigma$ converges to $\lambda \cap \Sigma$ with respect to the Hausdorff topology, which is non-empty.
If $\lambda \cap \Sigma$ has a minimal component contained in $\Int \Sigma$, then there is a sequence of positive numbers $r_i$ going to $0$ such that $r_i c_i \cap \Sigma$ converges  to a measured lamination $\gamma$ in $\Sigma$.
Otherwise we can find a bounded sequence of positive numbers $r_i$ such that $\{r_i c_i\}$ converges to a non-empty union $\gamma$ of essential arcs.
(The limit is taken in the space of weighted essential curves in $\Sigma$ with the weak topology of transverse measures.)
In either case, let $R$ be $\max_{i} r_i$.
Now, $\length_{g_i}(c_i) \geq r_i\length_{g_i}(c_i)/R$, where the right hand goes to $\infty$ since  $i(\mu_0, \gamma)>0$ and an arc with non-zero intersection with the Thurston limit has length going to $\infty$ if we consider hyperbolic structures on $\Sigma$ with geodesic boundaries.
This implies that $\length_{g_i}(c_i)$ must also go to $\infty$, which is a contradiction.

Next suppose that $\mu_0$ is a simple closed curve.
If the length of $\mu_0$ with respect to $g_i$ goes to $0$, then we can take an annular neighbourhood $A_i(\mu_0)$ of $\mu_0$ whose width (with respect to $n_i$) goes to $\infty$ as $i \rightarrow \infty$.
Since $\lambda$ intersects $\mu_0$ essentially, $c_i$ passes through $A_i(\mu_0)$ for large $i$.
This implies the length of $c_i$ in $(S,g_i)$ goes to $\infty$, which is a contradiction.
Next suppose that the length of $\mu_0$ is bounded from both above and below by positive constants.
Then we can take an annular neighbourhood $A_i(\mu_0)$ whose width is bounded away from $0$.
Consider the shortest essential arc $\alpha_i$ in $A_i(\mu_0)$.
Since $\mu_0$ is contained in the limit lamination $[\mu]$ of $\{g_i\}$, the shortest arc $\alpha_i$ must spirals around $\mu_0$ more and more as $i \rightarrow \infty$.
Since $\lambda$ does not contain $\mu_0$, the number of spiralling of $c_i$ around $\mu_0$ is bounded.
This means twisting number between $\alpha_i$ and $c_i|A_i(\mu_0)$  goes to $\infty$.
Therefore, the length of $c_i|A_i(\mu_0)$ goes to $\infty$ as $i \rightarrow \infty$, which is a contradiction.
In the case when the length of $\mu_0$ goes to $\infty$, we take $A_i(\mu_0)$ whose width goes to $0$ as $i \rightarrow \infty$.
Also in this case, the shortest essential arc $\alpha_i$ spirals around $\mu_0$ more and more as $i \rightarrow \infty$.
Since the twisting number between $\alpha_i$ and $c_i|A_i(\mu_0)$ goes to $\infty$ also in this case, we see that the length of $c_i|A_i(\mu_0)$ goes to $\infty$.
This is a contradiction.
%
%
\end{proof}

Combining this lemma with Theorem \ref{limit laminations}, we get the following corollary.

\begin{corollary}
\label{Thurston limit}
In the setting of Theorem \ref{main limit laminations}, let $[\mu^-]$ and $[\mu^+]$ be   projective laminations to which $\{m_i\}$ and $\{n_i\}$ converge  in the Thurston compactification of $\mathcal T(S)$ after taking subsequences.
Then each minimal component of $|\mu^+|$ is either an upper algebraic ending lamination  or an upper algebraic parabolic curve of $\mathbf M$.
Similarly, each component of $|\mu^-|$ is either  the ending lamination of a lower algebraic simply degenerate end or a lower algebraic parabolic curve of $\mathbf M$.
\end{corollary}
\begin{proof}
As usual, we shall only deal with $|\mu^+|$.
%
Each component of $|\mu^+|$ that is not a simple closed curve is the ending lamination of an upper simply degenerate end by Theorem \ref{limit laminations} and Lemma \ref{limit of pants}.
Let $c$ be a component of $|\mu^+|$ which  is a simple closed curve.
Then by Lemma \ref{mu is ending}, $\psi(c)$ is parabolic.
Therefore $c$ is an algebraic parabolic locus in $\mathbf M$.
It remains to show that $c$ is upper.

Suppose that $c$ is not upper, seeking  a contradiction.
This assumption implies, in particular, that if $c$ is isolated and is homotopic to a curve lying on a torus boundary $T$ of $\mathbf M$, then the standard algebraic immersion $g'$ does not go around $T$ by our definition of the upperness.
Since there is no essential half-open annulus tending to a wild end, there are only three possibilities for the curve $c$: (1) the first  is when $c$ lies  in a domain $F$ of $S$ as a non-peripheral curve and $g'(F)$ is homotopic into some simply degenerate end above $g'(S)$, (2) the second is when there exists $F$ containing $c$ as above such that $g'(F)$ is homotopic into geometrically finite end, lying on $S \times \{1\}$, and (3) the third is when there are a domain $F$ as above and either a simply degenerate end or a horizontal annulus on a torus boundary component above $g'(S)$, corresponding to $\Sigma \times \{t\}$, such that $c$ intersects $\Sigma$ essentially and $g'(F \cap \Sigma)$ can be homotoped into $\Sigma \times \{t-\epsilon\}$ for arbitrarily small $\epsilon >0$.

(1) Suppose that $g'(F)$ is homotopic into a simply degenerate end $e$ corresponding to $\Sigma \times \{t\}$.
Then its  ending lamination $\lambda$ intersects $c$ essentially.
As in  \cref{simply degenerate case},  there are bricks $B_i \cong \Sigma_i \times J_i$ and a homeomorphism $f_i : \Sigma \rightarrow \Sigma_i$ induced from the approximate isometry between $\mathbf M$ and $\mathbf M_i$, such that $\{f_i^{-1}(c_{n_i}|\Sigma_i)\}$ converges to a geodesic lamination containing $\lambda$ in the Hausdorff topology.
Let $A(c)$ be an annulus with core curve $c$.
Since $g'(F)$ is homotopic into $e$, we see that $f_i|F$ is isotopic to the identity.
In particular, we can assume that $f_i|A(c)$ is the identity.
Therefore $c_{n_i}|A(c)$ regarded as a vertex of $\CC(A(c))$ converges a vertex represented by $\lambda|A(c)$.
On the other hand, by Lemma \ref{limit of pants}, $\{c_{n_i}\}$ must converge to a lamination containing $c$ in the Hausdorff topology; hence $c_{n_i}|A(c)$ diverges in $\CC(A(c))$.
This is a contradiction. 

%

(2) In the second case, $g'(F)$ is homotopic into an upper geometrically finite end of $\mathbf M$.
Take a simple closed curve $\delta$ on $F$ intersecting $c$ essentially.
Since $g'(\delta)$ is homotopic to a curve in an upper geometrically finite end, by \cref{homotopic into gf}, $\length_{m_i}(\delta)$ is bounded as $i \rightarrow \infty$.
This contradicts the assumption that $c$ is contained in $\mu^+$.

(3) Now we turn to the third case.
Suppose first that $\Sigma \times \{t\}$ is a simply degenerate end $e$ in a brick $B\cong \Sigma \times J$ with ending lamination $\lambda$.
Then, as in the case (1), by \cref{simply degenerate case}, there is a homeomorphism $f_i : \Sigma \to \Sigma_i$ induced from the approximate isometry between $\mathbf M$ and $\mathbf M_i$ such that the Hausdorff limit of $f_i^{-1}(c_{n_i}|\Sigma_i)$ contains $\lambda$.
Since $F \cap \Sigma$ is homotopic into $\Sigma \times \{t-\epsilon\}$, and $f_i$ fixes $F \cap \Sigma$, we see that the Hausdorff limit of $f_i^{-1}(c_{n_i}|\Sigma_i)$ contains $c \cap \Sigma$.
This is a contradiction since $c$ intersects $\lambda$ transversely.
The same kind of argument works also for the case when $\Sigma \times \{t\}$ is a horizontal annulus on a torus boundary using \cref{case a} instead of \cref{simply degenerate case}.
\end{proof}

\section{Proofs of Theorem \ref{+-side --side}, Theorem \ref{main} and Theorem \ref{scc case}}
We can now prove Theorem \ref{+-side --side}, Theorem \ref{main} and Theorem \ref{scc case} making use of our results obtained in the previous section.

\subsection{Proof of Theorem \ref{+-side --side}}
We consider the geometric limit $M_\infty$ of $M_i$ and the model $\mathbf M$ of its non-cuspidal part as before.
By the definition of algebraic simply degenerate ends of $\mathbf M$, they are mapped by the model map $f$ to simply degenerate ends of $(M_\infty)_0$ which lift to those of the algebraic limit $(M')_0$.
Upper ends among them are mapped to those lifted to upper ends of $(M')_0$, and lower ones to those lifted to lower ends of $(M')_0$.
Now, by Corollary \ref{Thurston limit}, every component of $|\mu^+|$ that is not a simple closed curve is an  upper algebraic ending lamination.
Therefore, it is the ending lamination of a upper simply degenerate end of $(M')_0$.
The same argument works for $|\mu^-|$.

The second paragraph of the statement also follows immediately from Corollary \ref{Thurston limit}.
\subsection{Proof of Theorem \ref{main}}
Suppose, seeking a contradiction, that $\{qf(m_i, n_i)\}$ as in the statement of Theorem \ref{main} converges after taking a subsequence.
Then by Theorem \ref{+-side --side}, $\mu_0^+$ is an ending lamination of an upper end of $(M')_0$ and $\mu_0^-$ is that of a lower end of $(M')_0$.
Let $\Sigma^+$ and  $\Sigma^-$ be the minimal supporting surfaces of $\mu_0^+$ and $\mu_0^-$ respectively, which were assumed to share at least one boundary component $c$.
Since $c$ lies on the boundary of both $\Sigma^-$ and $\Sigma^+$, it represents a $\integers$-cusp both above and below $\Psi(S)$.
This is impossible since no two distinct cusps have homotopic core curves.

\subsection{Proof of Theorem \ref{scc case}}
As in the statement, let $\mu^-$ and $\mu^+$ be two measured laminations on $S$ and suppose that there is a boundary component $c$ of a non-simple closed curve component of the minimal supporting surface of a non-simple closed curve component $\mu_0$ of $\mu^+$ which is  contained up to isotopy in either  $|\mu^-|$ 
as in the part (1) of the statement.
We can argue in the same way also when the part (2) of the statement holds just by exchanging $+$ and $-$.
%

Suppose, seeking a contradiction, that $\{qf(m_i,n_i)\}$ converges passing to a subsequence.
By Corollary \ref{Thurston limit}, there is an upper algebraic simply degenerate end of $\mathbf M$ in the form $\Sigma(\mu_0) \times \{t\}$ with $\Sigma(\mu_0)$ the minimal supporting surface of $\mu_0$, which has $|\mu_0|$ as the ending lamination.
This implies that there is a boundary component of $\mathbf M$ which is an open annulus with core curve homotopic to $c$,  and one of whose ends tends to this simply degenerate end.
By Lemmata \ref{isolated} and \ref{ends of M'}, this shows that $c$ is an upper algebraic parabolic curve of $\mathbf M$ and moreover that  it cannot be a lower algebraic parabolic curve at the same time since $c$ does not lie on a torus boundary component.
In the same way, by \cref{Thurston limit}, if $c$ is contained in $|\mu_1|$ or lies on the boundary of minimal supporting of $\mu_1$ up to isotopy, it must be a lower algebraic parabolic curve of $\mathbf M$.
This is a contradiction.
Thus we have completed  proof of Theorem \ref{scc case}.

\section{Proof of Theorem \ref{generalised Ito}}
\subsection{Necessity}
We shall first show that the condition (1) is necessary.
Suppose, on the contrary, that there is $c_j$ among $c_1, \dots , c_r$ such that $\len_{n_i}(c_j)$ goes to $0$ whereas $\{(G_i,\phi_i)=qf(m_i,n_i)\}$ converges.
(The argument for the case when $\len_{m_i}(c_j)$ goes to $0$ is quite the same.)
Let $G_\infty$ be the geometric limit of a subsequence of $\{G_i\}$ as before, and set $M_i=\hyperbolic^3/G_i$ and $M_\infty=\hyperbolic^3/G_\infty$.
Consider the model manifold $\mathbf M$ of $(M_\infty)_0$.
By Corollary \ref{Thurston limit}, $c_j$ is an upper algebraic parabolic curve of $\mathbf M$.
Let $g': S \rightarrow \mathbf M$ be a standard algebraic immersion.
Since $\len_{n_i}(c_j) \rightarrow 0$, the boundary blocks of $\mathbf M_i$ corresponding to the upper boundary are pinched along an annulus with core curve $c_j$, and the one in the geometric limit is split along the annulus.
By pushing forward this core curve and $\Phi_i(c_j)$ to $M_\infty$ and pull it back to $\mathbf M$, we get a homotopy between a core curve of an open annulus component of $\partial \mathbf M$ and  $g'(c_j)$.
By Lemma \ref{position of S}, this shows that $g'(S)$ cannot go around a torus boundary component whose longitude corresponds to $c_j$.
In particular, $c_j$ cannot be a lower algebraic parabolic curve of $\mathbf M$.
This contradicts, by way of Corollary \ref{Thurston limit}, the fact that $|\mu_-|$ also contains $c_j$.
This completes the proof of the necessity of the condition (1).

Next we turn to showing the necessity of the condition (2).
By Corollary \ref{Thurston limit} again, we see that if $\{(G_i,\phi_i)\}$ converges algebraically, each of $c_1, \dots , c_r$ must be both upper and lower algebraic parabolic curves of $\mathbf M$.
By Lemma \ref{position of S}, this is possible only when $g'(S)$ goes around a torus boundary component $T_j$ of $\mathbf M$ whose longitude is homotopic to $g'(c_j)$ for each $j=1, \dots , r$.
Suppose that $g'(S)$ goes $a_j$ times around $T_j$ for $a_j \in \integers \setminus \{0\}$, where we define the counter-clockwise rotation in Figure \ref{fig:go-around} (when viewed from the right to the left) to be the positive direction.
As before, we define the condition $a_j=0$ means that $g'(S)$ passes below $T_j$, and
if $g'(S)$ passes above $T_j$ not going around it, we define $a_j$ to be $-1$.
Since $g'(S)$ goes around $T_j$, we have $a_j \neq 0, -1$.

Let $\mathbf M_i$ be a model manifold of $(M_i)_0$.
Since $\mathbf M_i[0]$ converges to $\mathbf M[0]$ geometrically, there is a torus boundary  $T_j(i)$  of $\mathbf M_i[0]$ which is mapped to $T_j$ by the approximate isometry $\rho_i^{\mathbf M}$.
The torus $T_j(i)$ consists of two horizontal annuli and two vertical annuli.
We choose a meridian-longitude system of $T_j$ in such a way that the longitude $l_j$ lies on a horizontal annulus and the meridian $m_j$ is shortest among all the simple closed curves on $T_j$ intersecting the longitude at one point.
We choose orientations on $l_i$ and $m_i$ so that the three-dimensional orientation  determined by the frame formed by $l_i, m_i$ and the normal vector of $T_j$ pointing inward coincides with that of $\mathbf M$.
By pulling back this system using the approximate isometry $\rho_i^{\mathbf M}$ between  $\mathbf M_i[0]$ and $\mathbf M[0]$, we get an oriented longitude $l_j(i)$ and an oriented meridian $m_j(i)$ on $T_j(i)$.
There is a Margulis tube $V_j(i)$ attached to $T_j(i)$ in $\mathbf M_i$.
The compressing curve of $V_i(j)$ intersects the longitude $l_j(i)$ only at one point.
Therefore we can express the homology class of the compressing curve as $k_i^j[l_j(i)]+[m_j(i)]$ if we choose an orientation on the compressing curve.
Since $T_j(i)$ converges geometrically to a torus boundary component of $\mathbf M$, we have $|k_i^j| \rightarrow \infty$.


Fix some $j$,  and consider $c_j$.
In $\mathbf M$, there is a block $B^+$ intersecting $T_j$ by an annulus $A^+$ containing the upper horizontal  annulus of $T_j$ in the middle.
Similarly, there is a block $B^-$ intersecting $T_j$ by an annulus $A^-$ containing the lower horizontal annulus of $T_j$ in the middle.
One or both of these may be  boundary blocks.
We shall only consider the case when $a_j >0$, \ie the case when $g'(S)$ goes around $T_j$ counter-clockwise in Figure \ref{fig:go-around} if it proceeds from left to right.
Since $c_j$ is an algebraic parabolic curve, the standard immersion can be homotoped to pass through both $B^-$ and $B^+$.

Now, consider a simple closed curve $\gamma^-$ on the {\em lower} horizontal annulus of $B^-$ which is homotopic to  a core curve  if $B^-$ is an internal block.
When $B^-$ is a boundary block, we consider horizontal upper boundary components of $B^-$ adjacent to $A^-$.
If there are two such surfaces, we denote their union by $\Delta^-$, and if there is only one such surface, we denote it by $\Delta^-$.
We take a simple closed curve $\gamma^-$ which lies in $A^- \cup \Delta^-$ and intersects the core curve of $A^-$ at two points when $\Delta^-$ consists of two components and at one point when $\Delta^-$ is connected.
We define $\gamma^+$ in the same way.
By pulling back $\gamma^+$ and $\gamma^-$ by $(\rho_i^{\mathbf M})^{-1}$, we get  simple closed curves $\gamma^+(i)$ and $\gamma^-(i)$, which are horizontal except for vertical parts passing through vertical annuli contained in $A^-$ or $A^+$.
Using the vertical projection to $S$ in $\mathbf M_i$, we regard $\gamma^+(i)$ and $\gamma^-(i)$ also as curves on $S$.

Let $g'_i: S \rightarrow \mathbf M_i$ be a pull-back of the standard immersion $g'$ obtained by composing $(\rho_i^{\mathbf M})^{-1}$.
We consider to homotope $g'_i$ to unwrap it around $T_j(i)$ and make the surface lie under $T_j(i)$, by making it pass  $a_j$ times through $V_j(i)$.
Let $g''_i$ be a surface obtained by modifying the part of $g'_i$ going around $T_j(i)$ to a horizontal annulus and giving a natural marking coming from the structure of $S \times I$, which is equal to a marking determined by a pull-back of a horizontal surface in $\mathbf M$ obtained by removing the parts of $g'$ going around torus boundaries.
Note that $g''_i$ and $g'_i$ are not homotopic as maps because of the difference of markings.
(This means that $g''_i$ is not homotopic to $\Phi^{\mathbf M}_i$.)

Recall that the homology class of compressing curves for $V_j(i)$ is expressed as  $k_i^j[l_j(i)]+[m_j(i)]$.
We fix an orientation on the compressing disc whose normal vector points towards the orientation given on $l_j(i)$.
Recall that the torus $T_j(i)$ consists of four annuli; two horizontal annuli, the lower annulus and the upper annulus, which are expressed as $A_j(i) \times \{u\}$ and $A_j(i) \times \{v\}$ with respect to the inclusion of $\mathbf M_i$ into $S \times [0,1]$, and two vertical annuli $\partial A_j(i) \times [u,v]$.
Consider an essential simple arc $\alpha: [0,1] \to A_j(i)$  such that its image in the lower annulus $\alpha_u(t)=(\alpha(t), u)$ is a part of the meridian: $m_j(i) \cap (A_j(i) \times \{u\})$.
Then by considering the homotopy class of a compressing curve as above, we see that $\alpha_u$ is homotopic (relative to the endpoints) to an arc obtained by joining a vertical arc expressed as $\alpha(0) \times [u,v]$ oriented upward, a horizontal arc $\alpha_v(t)=(\alpha(t),v)$, another horizontal arc representing $-k_i^j[l_j(i)] \times \{v\}$, and another vertical arc expressed as $\alpha(1) \times [u,v]$ oriented downward.
Therefore, each time $g'_i(S)$ passes through $V_j(i)$ in the positive direction, a curve $\delta$ on $S$ intersecting $c_j$ is twisted $-k_i^j$-times around $c_j$.
Here  the positive direction is the direction to which a horizontal surface below $V_j(i)$ passes to one above $V_j(i)$.
We fix such a transverse orientation for immersions of $S$.
If $g'_i(S)$ passes through $V_j(i)$ in the negative direction $\delta$ is twisted $k_i^j$-times around $c_j$.
(It goes $|k_i^j|$-times around $c_i$ in the  direction of the right hand Dehn twist if $k_i^j>0$ and in the opposite  direction if $k_i^j <0$.)
Since $\gamma^-(i)$ is homotopic to $g''_i(\gamma^-)$, we see that $g'_i(\tau^{k_i^ja_j}_{c_j}(\gamma^-))$ is homotopic to $\gamma^-(i)$.
Similarly, since $\gamma^+(i)$ is homotopic to $g''_i(\tau^{k_i^j}(\gamma^+))$, we see that $g'_i(\tau^{k_i^j(a_j+1)}_{c_j}(\gamma^+))$ is homotopic to $\gamma^+(i)$.


Let $h_i$ be a hierarchy of tight geodesics for $\mathbf M_i$ as before.
Since $V_j(i)$ appears as a Margulis tube in $\mathbf M_i$, we see that an annular neighbourhood $A(c_j)$ of $c_j$ supports a geodesic $h_i(c_j)$ in $h_i$.
By the definition of gluing blocks in Minsky \cite{Mi}, we see that $\pi_{A(c_j)}(\gamma^-(i))$ is the initial marking and $\pi_{A(c_j)}(\gamma^+(i))$ is the terminal marking of $h_i(c_j)$ if both $B^-$ and $B^+$ are internal blocks.
By \S 6 of Masur-Minsky \cite{MaMi}, we see that $\pi_{A(c_j)}(I(h_i))$ is in a uniformly  bounded distance from $\pi_{A(c_j)}(\gamma^-(i))$ and $\pi_{A(c_j)}(T(h_i))$ is in a  uniformly bounded distance from $\pi_{A(c_j)}(\gamma^+(i))$.
Even when $B^-$ or $B^+$ is a boundary block, we have the same properties: for, since the length  of $\gamma^-(i)$ with respect to $m_i$ or that of $\gamma^+(i)$ with respect to $n_i$ is bounded, its projection to $A(c_j)$ is within uniformly bounded distance from those of $I(h_i)$ or $T(h_i)$.
Since $g'_i$ is homotopic to $\Phi_i$ for large $i$ and $\mathbf M_i$ is identified with a subset of $S \times (0,1)$ using $\Phi_i$ as a marking, we see that $\pi_{A(c_j)}(I(h_i))$ is within a uniformly bounded distance from $\pi_{A(c_j)}(\tau^{k_i^ja_j}_{c_j}(\gamma^-))$, whereas $\pi_{A(c_j)}(T(h_i))$ is within a bounded distance from $\pi_{A(c_j)}(\tau^{k_i^j(a_j+1)}_{c_j}(\gamma^+))$.
If we consider the pulled-back metrics $(\tau^{k_i^ja_j}_{c_j})^* m_i$ and $(\tau^{k_i^j(a_j+1)}_{c_j})^* n_i$ instead of $m_i$ and $n_i$, the initial and the terminal markings are twisted around $c_i$ by $-k_ia_j$ times and $-k_i^j(a_j+1)$ times respectively.
Therefore for the quasi-Fuchsian representation $qf((\tau^{k_i^ja_j}_{c_j})^* m_i, (\tau^{k_i^j(a_j+1)}_{c_j})^* n_i)$, the tight geodesic supported on $c_j$ has length bounded as $i \to \infty$.
Therefore,  the shortest pants decompositions  of $(S, (\tau^{k_i^ja_j}_{c_j})^* m_i)$ and $(S, (\tau^{k_i^j(a_j+1)}_{c_j})^* n_i)$ have Hausdorff limits which do not spiral around $c_j$.
Since the lengths of $c_j$ with respect to $(\tau^{k_i^ja_j}_{c_j})^* m_i$  and $(\tau^{k_i^j(a_j+1)}_{c_j})^* n_i$ do not go to $0$, by the proof of \cref{limit of pants}, this implies that the limits of $(\tau^{k_i^ja_j}_{c_j})^* m_i$ and $(\tau^{k_i^j(a_j+1)}_{c_j})^* n_i$ in the Thurston compactification of $\mathcal T(S)$ do not contain $c_j$ as a leaf.
We repeat the same argument for every $c_j$, and let $p_i^j$ and $q_i^j$ be $k_i^ja_j$ and $k_i^j(a_j+1)$ respectively.
This completes the proof of the necessity.

If $a_j > 0$, \ie  $g'$ goes around the torus containing $c_j$ in the counter-clockwise in Figure \ref{fig:go-around} as it proceeds from left to right, then the torus lies in the positive direction viewed from $g'(S)$.
Therefore $f(c_j)$ is lifted to a curve lying above the core surface obtained the lift of  $f \circ g'$ in $M'$.
This shows that $c_j$ is a core curve of an upper parabolic locus if $a_j > 0$.
We can argue in the same way also when $a_j <0$.

\subsection{Existence}
We shall next show that the existence of limits of quasi-Fuchsian groups satisfying the conditions (1) and (2).
Our construction just follows the argument of Anderson-Canary \cite{AC}.

We first construct a geometrically finite Kleinian group $\Gamma_0$ such that $N=\hyperbolic^3/\Gamma_0$ is homeomorphic to the complement of $c_j \times \{1/2\}\, (j=1, \dots, r)$ in $S \times (0,1)$, and the conformal structures corresponding to the ends $S \times \{0\}$ and $S \times \{1\}$ are the same point $m_0 \in \mathcal T(S)$.
(Here we identify $S$ with $S \times \{0\}$ and $S \times \{1\}$ by the natural inclusions.)
We consider an immersion $g_0: S \rightarrow N_0$ which is in the standard form in the sense of Lemma \ref{position of S}, and wraps $a_j$ times around each $c_j \times \{1/2\}$ counted counter-clockwise as it proceeds from left to right in Figure \ref{fig:go-around} when we identify $N_0$ with its embedding in $S \times [0,1]$.

Next, we consider a quasi-conformal deformation of $\Gamma_0$.
Let $\mu_1^-, \dots , \mu_s^-$ and $\mu_1^+, \dots \mu^+_t$ be the components of $\mu^-$ and $\mu^+$ that are not  shared simple closed curves.
We consider the minimal supporting surfaces $\Sigma(\mu_1^-), \dots, \Sigma(\mu_s^-)$ of $\mu_1^-, \dots , \mu_s^-$ and isotope them so that if two boundary components are isotopic, they coincide.
We move the supporting surfaces of $\mu_1^+, \dots \mu^+_t$ in the same way.
Let $d_1^-, \dots, d_\sigma^-$ be the simple closed curve components of $|\mu^-|$ which are not shared by $|\mu^+|$, and $d_1^+, \dots, d_\tau^+$ those of $|\mu^+|$ not shared by $|\mu^-|$.
We let $e_1^-, \dots, e_q^-$ be the frontier components of $\cup_{j=1}^s \Sigma(\mu_j^-)$ which do not appear in $d_1^-, \dots, d_\sigma^-$, and in the same way, we let $e_1^+, \dots , e_r^+$ be the frontier components of $\cup_{j=1}^t \Sigma(\mu_j^+)$ which do not appear in $d_1^+, \dots , d_\tau^+$.
We define a Kleinian group $\Gamma_k$ for $k \in \naturals$ to be the one obtained by quasi-conformally deforming the conformal structures $m_0$ on $S \times \{0\}$ by the earthquake with respect to $k(\cup_{j=1}^s \mu^-_j)$ and pinching along $d_1^-, \dots, d_\sigma^-; e_1^-, \dots , e_q^-$ so that the $\epsilon_0$-thin part around each of $d_1^-, \dots , d_\sigma^-$ has height $k$ whereas that around each of $e_1^-, \dots , e_q^-$ has height $\sqrt{k}$; and $m_0$ on $S \times \{1\}$ by the earthquake with respect to $k(\cup_{j=1}^t \mu_j^+)$ and pinching along $d_1^+, \dots, d_\tau^+; e_1^+, \dots , e_r^+$ in the same way.
%
The pinching is performed so that the conformal structures on $S \setminus ( (\cup_{j=1}^s \Sigma(\mu_j^-)) \cup d_1^- \cup \dots \cup d_\sigma^- )$  and  $S \setminus ( (\cup_{j=1}^t \Sigma(\mu_j^+)) \cup d_1^+\cup \dots \cup d_\tau^+ )$ do not change.
Let $N_k$ be $\hyperbolic^3/\Gamma_k$ and $h_k : N \rightarrow N_k$ a natural homeomorphism derived from the quasi-conformal deformation.
We regard $N_k$ also as embedded in $S \times [0,1]$ in such a way that the images of drilled out curves lie on $S \times \{1/2\}$, and the natural identification of this $S \times [0,1]$ with the one in which $N$ is embedded is compatible with $h_k$.
Then we get an immersion $g_k : S \rightarrow N_k$ which is defined to be the composition $h_k \circ g_0$.

In $S \times [0,1]$ where  $N$ is embedded, we regard $\mu^+$ as lying on $S \times \{1\}$ and $\mu^-$ as lying on $S \times \{0\}$.
Then  every essential annulus in $S \times [0,1]$ either intersects a torus cusp  or $\mu^+ \cup \mu^-$, or has  a boundary component contained  as a non-peripheral curve in a component of either $S\times \{0\} \setminus \mu^-$ or $S \times \{1\} \setminus \mu^+$, where the conformal structure is not deformed, by the conditions (1*), (2*), and (3*).
Therefore, by the main theorem of \cite{OhI}, we see that $N_k$ with marking determined by $h_k$ converges algebraically to a hyperbolic 3-manifold $N_\infty=\hyperbolic^3/\Gamma_\infty$ with a homeomorphism $h_\infty: N \rightarrow N_\infty$.
The laminations $\mu^-_1, \dots, \mu^-_s; \mu^+_1, \dots , \mu^+_t$ represent ending laminations of simply degenerate ends of $N_\infty$.
Since we pinched all frontier curves of $S \setminus (\cup_{j=1}^s \Sigma(\mu_j^-) \cup  d_1^- \cup \dots \cup d_\sigma^- )$ and of $S \setminus  (\cup_{j=1}^t \Sigma(\mu_j^+)  \cup d_1^+\cup \dots \cup d_\tau^+ )$, by Lemma 3 of \cite{Ab} each component $\Sigma^f$ of $S \setminus (\cup_{j=1}^s \Sigma(\mu_j^-) \cup  d_1^- \cup \dots \cup d_\sigma^- )$ or $S \setminus  (\cup_{j=1}^t \Sigma(\mu_j^+)  \cup d_1^+\cup \dots \cup d_\tau^+ )$ appears as a surface with punctures on the boundary at infinity of  $N_\infty$.
On the other hand, the same surface at infinity also appears in the geometric limit  since a subregion of a component of $\Omega_{G_i}$ invariant under $\pi_1(\Sigma^f)$ converges to the corresponding component of $\Omega_\Gamma$ in the sense of Carathéodory.
(See the proof of Proposition 4.2 of J{\o}rgensen-Marden \cite{JM}.))
This shows that each geometrically finite end of $(\hyperbolic^3/\Gamma_\infty)_0$ has a neighbourhood which is projected homeomorphically into the geometric limit.
 The covering theorem of Thurston and Canary (\cite{Th}, \cite{CaT}) and the argument of Lemma 2.3 \cite{OhQ} implies that every geometrically infinite end also has a neighbourhood descending homeomorphically to the geometric limit. 
Therefore the convergence is strong.
Let $g_\infty : S \rightarrow N_\infty$ be an immersion which is defined to be $h_\infty \circ g_0$.

Let $l_j$ and $m_j$ be respectively a longitude which lies in a tubular neighbourhood of $c_j \times \{1/2\}$ lying on a level surface along $c_j \times \{1/2\}$, and any meridian intersecting $l_j$ at one point.
Let $m_j(k)$ and $l_j(k)$  be a meridian and a longitude in $N_k$ obtained by pulling back $m_j$ and $l_j$ using approximate isometries.
We orient them so that the coordinate system $l_j-m_j$ determines an orientation of a torus around $c_j \times \{1/2\}$ whose normal vectors point to the inside of $N_\infty$.
Now, we consider a  Dehn filling of $N_k$ such that the compressing disc is attached along a curve represented by $k[l_j(k)]+[m_j(k)]$.
Since $N_k$ converges geometrically to $N_\infty$, we see, by passing to a subsequence, that the filling corresponding to $k[l_j(k)]+[m_j(k)]$ gives rise to a convex cocompact hyperbolic structure.
(A general version of hyperbolic Dehn surgery theorem, which applies to this case, was proved by Bromberg \cite{BromH}.)
We define $M_k$ to be thus obtained geometrically finite hyperbolic $3$-manifold, which is homeomorphic to $S \times (0,1)$.
We let $G_k$ be the corresponding quasi-Fuchsian group, and $\phi_k : \pi_1(S) \rightarrow \pi_1(M_k)$ an isomorphism derived from the pull-back of $g_k$ by an approximate isometry between $N_k$ and $M_k$.
By the same argument as in \cite{AC}, we see that the conformal structure  at infinity of $M_k$ on the end corresponding to $S \times \{0\}$, denoted by $m_k$,  is  obtained by performing the $-ka_j$-time (right-hand) Dehn twist around the $c_j$, the earthquake along $k(\mu_1^-\cup\dots \cup\mu_s^-)$ and pinching along $d_1^-, \dots , d_\sigma^-; e_1^-, \dots, e_q^-$ from $m_0$, and that on the end corresponding to $S \times \{1\}$, denoted by $n_k$, by performing the $-k(a_j+1)$-Dehn twist around the $c_j$, the earthquake along $k(\mu_1^+\cup  \dots \cup\mu_t^+)$ and pinching along $d_1^+, \dots , d_\tau^+; e_1^+, \dots , e_r^+$.
(Note that the deformation by $-ka_j$-time Dehn twist is the same as the pull-back by $ka_j$-time Dehn twist.)
By Kerckhoff's cosine formula (see Corollary 3.4 of \cite{Ke} and also \S2.2 of \cite{OhE}) we see the contribution of the part of $\mu_j^-$ or $\mu_j^+$ to the growth of the length for a curve $\gamma$ traversing it is asymptotically the same as $ki(\gamma , \mu_j^-)$ or $ki(\gamma, \mu_j^+)$ as $k \rightarrow \infty$.
Therefore the divergence transverse to $\mu_j^-$ (resp. $\mu_j^+$) has the same order as that of $d_1^-, \dots, d_\sigma^-$ (resp. $d_1^+, \dots , d_\tau^+$), whereas that transverse to $e_1^-, \dots , e_q^-$ (resp. $e_1^-, \dots , e_r^+$) has lower order.
This shows that the limits in the Thurston compactification of the conformal structures $m_k$ and $n_k$ are $[\mu^-]$ and $[\mu^+]$ respectively, and that those of $(\tau^{-ka_1}_{c_1} \circ \dots \circ \tau^{-ka_r}_{c_r})^*(m_k)$ and $(\tau^{-k(a_1+1)}_{c_1} \circ \dots \circ \tau^{-k(a_r+1)}_{c_r})^*(n_k)$ are $[\mu^-_1 \cup\dots \cup \mu^-_s \cup d_1^- \cup \dots \cup d_\sigma^-]$ and $[\mu^+_1 \cup \dots \cup \mu^+_t \cup d_1^+ \cup \dots \cup d_\tau^+]$ respectively.
In the case when one of these latter two projective laminations is empty, the corresponding conformal structures stay in a compact set of the Teichm\"{u}ller space. 

By the diagonal argument, we see that $\{(G_k,\phi_k)\}$ converges algebraically to a subgroup of $\Gamma_\infty$ corresponding to the covering of $N_\infty$ associated to $(g_\infty)_* \pi_1(S)$.
Thus we have obtained a sequence of quasi-Fuchsian groups $\{(G_k, \phi_k)\}$ as we wanted.

For the example which  we have constructed above, if either $\mu^-$ or $\mu^+$ consists only of $c_1, \dots , c_r$, then $\{(\tau^{-ka_1}_{c_1} \circ \dots \circ \tau^{-ka_r}_{c_r})^*(m_k)\}$ or $\{(\tau^{-k(a_1+1)}_{c_1} \circ \dots \circ \tau^{-k(a_r+1)}_{c_r})^*(n_k)\}$ stays in a compact set of the Teichm\"{u}ller space.
We can also make it converge to a projective lamination $\nu^-$ or $\nu^+$ not containing $c_1, \dots, c_r$ as leaves, by composing the earthquake along $\sqrt{k} \nu^-$ or $\sqrt{k} \nu^+$.

\section{Non-existence of exotic convergence}
We shall prove Theorem \ref{no exotic convergence} in this section.

Let $\Gamma$ be a b-group as in the statement and $\psi: \pi_1(S) \rightarrow \Gamma$ an isomorphism giving the marking.
Let $\{(G_i, \phi_i)\}$ be a sequence of quasi-Fuchsian groups converging to $(\Gamma, \psi)$.
What we need to show is that the conformal structures at infinity of the bottom ideal boundaries of the $M_i=\hyperbolic^3/G_i$ are bounded in the Teichm\"{u}ller space then.

Let $M_\infty$ be a geometric limit (of a subsequence) of $M_i$ with basepoint coming from a fixed basepoint in $\hyperbolic^3$ as usual.
Let $\mathbf M$ be a model manifold of $(M_\infty)_0$ with a model map $f_\infty : \mathbf M \rightarrow (M_\infty)_0$.
Let $g' : S \rightarrow \mathbf M$ be a standard algebraic immersion.
By our assumption, there is no isolated algebraic parabolic loci in $\mathbf M$.
This implies by Lemma \ref{isolated} that there is no torus boundary around which $g'$ can go.
If there is an algebraic simply degenerate end below $g'(S)$, it is mapped to a simply degenerate end of $(M_\infty)_0$ which is lifted to a lower simply degenerate end of $(\hyperbolic^3/\Gamma)_0$.
This is a contradiction since $\Gamma$ is a b-group.
Similarly, there is no lower algebraic parabolic locus.
Since $g'$ does not go around a torus boundary component, an end closest to $g'(S)$ among those below $g'(S)$ must be algebraic.
These imply that the only possible end below $g'(S)$ is a geometrically finite end corresponding to the entire $S \times \{0\}$.

The diameter of the manifold cobounded by $f_\infty \circ g'(S)$ and the lower boundary of the convex core in $(M_\infty)_0$ is finite since the manifold cobounded by $g'(S)$ and $S \times \{0\}$ in $\mathbf M$ has finite diameter.
This cobordism in $(M_\infty)_0$ can be pulled back to $(M_i)_0$ for large $i$.
Note that there is no closed geodesic in $(M_\infty)_0$ below the lower boundary of the convex core.
This implies that the distance from the lower boundary of the convex core of $M_i$ and the pull-back of the cobordism must go to $0$.
It follows that the lower boundary of the convex core of $M_i$ converges geometrically to that of $M_\infty$.
Since the cobordism above gives a marking on the lower boundary of the convex core homotopic to $g'$, the marked hyperbolic structure on the lower boundary component of  the convex core of  $M_i$ converges to that of $M_\infty$.
This shows that the lower conformal structure at infinity of $M_i$ is bounded in the Teichm\"{u}ller space as $i \rightarrow \infty$ by Sullivan's theorem (see Epstein-Marden \cite{EM}).

\section{Self-bumping}
\label{sec:self-bumping}
In this section, we shall prove Theorem \ref{bumping theorem} and Corollary \ref{no self-bumping}.
For that, we shall show that for $\{qf(m_i,n_i)\}$ as in the statement of Theorem \ref{bumping theorem}, there is a continuous deformation to a strong convergent sequence whose algebraic limit is a quasi-conformal deformation of $\Gamma$.
This will be done by using a model manifold of $qf(m_i,n_i)$ with a special property, which we shall construct in Lemma \ref{modified models}.
This model does not come from a hierarchy of tight geodesics $h_i$ as before, but its geometric limit can be better understood.
Let us state the existence of a deformation as a proposition.

\begin{proposition}
\label{deformation to strong one}
We consider quasi-Fuchsian groups $qf(m_i,n_i)$ and their algebraic limit $(\Gamma, \psi)$ as in Theorem \ref{bumping theorem}.
Let $c_1, \dots , c_s$ be upper parabolic curves and $c'_1, \dots , c_t'$  lower parabolic curves on $S$ of $(\Gamma, \psi)$.
Then, passing to a subsequence of $\{(m_i, n_i)\}$,  there is an arc $\alpha_i:[0,1] \rightarrow QF(S)$ with the following properties.
Let $(\bar m_i, \bar n_i)$ denote a point in $\mathcal T(S) \times \mathcal T(\bar S)$ such that $qf(\bar m_i, \bar n_i)=\alpha_i(1)$.
\begin{enumerate}
\item $\alpha_i(0)=qf(m_i,n_i)$.
\item $\{\alpha_i(1)=qf(\bar m_i, \bar n_i)\}$ converges strongly to a quasi-conformal deformation $(\Gamma', \psi')$ of $(\Gamma, \psi)$ as $i \rightarrow \infty$.
\item The length of each of $c_1, \dots , c_s$ with respect to $\bar n_i$ goes to $0$, and the length of each of $c'_1, \dots , c'_t$ with respect to $\bar m_i$ also goes to $0$ as $i \rightarrow \infty$.
\item In the case when $\Gamma$ is a b-group, the lower conformal structure at infinity of $\alpha_i(t)$ is constant with respect to $t$, for every $i$.
\item For any neighbourhood $U$ in $AH(S)$ of the quasi-conformal deformation space $QH(\Gamma, \psi)$, there exists $i_0$ such that for $i > i_0$, the arc $\alpha_i$ is contained in $U$.

\end{enumerate}

\end{proposition}

\subsection{Proof of Proposition \ref{deformation to strong one}}
Our basic strategy for the proof of this proposition is as follows.
We first consider to deform continuously the model manifold of the geometric limit of $\{qf(m_i,n_i)\}$ to that of the strong limit of $\{qf(\bar m_i, \bar n_i)\}$.
Corresponding to this, we get a continuous deformation of the model manifold of $qf(m_i,n_i)$ to that of $qf(\bar m_i, \bar n_i)$.
This will give rise to an arc $\alpha_i$ as desired.

Set $(G_i,\phi_i)=qf(m_i,n_i)$ and $M_i=\hyperbolic^3/G_i$.
We consider the geometric limit $M_\infty$ of $M_i$ with basepoints coming from some fixed basepoint in $\hyperbolic^3$.
Let $\mathbf M_i$ be a bi-Lipschitz model of $(M_i)_0$ with a model map $f_i$ and $\mathbf M$ that of $(M_\infty)_0$ as before.
Let $g': S \rightarrow \mathbf M$ be a standard algebraic immersion as in Lemma \ref{position of S}.
Let $E^1, \dots , E^p$ be the algebraic simply degenerate ends of $\mathbf M$.
Recall that $\mathbf M_i$ converges geometrically to the union of $\mathbf M$ and cusp neighbourhoods, and $\mathbf M_i$ corresponds to a hierarchy of tight geodesics $h_i$ determined by $G_i$.

%
We renumber $E^1, \dots , E^p$ so that $E^1, \dots , E^q$ are upper ends whereas $E^{q+1}, \dots , E^p$ are lower.
(It is possible that $q=0$ or $q=p$.)
We let $\Sigma^j$ be a subsurface of $S$ such that $E^j$ is contained in a brick $B^j$ of the form $\Sigma^j \times J^j$.
Since we assumed that $\Gamma$ has no isolated parabolic loci, each of $c_1, \dots , c_s$ is homotopic to a component of $\Fr \Sigma^j$ for $j=1, \dots, q$ and each of $c_1', \dots, c'_t$ is homotopic  to a component of $\Fr \Sigma^j$ for $j=q+1, \dots , p$.

We shall show that we can modify model manifolds $\mathbf M_i$ of $M_i$ to $\mathbf M_i'$ so that  in its geometric limit $\mathbf M'$ the ends $E^1, \dots , E^q$ lie on the same horizontal levels, and the same holds for  $E^{q+1}, \dots , E^p$.
This manifold $\mathbf M_i'$ will be constructed by removing product regions from $S \times (0,1)$ which should converge to neighbourhoods of these ends, and considering the standard brick decomposition.
We shall  construct a uniform Lipschitz map from $\mathbf M_i'$ to $M_i$ in \cref{Lipschitz} just by taking split level surfaces to pleated surfaces as in the construction of model manifolds by Minsky \cite{Mi}.
We need some more argument to show this map can be homotoped to a uniform bi-Lipschitz homeomorphism in \cref{bi-Lipschitz}.

In the following argument, we shall only describe the case when both $q>1$ and $p>q$; that is both upper algebraic simply degenerate ends and lower algebraic simply degenerate ends exist.
When one of these does not exist, we can modify the argument below easily just regarding it as an empty set.

%

\begin{lemma}
\label{modified models}
There are uniform bi-Lipschitz model manifolds $\mathbf M'_i$ for $(M_i)_0$ and $\mathbf M'$ for $(M_\infty)_0$, both of which are embedded in $S \times [0,1]$ preserving horizontal and vertical foliations, and have the following properties. 
\begin{enumerate}
\item We can choose basepoints in the thick parts so that $\mathbf M'_i[0]$ converges geometrically to $\mathbf M'[0]$ with these basepoints.
\item Under the same choice of the basepoints, $\mathbf M'_i$ converges geometrically to the union of $\mathbf M'$ and cusp neighbourhoods, which are geometric limits of Margulis tubes.
\item There is a homeomorphism from $\mathbf M$ to $\mathbf M'$ taking an algebraic locus of $\mathbf M$ to that of $\mathbf M'$.
(See \S \ref{algebraic limit} for the definition of algebraic locus.)
\item
We use the same symbol $E^j\, (j=1,\dots , p)$ to denote the end of $\mathbf M'$ corresponding to $E^j$ of $\mathbf M$.
Then, $E^j$  is contained in a brick $\bar B^j=\Sigma^j \times J^j$ of $\mathbf M'$ such that $\inf \bar B^1=\dots =\inf \bar B^q, \sup \bar B^1 =\dots =\sup \bar B^q, \sup \bar B^{q+1}=\dots =\sup \bar B^p,$ and $\inf \bar B^{q+1}=\dots =\inf \bar B^p$, with $\inf \bar B^1 > \sup \bar B^p$ unless $q=0$ or $p=q$.
\item An algebraic locus in $\mathbf M'$ can be homotoped to a horizontal surface lying between the horizontal levels of $\sup \bar B^p$ and $\inf \bar B^1$.
\end{enumerate}
\end{lemma}
\begin{proof}
For each component $c$ of $\Fr \Sigma^j$ for $\Sigma^j$ among $\Sigma^1, \dots, \Sigma^q$, which were defined above, we consider a solid torus $V(c)$ in $S \times (0,1)$ which has the form of $A(c) \times J_c$, where $A(c)$ is an annulus with core curve $c$ and $J_c$ is a closed interval in $[0,1]$ such that $\min J_c=\inf B^j$ and $\max J_c = \sup B^j$, where $B^j$ is a brick in $\mathbf M$ as described above.
Let $\bar {\mathcal T}^+$ be the union of  the $V(c)$ for all frontier components $c$ of $\Sigma^1, \dots , \Sigma^q$.
Even if a curve $c$ is homotopic to frontier components of two distinct $\Sigma^j$ and $\Sigma^{j'}$, we take only one solid torus.
In such a case, we take $J_c$ so that $\sup J_c = \min\{\sup B^j , \sup B^{j'}\}$, and $\inf J_c = \min\{\inf B^j, \inf B^{j'}\}$.
In the same way, we take $\bar{\mathcal T}^-$ for $\Sigma^{q+1}, \dots , \Sigma^p$.

Recall that there is an approximate isometry $\rho_i^{\mathbf M}$ between $\mathbf M_i[0]$ and $\mathbf M[0]$.
For each component $T$ of $\bar{\mathcal T}^+$ or $\bar{\mathcal T}^-$, the preimage of its boundary $(\rho_i^{\mathbf M})^{-1}(\partial T \cap \mathbf M[0])$ lies on a boundary component  of $\mathbf M_i[0]$, which bounds a solid torus in $\mathbf M_i$.
We denote the union of such solid tori for $\bar{\mathcal T}^+$ by $\bar{\mathcal T}^+_i$, and that for $\bar{\mathcal T}^-$ by $\bar{\mathcal T}^-_i$.
Although $(\rho_i^{\mathbf M})^{-1}$ is not defined on the entire $S \times (0,1)$, since $\rho_i^{\mathbf M}$ takes the algebraic locus $g'$ to an immersion homotopic to $\Phi_i$, it induces a homeomorphism $\varrho_i$ between $S\times (0,1) \setminus (\bar{\mathcal T}^+ \cup \bar{\mathcal T}^-)$ and $\mathbf M_i \setminus (\bar{\mathcal T}^+_i \cup \bar{\mathcal T}^-_i)$.
The model map $f_i$ takes $\bar{\mathcal T}_i^+$ to   Margulis tubes in $M_i$, which we denote by $V_1^i, \dots , V_s^i$ and $\bar{\mathcal T}^-_i$ to other Margulis tubes which we denote by ${V_1^i}', \dots , {V_t^i}'$, whose  core geodesics   $c_1^i, \dots , c_s^i; {c_1^i}', \dots , {c_t^i}'$ correspond to $\Phi_i(c_1), \dots, \Phi_i(c_s); \Phi_i(c_1') \dots , \Phi_i(c_t')$.
Note that the length of each of $c_1^i, \dots , c_s^i; {c_1^i}', \dots, {c_t^i}'$ goes to $0$ as $i \rightarrow \infty$.
 Since $f_i$ is a homeomorphism, $f_i \circ \varrho_i$ is also a homeomorphism between $S \times (0,1) \setminus (\bar{\mathcal T}^+ \cup \bar{\mathcal T}^-)$ and $M_i \setminus (V_1^i \cup \dots \cup V_s^i \cup {V_1^1}' \cup \dots \cup {V_t^i}')$.

Now for each  $V(c)=A(c) \times J_c$ in $\bar{\mathcal T}^+$, we let $V'(c)$ be $A(c) \times [5/8,3/4]$ in $S \times [0,1]$, and denote the union of such solid tori by $\mathcal T^+$.
Similarly, for each $V(c)$ in $\bar{\mathcal T}^-$, we let $V'(c)$ be $A(c) \times[1/4,3/8]$, and denote the union of such solid tori by $\mathcal T^-$.
Then, we see that there is a homeomorphism from $S \times (0,1) \setminus (\bar{\mathcal T}^- \cup \bar{\mathcal T}^+)$, in which $\mathbf M$ is embedded, to $S \times (0,1) \setminus (\mathcal T^- \cup \mathcal T^+)$ taking a standard algebraic immersion $g'$ to $S \times \{1/2\}$ since $\mathbf M$ has no torus boundary around which $g'$ can go (Lemma \ref{isolated}) and $\bar{\mathcal T}^-$ lies in the lower component of $S \times [0,1] \setminus g'(S)$ whereas $\bar{\mathcal T}^+$ lies in the upper component.

We let a new brick manifold $\bar{\mathbf  M}$ be the one obtained by the standard brick decomposition of   $S \times (0,1) \setminus (\mathcal T^- \cup \mathcal T^+)$.
Then $\bar {\mathbf M}$ consists of bricks lying on five levels: the top one $S \times [3/4,1)$, those touching $\mathcal T^+$ along vertical boundaries, the middle one $S \times [3/8,5/8]$, those touching $\mathcal T^-$ along vertical boundaries, and the bottom one $S \times (0,1/4]$.
For each $i$, we define a labelled brick manifold $\check{\mathbf M}_i$ to be the one obtained by 
giving the conformal structures $m_i$ to the bottom and $n_i$ to the top.

Corresponding to $\check{\mathbf M}_i$, we shall construct a geometrically finite hyperbolic manifold  from $M_i$ using the drilling theorem of Bromberg \cite{BromL} and Brock-Bromberg \cite{BB}.
First, recall that we have closed geodesics $c_1^i, \dots , c_s^i; {c_1^i}', \dots , {c_t^i}'$ in $M_i$, which are core curves of Margulis tubes.
The complement of the core curves $\check M_i =M_i \setminus (\cup_{j=1}^s c_j^i \cup \cup_{j=1}^t {c_j^i}')$, regarded as a 3-manifold topologically, admits a geometrically finite hyperbolic structure with conformal structures $m_i$ at the bottom and $n_i$ at the top, by Thurston's uniformisation theorem.
Since the lengths of $c_1^i, \dots , c_s^i; {c_1^i}', \dots , {c_t^i}'$ go to $0$ as $i \rightarrow \infty$, we can apply the drilling theorem to see that there are a constant $K$ independent of $i$ and a $K$-bi-Lipschitz homeomorphism $\check f_i$ between $M_i \setminus (\cup_{j=1}^s V_j^i \cup \cup_{j=1}^t {V_j^i}')$ and $(\check M_i)_0$.
Composing this with  the homeomorphism $f_i \circ \varrho_i$ from $S \times (0,1) \setminus (\bar{\mathcal T}^- \cup \bar{\mathcal T}^+)$ to $M_i \setminus (\cup_{j=1}^s V_j^i \cup \cup_{j=1}^t {V_j^i}')$ and the one from $\check{\mathbf M}_i$ to $S \times (0,1) \setminus (\bar{\mathcal T}^- \cup \bar{\mathcal T}^+)$ described above, we get a homeomorphism from $\check{\mathbf M}_i$ to $(\check M_i)_0$.

Now, we shall show that for sufficiently large $i$,  this homeomorphism is taken to be a $K'$-bi-Lipschitz homeomorphism for a constant $K'$ independent of $i$.
Any labelled brick manifold has a decomposition into blocks and geometrically finite bricks,  obtained by putting  unions of solid tori corresponding to tight geodesics, which we call {\em tube unions}, as we mentioned in \S\ref{label} and can be found in \S3 of \cite{OhSo}.
We now show how this decomposition is obtained in our specific situation.
We first take shortest pants decompositions of $(S, m_i)$ and $(S, n_i)$, and denote their decomposing curves by $e_1^-, \dots, e_{3g-3}^-$ and $e^+_1, \dots, e_{3g-3}^+$ respectively, where $g$ denotes the genus of $S$.
We put tubes $V^-_1, \dots , V^-_{3g-3}$, each of which is  bounded by two horizontal annuli and two vertical annuli, into the bottom brick corresponding to $S \times (0, 1/4]$ so that $\inf V^-_j=1/16, \sup V^-_j=3/16$ for every $j=1, \dots, 3g-3$ and the vertical projection of a core curve of $V^-_j$ is $e^-_j$.
In the same way we put $V^+_1, \dots, V^+_{3g-3}$ into the top brick corresponding to $S \times [3/4,1)$ so that $\inf V^+_j=13/16, \sup V^+_j=15/16$ and the vertical projection of a core curve of $V^+_j$ is $e^+_j$.

We remove the interiors of these tubes from $\check{\mathbf M}_i$ and regard the remaining manifold $\check{\mathbf M}_i^1$ as a labelled brick manifold by considering the standard brick decomposition and keeping the conformal structures at infinity $m_i$ and $n_i$.
The bricks corresponding to $S \times (0,1/16]$ and $S \times [15/16, 1)$ turn into geometrically finite bricks.
For each brick $B=\Sigma \times J$ of $\check{\mathbf M}_i^1$, where $\Sigma$ is either $S$ or one of $\Sigma^1, \dots , \Sigma^p$, such that $J$ is one of $[3/16,1/4], [1/4, 3/8], [3/8, 5/8], [5/8,3/4],$ and  $[3/4, 13/16]$,  we consider $\mathcal A^-=\partial_-B \cap \partial \check{\mathbf M}_i^1$ and $\mathcal A^+=\partial_+B \cap \partial \check{\mathbf M}_i^1$, both of which consist of disjoint horizontal annuli.
We consider a tight geodesic $g_B$ in $\CC(\Sigma)$ by setting $I(g_B)$ to be the vertical projections of  core curves of  $\mathcal A^-$, and $T(g_B)$ to be the vertical projections of core curves of $\mathcal A^+$.
We put a tube union into $B$ corresponding to $g_B$, \ie for each simplex $s$ of $g_B$, we put a union $V_s$ of solid tori bounded by two horizontal annuli and two vertical annuli, whose core curves are vertically homotopic to $s$, in such a way that $\inf V_{v_I}=\inf B$ for the initial vertex $v_I$, $\sup V_{v_T}=\sup B$ for the terminal vertex $V_T$, and $\inf V_s=\sup V_{\pre{s}}$ when $\xi(B) > 4$, whereas $\inf V_s > \sup V_{\pre{s}}$ if $\xi(S)=4$.
If two tubes contained in different bricks have  core curves  homotopic to each other without touching other tubes, we fuse them into one by putting a tube between the two.
We remove the interior of this tube union from $\check{\mathbf M}_i^1$ for every $B$ as above and get a brick manifold $\check{\mathbf M}_i^2$ by considering the standard brick decomposition.
We repeat the same construction for every brick of $\check{\mathbf M}_i^1$ that is neither a geometrically finite brick nor with $\xi(B)=3$, and remove the interiors of the tube unions to get a new brick manifold $\check{\mathbf M}_i^2$.
We stop this process when we reach the situation where every brick $B$ either is geometrically finite or $\xi(B)=3$ or $\xi(B)=4$ but $g_B$ has length $1$.
We denote the brick manifold at the final stage by $\check{\mathbf M}_i[0]$.
By cutting every brick of $\check{\mathbf M}_i[0]$ with $\xi=3$ into halves and paste them to  bricks above it and below it, we get a decomposition of $\check{\mathbf M}_i[0]$ into blocks in the sense of Minsky.


We decompose $\check{\mathbf M}_i[0]$ into blocks  in this way, and put a metric on $\check{\mathbf M}_i[0]$ defined by  them, \ie from the standard metric on each block.
Next we attach a Margulis tube to each boundary component of  $\check{\mathbf M}_i[0]$ except for those corresponding to $\mathcal T^-$ and $\mathcal T^+$ in the same way as was explained in the proof of Lemma \ref{model limit}.
For this model manifold, we have the following:
\begin{claim}
\label{Lipschitz}
There are a constant $K_0'$ depending only on $\xi(S)$ and a $K_0'$-Lipschitz map $\bar f_i' : \check{\mathbf M}_i\to (\check M_i)_0$ which takes every split level surface to a pleated surface realising the closed geodesics corresponding to  core curves of Margulis tubes which the split level surface touches.
\end{claim}
\begin{proof}
Recall that by construction, $\check{\mathbf M}_i$ consists of the following parts: two geometrically finite bricks, which are also boundary blocks, at the top and the bottom; the second bottom part, the second top part, and the middle part each of which is homeomorphic to $S \times I$; and the parts homeomorphic  to $\Sigma^j \times I$  for $j=1, \dots, p$.
The decomposition into blocks constructed above induces a hierarchy of tight geodesics in each of these parts, except for the two geometrically finite bricks.
In particular, by  Lemma 7.9 of Minsky \cite{Mi}, the closed geodesics in $M_i$ corresponding to core curves of Margulis tubes in $\check{\mathbf M}_i$ have lengths bounded by a constant $D$ depending only on $\xi(S)$, for $\Sigma^j$ is also a subsurface of $S$.
We then construct a map taking every split level surface to the corresponding pleated surface as in \cite{Mi}, and this map extends to  a Lipschitz map $\bar{f}_i': \check{\mathbf M}_i \rightarrow (\check M_i)_0$ by the technique of \lq interpolating pleated surfaces', whose Lipschitz constant depends only on $\xi(S)$ by the same argument as Minsky \cite{Mi}.
%
\end{proof}

Next we shall show that this map can be homotoped to a uniform bi-Lipschitz homeomorphism.
\begin{claim}
\label{bi-Lipschitz}
The Lipschitz map $\bar{f}_i'$ can be homotoped to a $K_0$-bi-Lipschitz homeomorphism $f'_i: \check{\mathbf M}_i \rightarrow (\check M_i)_0$, where $K_0$ depends only on $\xi(S)$.
\end{claim}
\begin{proof}
For Kleinian surface groups, it was shown in Brock-Canary-Minsky \cite{BCM} that the map is homotoped to a homeomorphism keeping the Lipschitz property, by rearranging the order of pleated surfaces to make it accord with the order of split level surfaces, and then homotoping each pleated surface to an embedding.
Some argument involving geometric limits was then used to show the resulting map is in fact uniformly bi-Lipschitz, which works without changes in our situation. 
In each of the first two processes, to guarantee the Lipschitz property of the map, it was necessary to show that we can choose a homotopy which does not pass through a Margulis tube around a very short closed geodesic.
To homotope a pleated surface to an embedding, the technique of Freedman-Hass-Scott \cite{FHS} was used, and we can use the same argument in our situation.
What we only need to check in our settings is that we can rearrange the order of pleated surfaces without passing through Margulis tubes around very short closed geodesics.

In each of the parts of $\check{\mathbf M}_i$ that we have described above except for the two geometrically finite blocks, the block decomposition of $\check{\mathbf M}_i[0]$ gives a sequence of split level surfaces $\{F_k\}$ in which $F_{k+1}$ is obtained from $F_k$ by changing two adjacent thrice-punctured spheres to another pair of thrice-punctured spheres   in such a way the step corresponds to an elementary move of slices of a hierarchy. 
In the construction of the Lipschitz map $\bar f_i'$ above, the homotopy between the pleated surface corresponding to $F_k$ and $F_{k+1}$ was chosen to move only such a four-times punctured sphere or an once-punctured torus so that intermediate surfaces have curvature less than $-1$ in the moving part.
The standard argument using the area of meridional discs in Margulis tubes implies that for sufficiently small $\epsilon>0$ depending only on $\xi(S)$, the image of $\bar f_i'$ can intersect any $\epsilon$-Margulis tube $V$ in $\check{M}_i$ only as the image of the corresponding Margulis tube in $\check{\mathbf M}_i$.

Now, we see how we can rearrange the order of pleated surfaces without passing through $\epsilon$-Margulis tubes.
Since the argument is the same for every part, we only consider the case of the second bottommost part corresponding to $S \times [3/16, 1/4]$.
The uppermost split level surface is a union of three-holed spheres corresponding to  pants decomposition of $S$ obtained by deleting  from $S \times \{1/4\}$ the tube unions which we put into $\check{\mathbf M}_i$, and each of the pairs of pants is mapped to a pleated surface which is a thrice-punctured sphere with each of its punctures lying on the axes of Margulis tubes or extends to a torus cusp.
We note that each of $c_1', \dots , c_t'$ is homotopic to a boundary component of one of these  pairs of pants.
We denote by $U(\Sigma)$ the pleated surface which is the union of (the closures of ) these thrice-punctured spheres.
Let $\bar V$ be a Margulis tube appearing in the part of $\check{\mathbf M}_i$ corresponding to $S \times [3/16, 1/4]$ below those which $S \times \{1/4\}$ passes through, and suppose that it is mapped  by $\bar f_i'$ to a Margulis tube $V$ with axis having length less than $\epsilon$ in $\check{M}_i$ which lies above $U(\Sigma)$, \ie such that there is no proper half-open arc starting from  $V$ and tending to $S \times \{0\}$ without algebraic intersection with $U(\Sigma)$. 
We can take a sub-solid torus $\bar V'$ of $\bar V$ which is mapped by $\bar f'_i$ onto the $\epsilon$-Margulis tube $V'$ in $V$.
Since each vertex can appear only once in a hierarchy and by our choice of $\epsilon$, we see that  $\bar f_i'((\check{M}_i \cap S \times (0,1/4]) \setminus \bar V')$ is disjoint from $V'$.
Therefore, $\bar f'_i|((\check{M}_i  \setminus (f'_i)^{-1}(\Int V'))$ is a proper  map to $(\check{M}_i)_0 \setminus \Int V'$.
Since $\bar V'$ lies below $S \times \{1/4\}$, its longitude can be homotoped towards the lower end corresponding to $S \times \{0\}$  without algebraic intersection with the split level surface lying on $S\times \{1/4\}$ nor $(f'_i)^{-1}(\Int V')$.
This property must be preserved by the map $\bar f'_i$, whereas the longitude of $V'$ cannot be homotoped towards the lower end without touching $U(\Sigma)$.
This is a contradiction.

Next suppose that there is a split level surface $F$ of $\check{\mathbf M}_i$ which is mapped into a pleated surface $P(F)$ above $U(\Sigma)$.
Since $F$ is homotopic to the lower end of $\check{\mathbf M}_i$, the pleated surface $P(F)$ can be homotoped below $U(\Sigma)$.
If the image of the homotopy can intersect an $\epsilon$-Margulis tube, it must be one coming from a Margulis tube lying below $F$ in $\check{\mathbf M}_i$.
Since we have already shown that such an $\epsilon$-Margulis tube cannot lie above $U(\epsilon)$, we see that $P(F)$ can be homotoped below $U(\Sigma)$ without passing through $\epsilon$-Margulis tubes.

By replacing the uppermost split level surface with any split level surface in the part corresponding to $S \times [3/16, 1/4]$ and repeating the argument above, we see that we can  homotope $\bar f_i'$ without passing $\epsilon$-Margulis tubes to make it preserve the order of split level surfaces.
This proves the most important step of the proof of our claim.

As we explained at the beginning of the proof, we can then homotope $\bar f_i'$ to make the image of each split level surface an embedding using the result of Freedman-Hass-Scott \cite{FHS}, and the entire map a homeomorphism, both without passing through $\epsilon$-Margulis tubes.
The fact that the homotopy which we constructed does not pass through $\epsilon$-Margulis tubes implies the resulting homeomorphism is again $K_0''$-Lipschitz for some $K_0''$ depending only on $\xi(S)$.
Finally, by applying the argument of \cite{BCM} involving geometric limits, we can show that the resulting map is a $K_0$-bi-Lipschitz $f_i'$ for some constant $K_0$ depending only on $\xi(S)$.
\end{proof}

Now, by composing $(\check f_i)^{-1}$ with $f_i'$, we get a $K'$-bi-Lipschitz embedding of $\check{\mathbf M}_i'$ into $M_i$, with $K'$ depending only on $K$ and $K_0$.
By filling appropriate Margulis tubes into $\check{\mathbf M}_i'$, we get a $K''$-bi-Lipschitz model manifold $\mathbf M_i'$ for $(M_i)_0$, with $K''$ independent of $i$, for sufficiently large $i$.

It remains to verify that this $\mathbf M_i'$ has a geometric limit with the desired properties.
Since $\mathbf M_i'$ has decomposition into blocks and geometrically finite bricks, and is a uniform bi-Lipschitz model for  $(M_i)_0$, by the same argument as the  proof of Theorem \ref{Ohshika-Soma} (see also \S 5 of \cite{OhSo} for the original argument), there is  a labelled brick manifold $\mathbf M'$ which is a bi-Lipschitz model manifold of $(M_\infty)_0$ such that $\mathbf M'_i$ converges to the union of $\mathbf M'$ and cusp neighbourhoods geometrically.
This shows the conditions (1) and (2).
Since both $\mathbf M$ and $\mathbf M'$ are model manifolds of the same geometric limit $(M_\infty)_0$, there is a  homeomorphism taking an algebraic locus to an algebraic locus between them, which shows the condition (3).

Since the geometric convergence of $\mathbf M_i'$ to $\mathbf M'$ preserves horizontal foliations, and the embedding of $\mathbf M'$ also preserves the horizontal levels, we see that the two horizontal annuli of each component of $\bar{\mathcal  T}^-$ lie on $S \times \{1/4\}$ and $S \times \{3/8\}$, whereas those of $\bar{\mathcal T}^+$ lies on $S\times \{5/8\}$ and $S \times \{3/4\}$.
Therefore, we have $\inf \bar B^1=\dots =\inf \bar B^q$ and $\sup \bar B^{q+1}=\dots =\sup \bar B^p$.
We can modify the embedding of $\mathbf M'$ into $S \times [0,1]$ only at these bricks to make them have  the same height and to make the condition  $\sup \bar B^1 =\dots =\sup \bar B^q, \inf \bar B^{q+1}=\dots =\inf \bar B^p$ hold.
Finally, we verify the condition (5).
Since $E^1, \dots , E^q$ are upper ends and $E^{q+1}, \dots , E^p$ are lower ends, an algebraic locus must pass under $\bar B^1, \dots , \bar B^q$ and above $\bar B^{q+1}, \dots , \bar B^p$.
By our assumption, there are no other algebraic ends, neither torus boundary components containing  algebraic parabolic curves.
Therefore, there is no obstruction to homotope an algebraic locus to a horizontal surface in this region above $\bar B^{q+1}, \dots , \bar B^p$ and below $\bar B^1, \dots , \bar B^q$.
\end{proof}

We shall use the symbol $g'$ to denote the horizontal algebraic locus in $\mathbf M'$ as in (5) of the above lemma.
(This is the same symbol as the standard immersion in $\mathbf M$, but there is no fear of confusion since we can distinguish them by model manifolds in which they are lying.)

Take $t$ and $t'$ such that $\bar B^j=\Sigma^j \times [5/8,t)$ for $j=1, \dots ,q$ and $\bar B^j=\Sigma^j \times (s',3/8]$ for $j=q+1, \dots , p$.
Let $\rho_i^{\mathbf M'}$ denote an approximate isometry between $\mathbf M_i'$ and $\mathbf M'$ which is associated to the geometric convergence of $\mathbf M_i'$ to the union of $\mathbf M'$ and cusp neighbourhoods.
We denote by $x_i'$ and $x_\infty'$ basepoints in the thick parts of $\mathbf M_i'$ and $\mathbf M'$, which we used for the geometric convergence.
By our construction of $\mathbf M_i'$, there are bricks $B_i^j \cong \Sigma^j \times [5/8, 3/4]\ (j=1,\dots, q)$ and $B_i^j \cong \Sigma^j \times [1/4, 3/8]\ (j=q+1, \dots, p)$ of $\check{\mathbf M}_i'$ which contains  $(\rho_i^{\mathbf M'})^{-1}(\bar B^j \cap B_{K_i r_i}(\mathbf M', x_\infty))$.

%

Now, we shall construct two sequences of markings on $S$ starting from $c_{m_i}$ and $c_{n_i}$ respectively, and a sequence of markings on $\Sigma^j$ for $j=1, \dots ,p$, in both of which  a marking advances by an elementary move.
 Recall that $\check{\mathbf M}'_i$, which we defined in the proof of Lemma \ref{modified models} above, consists of interior bricks lying on five levels, a disjoint union of the product $I$-bundles over three-holed spheres lying at two levels, and two geometrically finite bricks.
As was described there, there are tube unions in $\check{\mathbf M}'_i$ which decompose it into blocks.
Since the hierarchies which we used for the decomposition are all complete, by perturbing tubes vertically if necessary, each real front of every brick can be assumed to intersect the tubes constituting the decomposition in such a way that the complement of the tubes in the front is a disjoint union of thrice-punctured spheres.

We first consider the bottom interior brick corresponding to $S \times [3/16, 1/4]$.
We denote this brick by $\bar B_i$.
By our construction of tube unions in the proof of Lemma \ref{modified models}, we have a $4$-complete hierarchy $\check h(\bar B_i)$ on $S$  corresponding to tube unions in $\bar B_i$, whose initial marking is $c_{m_i}$.
In each annular component domain in $\check h(\bar B_i)$ except for those corresponding to tubes intersecting the upper front of $\bar B_i$, we can put a tight geodesic, which is uniquely determined.
The length of such a geodesic  determines the $\omega_{\check{\mathbf M}_i'}$ of  the Margulis tube which was filled in $\check{\mathbf M}_i'[0]$, attached to the corresponding  torus boundary.
For each annular component domain corresponding to a tube intersecting the upper front, we put a geodesic of length $0$ at this stage, for the terminal markings for such geodesics are not determined if we look only at $\bar B_i$.
We denote by $h(\bar B_i)$ the hierarchy obtained by adding such annular geodesics to $\check h(\bar B_i)$.
A resolution $\tau(\bar B_i)=\{\tau(\bar B_i)_k\}$ of $h(\bar B_i)$ gives rise to a sequence of split level surfaces starting from the one lying on the lower front and ending at the one lying on the upper front.
On the other hand, forward steps in the resolution $\tau(\bar B_i)$ correspond to  elementary moves of markings which can be assumed to be clean.
(See Minsky \cite{Mi}.)

The situation is quite similar for the top interior brick, which corresponds to $S \times [3/4,13/16]$, and we denote by $\hat B_i$.
As in the case of $\bar B_i$, we have a  hierarchy $h(\hat B_i)$.
Reversing the order of slices of this hierarchy, we get a resolution, which we denote by $\tau(\hat B_i)=\{\tau(\hat B_i)_k\}$, giving rise to  split level surfaces starting from the one on the upper front and ending at the one on the lower front, and a sequence of clean markings on $S$ advancing by elementary moves.
Since $\mathbf M_i'$ is obtained from $\check{\mathbf M}_i'$ by filling Margulis tubes, the split level surfaces as above for $\bar B_i$ and $\hat B_i$ can be also regarded as lying in $\mathbf M_i'$.

Next we consider a brick $B_i^j$ for $j=q+1, \dots , p$, which corresponds to $\bar B^j$ by $\rho_i^{\mathbf M'}$.
Again, we have a complete hierarchy $h(B_i^j)$ supported on $\Sigma^j$ corresponding to the decomposition of $B_i^j$ into blocks and filled-in Margulis tubes.
We can take a resolution $\tau(B_i^j)=\{\tau(B_i^j)_k\}$ starting from the restriction of the last slice of $\tau(\bar B_i)$ to $\Sigma^j$.
Similarly, for each brick $B_i^j$ for $j=1, \dots ,p$, we have a complete hierarchy $h(B_i^j)$.
We reverse the order of slices in this case, and consider a resolution $\tau(B_i^j)=\{\tau(B_i^j)_k\}$ starting from the restriction of the last slice of $\tau(\hat B_i)$.

For each $n \in \naturals$, we consider the $n$-th slice  $\tau(B_i^j)_{-n}$ counting from the last one of each $\tau(B_i^j)$ for $j=q+1, \dots , p$.
Corresponding to the slice $\tau(B_i^j)_{-n}$,  we have a split level surface $S(B_i^j)(n)$ embedded in $B_i^j$.
Taking the union of all $S(B_i^j)(n)$ for $j=q+1, \dots , p$ together with the thrice-punctured spheres lying on $\partial_+ \bar B_i \setminus (\cup_{j=p+1}^q B_i^j)$ which are parts of the last slice of $\tau(\bar B_i)$, we get a split level surface $S_i^-(n)$.
From this, we construct an {\em extended  level surface} $\hat S_i^-(n)$ as follows.

Let $\mathbf T_-$ be the union of Margulis tubes intersecting $S_i^-(n)$.
Each torus $T$  in $\partial \mathbf T_-$ is split into two annuli by cutting it  along $T \cap S_i^-(n)$, which we call the upper annulus and the lower annulus depending on their locations.
We paste the upper one to $S_i^-(n)$ for each $T$ and get a surface homeomorphic to $S$, which we define to be $\hat S_i^-(n)$.
We also regard a union of slices, taken one from each $\tau(B_i^j)\, (j=q+1, \dots ,p)$ as above, as a marking on the entire $S$ by defining its restriction to $S \setminus \cup_{j=q+1}^p \Sigma^j$ to be the marking defined by the last slice of $\tau(\bar B_i)$, and setting a  transversal of a component $c$ of $\Fr \Sigma^j$ to be the one determined by the first vertex of a geodesic in $h(\bar B_i)$ supported on an annular neighbourhood  $A(c)$ of $c$, which is uniquely determined by the Margulis tube of $\mathbf M'_i$ corresponding to $c$.
We note that every curve in the $\Fr B_i^j$ is a base curve of the last slice of $\tau(\bar B_i)$ since it is a core curve of the intersection of the upper front of $\bar B_i$ and a torus boundary component of $\check{\mathbf M}_i$.
We repeat the same construction for $B_i^j$ with $j=1, \dots ,q$, considering the $n$-th slice $\tau(B_i^j)_n$  $j=1, \dots, q$ counting from the first one, and get a split level surface $S_i^+(n)$ and an extended split level surface $\hat S_i^+(n)$, this time using lower annuli.
Also, we can regard a union of slices taken one from each $\tau(B_i^j)\, (j=1, \dots ,q)$ as a marking on the entire $S$ in the same way, by setting a transversal of each component of $\Fr \Sigma^j$ to be the last vertex this time.

Recall that in the construction of model manifolds, Minsky defined the initial and terminal markings to be the shortest markings for the conformal structures at infinity.
We consider a correspondence in the opposite direction.
We can choose some positive constant $\delta_0$ such that for any  clean marking, there is a marked conformal structure on $S$ for which the marking is shortest and in which all the curves (both base curves and transversals) of the marking have hyperbolic lengths greater than $\delta_0$.
(The existence of such a constant is easy to see since there are only finitely many configurations of clean markings up to homeomorphisms of $S$.
Recall that a clean marking is said to be shortest when its base curves constitute a shortest pants decomposition and the transversals are shortest among those obtained by Dehn twists around the base curves.)
\begin{definition}
\label{m}
For each marking $\mu$, we define the marked conformal structure $m(\mu)$ to be one for which a clean marking $\mu_0$ compatible with $\mu$ is  shortest  and such that every  curve of $\mu_0$ has hyperbolic length greater than $\delta_0$.
\end{definition}
Although there are more than one such structures, we just choose any one of them.
By using Kerckhoff's formula in \cite{Ker} for instance, we can easily see that there is a universal constant $K$ depending only on $\delta_0$ (and $S$) such that if $\mu'$ is obtained from $\mu$ by one step in the resolutions $\tau(\bar B_i)$ or $\tau(\hat B_i)$ or $\tau(B_i^j)$, which corresponds to an elementary move of markings on $S$ or $\Sigma_j$, then the Teichm\"{u}ller distance between $m(\mu)$ and $m(\mu')$ is bounded by $K$, whatever choices of $m(\mu)$ and $m(\mu')$ we make.

Now, we shall consider two sequences of markings: the first one connects the marking corresponding to first marking of $\tau(\bar B_i)$ with that corresponding to  the union of $\tau(B_i^j)_{-n}$ $(j=q+1, \dots , p)$; and the other, which proceeds in the negative direction, connects the marking corresponding to the last slice of $\tau(\hat B_i)$ with that corresponding to the union of $\tau(B_i^j)_n$ $j=1, \dots q$.
The first one is obtained by combining a sequence of markings corresponding to  $\tau(\bar B_i)$ with those corresponding to the $\tau(B_i^j)\, (j=q+1, \dots , p)$.
By our definition of markings corresponding to slices in $\tau(B_i^j)$,  if we choose the first slice from every $\tau(B_i^j)$, the corresponding marking coincides with the one corresponding to the last slice of $\tau(\bar B_i)$, where vertices on annular geodesics are set to be the first ones.
Therefore after the sequence of markings corresponding to $\tau(\bar B_i)$, we can append the one obtained by advancing slices in the $\tau(B_i^j)\, (j=q+1, \dots ,p)$, so that we proceed at each step by advancing a slice in $\tau(B_i^j)$ which is farthest from the goal, \ie the $n$-th slice counted from the last one, up to the point where all the slices are the $n$-th counting from the last one.
We denote thus obtained sequence of markings by $\mu^-(i,n)=\{\mu^-_k(i,n)\}_k$.
In the same way, we define a sequence of markings obtained by combining one corresponding to $\tau(\hat B_i)$ with those corresponding to $\tau(B_i^j)\, (j=1, \dots , q)$.
We denote this sequence by $\mu^+(i,n)=\{\mu^+_l(i,n)\}_l$.
To simplify the notation, we denote the last markings in $\mu^-(i,n)$ and $\mu^+(i,n)$ by $\mu^-_\infty(i,n)$ and $\mu^+_\infty(i,n)$ respectively.

For any $n$, if we take a sufficiently large $i$, the  component of $(\rho_i^{\mathbf M'} )^{-1}(\mathbf M' \cap B_{K_ir_i}(\mathbf M', x_\infty))$ intersecting $(\rho_i^{\mathbf M'})^{-1}(g'(S))$ contains all $\hat S^-_i(k)$ and $\hat S^+_i(k)$ with $k \leq n$, since the distance between $B_i^j$ and $(\rho_i^{\mathbf M'})^{-1}(g'(S))$ is uniformly bounded, and so are the diameters of extended  level surfaces.
Moreover making $i$ larger if necessary, we can make $\hat S^-_i(k)$ and $\hat S^+_i(k)$ with $k \leq n$ all homotopic to $(\rho_i^{\mathbf M'})^{-1}(g'(S))$ in $(\rho_i^{\mathbf M'} )^{-1}(\mathbf M' \cap B_{K_ir_i}(\mathbf M', x_\infty))$ since there is no torus boundary inside the $\bar B^j$, which implies that if we fix $k$, then the diameters of Margulis tubes intersecting essentially a homotopy between $\hat S^-_i(k)$ (or $\hat S^+_i(k)$) and  $(\rho_i^{\mathbf M'})^{-1}(g'(S))$ in $\mathbf M_i'$ are bounded independently of $i$.
Now, for each $i$, let $\mathbf n_i$ be the largest number such that both $S^-_i(k)$ and $S^+_i(k)$ with all $k \leq \mathbf n_i$ are contained in the component of  $(\rho_i^{\mathbf M'})^{-1}(\mathbf M' \cap B_{K_ir_i}(\mathbf M', x_\infty))$ intersecting $(\rho_i^{\mathbf M'})^{-1}(g'(S))$ and are homotopic to $(\rho_i^{\mathbf M'})^{-1}(g'(S))$ there.
By the above observation, we have $\mathbf n_i \rightarrow \infty$ as $i \rightarrow \infty$.
We then take another $\mathbf n_i'$ which also goes to $\infty$, such that $\mathbf n_i-\mathbf n_i'$ is positive and goes to $\infty$ as $i \rightarrow \infty$.

With this preparation in hand, we can now construct  arcs in the Teichm\"{u}ller space as follows.
Since  $\mu^-_k(i, \mathbf n_i')$ proceeds by advancing the slice among those  in $\tau(B_i^j)$ which is farthest from the $\mathbf n_i'$-th one counting from the last one, we see that for some $k=k_0$ the marking $\mu^-_k(i, \mathbf n_i')$ reaches the split level surface $S^-(\mathbf n_i)$.
Recall that for each $\mu^-_k(i, \mathbf n_i')$ in  $\mu^-(i,\mathbf n_i')$, the point $m(\mu^-_k(i, \mathbf n_i'))$ in the Teichmüller space was defined.
\begin{definition}
\label{m'}
We define a new function $m'$ by setting $m'(\mu^-_k(i, \mathbf n_i'))$ to be $m((\mu^-_k(i, \mathbf n_i'))$ for $k \leq k_0$, and for $k> k_0$ to be the hyperbolic structure obtained from  $m((\mu^-_k(i, \mathbf n_i'))$ by pinching the curves $c_1', \dots , c_t'$ so that their extremal lengths become $e^{k_0-k}$ and  the lengths of other curves in $\base(\mu^-_k(i, \mathbf n_i'))$ remain bounded below by $\delta_0$. 
\end{definition}
Starting from $\mu^-_1(i,\mathbf n_i')$, which is a shortest marking for $m_i$, we consider for each step $\mu^-_k(i,\mathbf n_i') \to \mu^-_{k+1}(i,\mathbf n_i')$ in $\mu^-(i,\mathbf n_i')$, a Teichm\"{u}ller geodesic arc connecting $m'(\mu^-_k(i,\mathbf n_i'))$ with $m'(\mu^-_{k+1}(i,\mathbf n_i'))$, and then construct a broken geodesic arc $\alpha^-(i,\mathbf n_i')$ connecting $m'(\mu^-_1(i,\mathbf n_i'))$ to $m'(\mu^-_\infty(i,\mathbf n_i'))$ by joining them.
We note that by the definition of $m'$, each constituting geodesic arc has uniformly bounded length.
Both $m'(\mu^-_1(i,\mathbf n_i'))$ and $m_i$ have $\mu^-_1(i,\mathbf n_i')$ as a shortest marking, but the hyperbolic lengths of base curves of $\mu^-_1(i,\mathbf n_i')$   in $m_i$ may be different from those in $m'(\mu_1^-(i,\mathbf n_i'))$.
We can connect $m_i$ with $m'(\mu^-_1(i,\mathbf n_i'))$ by a Teichm\"{u}ller quasi-geodesic so that $\base(\mu^-_1(i,\mathbf n_i'))$ remains a shortest pants decomposition throughout the points on the geodesic.
(The quasi-geodesic constant depends only on $S$.)
We define $\bar \alpha^-(i,\mathbf n_i')$ to be a broken quasi-geodesic arc obtained by joining this quasi-geodesic with $\alpha^-(i,\mathbf n_i')$.
In the same way, we construct a broken quasi-geodesic arc $\bar \alpha^+(i,\mathbf n_i')$, which connects $n_i$ with $m(\mu^+_\infty(i,\mathbf n_i'))$ by pinching the curves $c_1, \dots , c_s$.
In the case when either lower or upper algebraic simply degenerate ends do not exist,  we define the corresponding $\bar \alpha^-(i,\mathbf n_i')$ or $
\bar\alpha^+(i,\mathbf n_i')$ to be a constant map.

Let $\alpha^\pm_i: [0,k_i^\pm] \rightarrow \mathcal T(S)$  be broken quasi-geodesic arcs $\bar \alpha^\pm(i,\mathbf n_i)$ connecting $m'(\mu^\pm_\infty(i,\mathbf n_i'))$ with $m_i, n_i$ which are constructed as above by joining Teichm\"{u}ller geodesics and one quasi-geodesic.
We parametrise $\alpha^\pm_i$
so that  $\alpha^\pm_i| [0,1]$ is the appended quasi-geodesics connecting $m_i$ or $n_i$ with $m'(\mu_1^\pm(i, \mathbf n_i'))$, and  $\alpha^\pm_i| [s,s+1]$ for $s \in \naturals$ corresponds to the Teichm\"{u}ller geodesic constituting $\alpha^\pm_i$ which connects $m'(\mu_s^\pm(i, \mathbf n_i'))$ with $m'(\mu_{s+1}^\pm(i, \mathbf n_i'))$.
We define $\beta_i: [0,\bar k_i] \rightarrow QF(S)$ by setting $\beta_i(t)$ to be $qf( \alpha_i^-(t),  \alpha^+_i(t))$, where $\bar k_i=\max\{k_i^+,k_i^-\}$ and assuming that the arcs $\alpha_i^\pm$ stay at the endpoint after $t$ gets out of the domain.

\begin{claim}
\label{strong limit}
A sequence of quasi-Fuchsian group $\{qf(m'(\mu_\infty^-(i,\mathbf n_i')), m'(\mu_\infty^+(i,\mathbf n_i'))\}$ converges strongly to a quasi-conformal deformation of $(\Gamma, \psi)$ as $i \rightarrow \infty$ after passing to a subsequence.
\end{claim}
\begin{proof}
By our definition of the function $m'$ and the  argument in the proof of Lemma \ref{modified models}, a uniformly bi-Lipschitz model manifold for the quasi-Fuchsian group $qf(m'(\mu^-_\infty(i,\mathbf n_i')), m'(\mu^+_\infty(i,\mathbf n_i')))$, which we denote by $\mathbf M'({\mathbf n_i'})$, can be obtained from the submanifold of $\mathbf M_i'[0]$ cobounded by $S^-_i(\mathbf n_i')$ and $S^+_i(\mathbf n_i')$ by pasting boundary blocks corresponding to $m'(\mu^-_\infty(i,\mathbf n_i'))$ and $m'(\mu^+_\infty(i,\mathbf n_i'))$ respectively and filling Margulis tubes.
For large $i$, this model manifold contains $(\rho_i^{\mathbf M'})^{-1} \circ g'(S)$ and $\rho_i^{\mathbf M'} \circ g'$ is homotopic to $\Phi_i$.
It follows that for a fixed generator system of $\pi_1(S)$ with base point on $(\rho_i^{\mathbf M'})^{-1} \circ g'(S)$, the   length of the shortest closed loop in $\mathbf M'({\mathbf n_i'})$ representing each generator is bounded as $i \rightarrow \infty$, and hence that $qf(m'(\mu^-_\infty(i,\mathbf n_i')), m'(\mu^+_\infty(i,\mathbf n_i')))$ converges algebraically after passing to a subsequence.
Let $(\Gamma', \psi')$ be the algebraic limit, and denote $\hyperbolic^3/\Gamma'$ by $M_{\Gamma'}$.

Let $\mathbf M'({\mathbf n_\infty'})$ be the geometric limit  of the complement of the boundary blocks in $\mathbf M'({\mathbf n_i'})$  as $i \rightarrow \infty$.
Then $\mathbf M'({\mathbf n_\infty'})$ is embedded in $\mathbf M'$ as a submanifold, by our definition of $\mathbf M'({\mathbf n'_i})$.
For each $j=1, \dots , q$, the intersection of $\rho_i^{\mathbf M'}(S^+_i(\mathbf n_i'))$ with $\bar B^j$ goes deeper and deeper into $\bar B^j$  to the direction of the end $E_j$ as $i \rightarrow \infty$ since $\mathbf n_i' \rightarrow \infty$.
The same holds for $\rho_i^{\mathbf M'}(S^-_i(\mathbf n_i')) \cap B^j$ for $j=q+1, \dots p$.
Therefore the entire $\bar B^j$ is contained in $\mathbf M'({\mathbf n'_\infty})$ for every $j=1, \dots ,p$.

Now we look at the two boundary blocks of $\mathbf M'(\mathbf n_i')$, which we denote by $b^-_i$ and $b^+_i$, where $b^-_i$ has conformal structure at infinity $m'(\mu^-_\infty(i,\mathbf n_i'))$ and $b^+_i$ has conformal structure at infinity $m'(\mu^+_\infty(i,\mathbf n_i'))$, and consider their geometric limits.
By our definition of $m'$ and from the fact that $\mathbf n_i-\mathbf n_i' \rightarrow \infty$, we see that the hyperbolic lengths of ${c_1^i}', \dots , {c_t^i}'$ with respect to $m'(\mu^-_\infty(i,\mathbf n_i'))$ and those of $c_1^i,\dots , c_s^i$ with respect to $m'(\mu^+_\infty(i,\mathbf n_i'))$ go to $0$.
Let $F^1, \dots, F^r$ be the components of $S \setminus \cup_{k=1}^q \Sigma^k$.
Then $m'(\mu^+_\infty(i, \mathbf n_i'))$ restricted to each one of $F^1, \dots, F^r$ has pants decomposition whose lengths are bounded below by $\delta_0$, and 
 the boundary components of $F^1, \dots, F^r$ correspond to  $c_i^1, \dots , c_s^i$.
Therefore, the geometric limit of $b^+_i$ with basepoint in the part of $F^k$ is homeomorphic to $\Int F^k \times I$ with conformal structure at infinity corresponding to a complete hyperbolic structure on $\Int F^k$ making each frontier component a cusp.
The same holds for $b^-_i$ considering the components of $S \setminus \cup_{k=q+1}^p \Sigma^k$.

Therefore, the geometric limit of $\mathbf M'(\mathbf n_i')$ has upper simply degenerate ends corresponding to $\Sigma^1, \dots , \Sigma^q$, lower simply degenerate ends corresponding to $\Sigma^{q+1}, \dots , \Sigma^p$, upper geometrically finite ends corresponding to the components of $S \setminus \cup_{k=1}^q \Sigma^k$, and lower geometrically finite ends corresponding to the components of $S \setminus \cup_{k=q+1}^p \Sigma^k$, all of which are algebraic.
This shows that the geometric limit has fundamental group isomorphic to $\pi_1(S)$, and hence that the convergence is strong.
We denote a representative of this strong limit by $(\Gamma', \psi')$.
The model manifold shows that the ends of $M_{\Gamma'}$ consists of simply degenerate ends corresponding to $E^1, \dots , E^p$ and geometrically finite ends and that every parabolic locus touches one of $E^1, \dots , E^p$.
Since the ending laminations of $E^1, \dots , E^p$ are the same as those of the corresponding ends of $M'=\hyperbolic^3/\Gamma$, by the ending lamination theorem due to Brock-Canary-Minsky \cite{BCM}, we see that $(\Gamma',\psi')$ is a quasi-conformal deformation of $(\Gamma, \psi)$.
\end{proof}

\begin{claim}
\label{any limit}
For any sequence $\{t_i \in [0,\bar k_i]\}$, the sequence $\{\beta_i(t_i)\in QF(S)\}$ converges algebraically to a quasi-conformal deformation of $(\Gamma, \psi)$.
\end{claim}
\begin{proof}
Recall that $\beta_i(t_i)=qf(\alpha^-_i(t_i), \alpha^+_i(t_i))$, and that  $\alpha^-$ and $\alpha^+$ are broken quasi-geodesic arcs consisting  of Teichm\"{u}ller geodesics with bounded lengths except for the first quasi-geodesics; $\alpha^\pm|[0,1]$, which may be long.

We first assume that neither $\alpha_i^-(t_i)$ nor $\alpha^+(t_i)$ lies on the first quasi-geodesics, which connect   $m'(\mu_1^-(i, \mathbf n_i))$ with and $m'(\mu_1^+(i, \mathbf n_i))$ with respectively.
In this case, since $\alpha^-_i(t_i)$ and $\alpha^+(t_i)$ lie on Teichm\"{u}ller geodesics with bounded length,  we have only to consider the case when both $\alpha^+(t_i)$ and $\alpha_i^-(t_i)$ are endpoints of Teichm\"{u}ller geodesic arcs constituting $\alpha^-$ and $\alpha^+$, \ie the case when $t_i$ is an integer.
The marking $\alpha^-(t_i)$ corresponds to either a slice in $\tau(\bar B_i)$ or a union of slices, taken one from each of $\tau(B_i^j)\, (j=q+1, \dots , p)$.
This corresponds in turn to an extended level surface $\hat S ^-_i(t_i)$.
Similarly, we have an extended level surface $\hat S^+_i(t_i)$ corresponding to $\alpha^+(t_i)$.
Then, we can see that a uniform bi-Lipschitz model manifold $\mathbf M'(\beta_i(t_i))$ for $\beta_i(t_i)$ is obtained from the submanifold of $\mathbf M_i'$ cobounded by $S^-_i(t_i)$ and $S^+_i(t_i)$ by pasting boundary blocks and filling Margulis tubes by the argument of Lemma \ref{modified models}.
Now, by the same argument as in the proof of the previous claim, we see that $\{\beta_i(t_i)\}$ converges algebraically after passing to a subsequence, and the geometric limit $\mathbf M_\infty'(\beta)$ of $\mathbf M'(\beta_i(t_i))$ contains bricks $\bar B^1, \dots , \bar B^p$.
Since the internal blocks of $\mathbf M_\infty'(\beta)$ also lie in $\mathbf M'$, the ends of the algebraic limit other than those corresponding to $E^1, \dots , E^p$ are geometrically finite and there are no extra parabolic loci.
This implies that the algebraic limit is a quasi-conformal deformation of $(\Gamma, \psi)$ by the ending lamination theorem, as in the previous claim.

Next suppose that $\alpha^+(t_i)$ lies in the first quasi-geodesic, \ie $t_i \in [0,1]$.
Then all the internal blocks of $\mathbf M_i'$ above $(\rho_i^{\mathbf M'})^{-1} \circ g'(S)$  lie also in a model manifold $\mathbf M'(\beta_i(t_i))$ for $\beta_i(t_i)=qf(\alpha_i^-(t_i), \alpha_i^+(t_i))$ defined in the same way as above.
Therefore, the algebraic ends of $\mathbf M'_i$ lie also in the model  manifold $\mathbf M'(\beta_i(t_i))$.
The above argument for showing that there is no extra parabolic loci works also in this case.

We can argue in the same way also in the case when $\alpha^-(t_i)$ lies (or both $\alpha^-(t_i)$ and $\alpha^+(t_i)$ lie) in the first quasi-geodesic.
\end{proof}

Now, we shall complete the proof of Proposition \ref{deformation to strong one}, by setting $\alpha_i(t)=\beta_i(\bar k_i t)$ for $\beta_i$ defined above, and showing thus constructed arc $\alpha_i$ satisfies the conditions in the statement.
The conditions (1) and (2) have already been shown.
Let us show (3).
By our definition of $m'$, the lengths of $c_1, \dots , c_s$ with respect to  $m'(\mu_\infty^+(i, \mathbf n_i'))$ go to $0$, and also those of $c_1', \dots , c_t'$ with respect to $m'(\mu_\infty^-(i,\mathbf n_i'))$ go to $0$.
This means that the length of each of $c_1, \dots , c_s$ with respect to $\bar n_i$ and that of each of $c_1', \dots , c_t'$ with respect to $\bar m_i$ goes to $0$.

Next we turn to the condition (4).
In the case when $\Gamma$ is a b-group, there is no lower algebraic simply degenerate ends for $\mathbf M'$.
By our definition of $\beta_i$, in this case the lower conformal structure $\alpha_i^-(t)$ is the same as $m_0$ for every $t$.

Finally, we verify the condition (5).
Suppose, seeking a contradiction, that the condition (5) does not hold.
Then, there exist a neighbourhood $U$ of $QH(\Gamma, \psi)$ and $t_i \in [0,k_i]$ such that $\beta_i(t_i)$ stays outside $U$ for every $i$ after passing to a subsequence.
Now, we apply Claim \ref{any limit} to see that $\beta(t_i)$ converges to a point in $QH(\Gamma, \psi)$ after passing to a subsequence.
This is a contradiction.
Thus we have completed the proof of Proposition \ref{deformation to strong one}.

\subsection{Proofs of Theorem \ref{bumping theorem} and Corollaries \ref{no self-bumping}, \ref{no self-bumping with isolated}}
\begin{proof}[Proof of Theorem \ref{bumping theorem}]
Having proved Proposition \ref{deformation to strong one}, to prove Theorem \ref{bumping theorem}, it remains to show that two sequences strongly converging to groups in the quasi-conformal deformation space of $(\Gamma, \psi)$ as constructed in the proof of Proposition \ref{deformation to strong one} can be joined by arcs in  small neighbourhoods of the deformation space, fixing the conformal structure at bottom when $\Gamma$ is a b-group.

Let $\{(G_i, \phi_i)\}$ and $\{(H_i, \varphi_i)\}$ be two sequences of quasi-Fuchsian groups both of which converge algebraically to quasi-conformal deformations of $(\Gamma, \psi)$.
Let $\mathbf M_i'$ and $\mathbf N_i'$ be model manifolds for $\hyperbolic^3/G_i$ and $\hyperbolic^3/H_i$ constructed as in Lemma \ref{modified models}, converging geometrically to model manifolds $\mathbf M'$ and $\mathbf N'$ for the geometric limits of $\{G_i\}$ and $\{H_i\}$ respectively.
Since the algebraic limits of $\{(G_i, \phi_i)\}$ and $\{(H_i, \varphi_i)\}$ are quasi-conformally equivalent, $\mathbf M'$ and $\mathbf N'$ have the same number of  algebraic simply degenerate ends with the same ending laminations $\lambda_1, \dots , \lambda_p$.
As in Proposition \ref{deformation to strong one}, we renumber them so that those having $\lambda_1, \dots , \lambda_q$ as ending laminations are upper whereas the rest are lower.
In particular, we can assume that the boundary components  touching  these ends are the same open annuli in $S \times [0,1]$ for $\mathbf M'$ and $\mathbf N'$.
We denote the union of these annuli by $\mathcal T$.

Let $\{(G_i', \phi_i')=qf(\bar m_i, \bar n_i)\}$ and $\{(H_i', \varphi_i')=qf(\bar \mu_i, \bar \nu_i)\}$ be two strongly convergent sequences as constructed in Proposition \ref{deformation to strong one} for the two sequences $\{(G_i,\phi_i)\}$ and $\{(H_i, \varphi_i)\}$.
Recall that in the construction of such a sequence in Proposition  \ref{deformation to strong one}, we took a number $\mathbf n_i'$.
Note that the construction works if we take a number smaller than $\mathbf n_i'$ for each $i$ provided that it goes to $\infty$.
Therefore, we can let the number $\mathbf n_i'$ be common to both $G_i'$ and $H_i'$.
Let $\mathbf M'({\mathbf n_i'})$ and $\mathbf N'({\mathbf n_i'})$ be model manifolds for them as in the proof of \cref{deformation to strong one}.
Then both $\mathbf M'(\mathbf n_i')$ and $\mathbf N'(\mathbf n_i')$, regarded as embedded in $S \times [0,1]$, have the following properties.
There are unions of Margulis tubes $\mathbf V_i$ and $\mathbf V_i'$  coming from $\mathcal T$ in  $\mathbf M'(\mathbf n_i')$ and $\mathbf N'(\mathbf n_i')$ respectively, which can be assumed to correspond to the same union of solid tori in $S \times [0,1]$.
As in the proof of Proposition \ref{deformation to strong one}, the complements $\mathbf M'(\mathbf n_i') \setminus \mathbf V_i$ and $\mathbf N'(\mathbf n_i') \setminus \mathbf V_i'$ have decompositions into bricks among which there are ${B_i^{j}}'=\Sigma^j \times {J_i^{j}}'$ in $\mathbf M'(\mathbf n_i')$ and ${B_i^{j}}''=\Sigma^j \times {J_i^{j}}''$ in $\mathbf N'(\mathbf n_i')$ on the same side of the preimages of the standard algebraic immersions.
The bricks ${B_i^{j}}'$ and ${B_i^{j}}''$ contain tube unions corresponding to tight geodesics ${\gamma_i^{j}}'$ and ${\gamma_i^{j}}''$ supported on $\Sigma^j$ whose lengths go to $\infty$ as $i \rightarrow \infty$.
Furthermore, for each of ${\gamma_i^{j}}'$ and ${\gamma_i^{j}}''$, one of the endpoints stays in a compact set and the other endpoint (which we denote by ${b_i^{j}}'$ for ${\gamma_i^{j}}'$ and ${b_i^{j}}''$ for ${\gamma_i^{j}}''$) goes to the ending lamination $\lambda_j$ as $i \rightarrow \infty$.
Also, the tube unions put in ${B_i^j}'$ and ${B_i^j}''$ induce hierarchies ${h_i^j}'$ and ${h_i^j}''$ on $\Sigma^j$ with main geodesics ${\gamma_i^j}'$ and ${\gamma_i^j}''$ respectively.

In the following, we only consider the case when the end of $\mathbf M'$ (hence also that of $\mathbf N'$) having $\lambda_j$ as the ending lamination is an upper end, \ie $j=1, \dots q$.
The case when it is an lower end can be dealt with just by turning everything upside down as usual.
Recall that we constructed a sequence of markings $\mu^+(i,\mathbf n_i')$ in the proof of Proposition \ref{deformation to strong one}.
We construct its counterpart  ${\mu^+}'(i,\mathbf n_i')$ for $\mathbf N'$.
As can been seen in the construction of $\mathbf M'(\mathbf n_i')$ in the proof of Proposition \ref{deformation to strong one}, the last slices of ${h_i^j}'$ and ${h_i^j}''$ correspond to the restrictions of the last terms $\mu^+_\infty (i,\mathbf n_i')$ and ${\mu^+_\infty}'(i,\mathbf n_i')$ of the sequences of markings $\mu^+(i,\mathbf n_i')$ and ${\mu^+}'(i,\mathbf n_i')$.

Since the length of the frontier of $\Sigma^j$ goes to $0$ with respect to both $\bar n_i$ and $\bar \nu_i$, we can assume that the length of each component of $\Sigma^j$ with respect to $\bar n_i$ is equal to that with respect to $\bar \nu_i$ without changing the algebraic limit and the structure of the model manifolds except for the boundary blocks, by deforming $\bar m_i$ and $\bar \nu_i$ locally in a thin annular neighbourhood of $\Sigma^j$.

Now, we connect the endpoint $b_i^{j'}$ of $\gamma_i^{j'}$ with the endpoint $b_i^{j''}$ of $\gamma_i^{j''}$ by a tight sequence $\gamma_i=\{s^i_1, \dots , s^i_{u_i}\}$ in $\CC(\Sigma^j)$.
By the hyperbolicity of $\CC(\Sigma^j)$, for any choice of an integer $v_i$ between $1$ and $u_i$, the simplex $s^i_{v_i}$ also converges to the lamination $\lambda_j$ as $i \rightarrow \infty$.
By letting $\mu^+_\infty(i, \mathbf n_i')$ and ${\mu^+_\infty}'(i,\mathbf n_i')$ be the initial and the terminal markings respectively, we can regard $\gamma_i$ as a tight geodesic, and there is a hierarchy $h(\gamma_i)$ on $\Sigma^j$ which has $\gamma_i$ as its main geodesic.
Considering a resolution of this hierarchy $h(\gamma_i)$, we can connect $\mu^+_\infty(i, \mathbf n_i')$ with ${\mu^+_\infty}'(i,\mathbf n_i')$ by elementary moves of markings $\bar \mu_\infty^i(r), r=0, \dots , w_i$.
Recall that $\mu^+_\infty(i, \mathbf n_i')$ is a shortest marking for $(S, \bar n_i)$ and ${\mu^+_\infty}'(i, \mathbf n_i')$ is a shortest marking for $(S, \bar \mu_i)$ and that both of them contain every component of $\Fr \Sigma^j$.
Therefore we can connect   $\bar n_i$ with $\bar n_i|(S\setminus \Sigma^j) \cup \bar \nu_i|\Sigma^j$ by a piecewise Teichm\"{u}ller geodesic arc $\delta_i: [0,w_i] \rightarrow \mathcal T(S)$ so that for any integer $r$, the restriction  $\bar \mu_\infty^i(r)|\Sigma^j$ is a shortest marking for $\delta_i(r)|\Sigma^j$ whereas $\delta_i(r)|S\setminus \Sigma^j)=\bar n_i|(S\setminus \Sigma^j)$.
Since the base curve of $\bar \mu^i_\infty(r)|\Sigma^j$ lies among $c^i_1, \dots, c^i_{u_i}$, every $c^i_{s_i}$ converges to $\lambda_j$, and the structure of $\mathbf M_i'(\mathbf n_i')$ outside $\bar B^j$ is unchanged, we see that $qf(\bar m_i, \delta_i(t_i))$ converges to $(\Gamma, \psi)$ strongly for any sequence $\{t_i \in [0,w_i]\}$.
By the same argument as in the proof of Proposition \ref{deformation to strong one}, we see that for any neighbourhood $U$ of $QH(\Gamma, \psi)$, the arc $qf(\bar m_i,\delta_i[0,w_i])$ is contained in $U$ for large $i$.

We repeat the same procedure for each $j=1, \dots, p$.
Then, we get an arc $\alpha'_i$ connecting $qf(\bar m_i, \bar n_i)$ with $qf(\bar m_i|(S \setminus \cup_{j=q+1}^p \Sigma^j) \cup \bar \mu_i|\cup_{j=q+1}^p \Sigma^j, \bar n_i|(S \setminus \cup_{j=1}^q \Sigma^j) \cup \bar \nu_i|\cup_{j=1}^p \Sigma^j)$ such that for any neighbourhood $U$ of $QH(\Gamma, \psi)$, the arc $\alpha'_i$ is contained in $U$ for sufficiently large $i$.
Since all of the $\bar m_i|(S \setminus \cup_{j=q+1}^p \Sigma^j)$, the $\bar \mu_i|(S \setminus \cup_{j=q+1}^p \Sigma^j)$, the $\bar n_i|(S \setminus \cup_{j=1}^p \Sigma^j)$, and the $\bar \nu_i|(S \setminus \cup_{j=1}^p \Sigma^j)$ are bounded in the Teichm\"{u}ller spaces with free boundary, and the lengths of the boundary components are the same for $\bar m_i$ and $\bar \mu_i$ and for $\bar n_i$ and $\bar \nu_i$ and go to $0$ as $i \to \infty$, we can deform $qf(\bar m_i|(S \setminus \cup_{j=q+1}^p \Sigma^j) \cup \bar \mu_i|\cup_{j=q+1}^p \Sigma^j, \bar n_i|(S \setminus \cup_{j=1}^q \Sigma^j) \cup \bar \nu_i|\cup_{j=1}^p \Sigma^j)$ to $qf(\hat \mu_i, \hat \nu_i)$ by a uniformly bounded quasi-conformal deformation, where the differences between $\hat \mu_i$ and $\bar \mu_i$ and between $\hat \nu_i$ and $\bar \nu_i$ are just compositions of Fenchel-Nielsen twists around $c_1, \dots, c_s$ and $c_1' ,\dots , c'_t$ respectively.
This quasi-conformal deformation gives an arc connecting $qf(\bar m_i|(S \setminus \cup_{j=q+1}^p \Sigma^j) \cup \bar \mu_i|\cup_{j=q+1}^p \Sigma^j, \bar n_i|(S \setminus \cup_{j=1}^q \Sigma^j) \cup \bar \nu_i|\cup_{j=1}^p \Sigma^j)$ to $qf(\hat \mu_i, \hat \nu_i)$ which is contained in the neighbourhood $U$ for large $i$.

Now, we connect $\hat \mu_i$ with $\bar \mu_i$ and $\hat \nu_i$ with $\bar \nu_i$ by arcs $\boldsymbol \mu_i(x), \boldsymbol \nu_i(x)$ in the Teichm\"{u}ller space realising these compositions of Fenchel-Nielsen twists corresponding to their differences.
Then the lengths of $c_1, \dots , c_s$ with respect to $\boldsymbol \mu_i(x)$ also go to $0$ as $i \rightarrow \infty$ for any $x$, and so do the lengths of $c_1', \dots , c_t'$ with respect to $\boldsymbol \nu_i(x)$.
Therefore, by the same argument as before, we can connect $qf(\hat \mu_i, \hat \nu_i)$ with $qf(\bar \mu_i, \bar \nu_i)$ by an arc which is contained in $U$ for large $i$.

Thus, joining the arcs obtained this way, we have shown that we can connect $(G_i',\phi_i')=qf(\bar m_i,\bar n_i)$ with $(H_i',\psi_i')=qf(\bar \mu_i,\bar \nu_i)$ by an arc $\alpha_i$ in $QF(S)$ such that for any neighbourhood $U$ of $QH(\Gamma, \psi)$, the arc $\alpha_i$ is contained in $U$ large $i$.
On the other hand,  by our definition of $(\bar m_i, \bar n_i)$ and $(\bar \mu_i,\bar \nu_i)$ which uses Proposition \ref{deformation to strong one}, there are arcs with such a property connecting $(G_i,\phi_i)$ with $(G'_i,\phi_i')$ and $(H_i, \psi_i)$ with $(H_i',\psi_i')$.
Therefore, connecting these three arcs, we get an arc as we wanted. 

In the case when $\Gamma$ is a b-group, by Theorem \ref{no exotic convergence}, the lower conformal structures at infinity of both $(G_i, \phi_i)$ and $(H_i, \psi_i)$ converge to $m_0$.
Therefore, we can construct an arc as above keeping the lower conformal structures in an arbitrarily  small neighbourhood of $m_0$ for large $i$.
This shows the second paragraph of our theorem.
\end{proof}

Corollary \ref{no self-bumping} follows easily from this as follows.

\begin{proof}[Proof of Corollary \ref{no self-bumping}]
Suppose that every small neighbourhood   of $(\Gamma, \psi)$ intersects more than one component of $QF(S)$.
Then  for every small neighbourhood $U$ of $(\Gamma, \psi)$ in $AH(S)$, there are sequences $\{qf(m_i,n_i)\}$ and $\{qf(\hat m_i,\hat n_i)\}$ both converging to $(\Gamma, \psi)$ such that $qf(m_i,n_i)$ and $qf(\hat m_i,\hat n_i)$ belong to different component of $U \cap QF(S)$.
Applying Theorem \ref{bumping theorem}, we see that for any small neighbourhood $V$ of $QF(\Gamma, \psi)$ in $AH(S)$, for sufficiently large $i$, the two points $qf(m_i,n_i)$ and $qf(\hat m_i,\hat n_i)$ must be connected in $V \cap QF(S)$.
In the case when every component of $\Omega_\Gamma/\Gamma$ is a thrice-punctured sphere, this is a contradiction since $QH(\Gamma,\psi)$ consists of only $(\Gamma, \psi)$ then and we can let $V$ be $U$.

In the case when there is a component of $\Omega_\Gamma/\Gamma$ which is homeomorphic to $S$, the Kleinian group $\Gamma$ is a b-group.
Then the second paragraph of Theorem \ref{bumping theorem} shows that we can connect $qf(m_i,n_i)$ to $qf(\hat m_i,\hat n_i)$ keeping the lower conformal structure within a small neighbourhood.
This means that they can be connected in a small neighbourhood of $(\Gamma, \psi)$ in $QF(S)$.
This again is a contradiction.
\end{proof}

To get Corollary \ref{no self-bumping with isolated}, we need to review the argument of the proof of Theorem \ref{bumping theorem}.

\begin{proof}[Proof of Corollary \ref{no self-bumping with isolated}]
In Corollary \ref{no self-bumping with isolated}, we have allowed isolated parabolic loci to exist.
We consider a sequence of quasi-Fuchsian groups $\{qf(m_i, n_i)\}$  converging to $(\Gamma, \psi)$ in this generalised setting, and shall show that we can deform $\{qf(m_i, n_i)\}$ continuously to a sequence which converges strongly to $(\Gamma, \psi)$.
In the proof of Theorem \ref{bumping theorem}, we used the assumption that $\Gamma$ has no isolated parabolic loci in two places; first to show that the standard immersion $g'$ can be isotoped to a horizontal embedding, and secondly in the proofs of Claims \ref{strong limit} and \ref{any limit} to show that the limits are quasi-conformal deformations of $(\Gamma, \psi)$.
We shall show how to modify the argument in the latter part first.

(1) {\bf How to modify the proofs of Claims \ref{strong limit} and \ref{any limit}.}
In the proofs of these claims, we used the fact that all the parabolic loci of the algebraic limit $\Gamma'$ of the new sequence touch simply degenerate ends to imply that the parabolic loci of $\Gamma'$ and $\Gamma$ are the same.
In our setting now, if we do the same construction, this is not the case any more.
To preserve the parabolic loci of  algebraic limits throughout the modification of quasi-Fuchsian groups, we consider to modify the model manifolds $\mathbf M'$ and $\mathbf M'_i$ constructed in  Lemma \ref{modified models} and the function $m$ as follows.

We divide the isolated parabolic loci of $\Gamma$ into two categories.
Let $c$ be a core curve of an isolated parabolic locus $P$ of $\Gamma$.
We can regard $c$ as lying on the boundary of the model $\mathbf M'$ of the non-cuspidal part of the geometric limit $(M_\infty)_0$.
We say that the parabolic locus $P$ and its core curve $c$ are of torus type if $c$ can be regarded as lying on a torus boundary component of $\mathbf M'$ and are of annulus type if $c$ is regarded as lying on an open annulus component of $\mathbf M'$.

Let $d_1, \dots, d_u$ be  core curves of the isolated parabolic loci, not caring whether they are of torus type or annulus type.
We renumber them so that $d_1, \dots, d_v$ are upper and $d_{v+1}, \dots , d_u$ are lower.
These curves correspond to algebraic parabolic curves lying on boundary components in the original model manifold $\mathbf M$.
We take annular neighbourhoods $A(d_1), \dots , A(d_v)$ and $A(d_{v+1}), \dots, A(d_u)$ of these curves on $S$ so that both $A(d_1), \dots , A(d_v)$ and $A(d_{v+1}), \dots, A(d_u)$ are pairwise disjoint.
We set $V(d_j)=A(d_j) \times [5/8,3/4]$ for $j=1, \dots, v$ and  $V(d_j)=A(d_j) \times [1/4,3/8]$ for $j=v+1, \dots , u$.
We set $\mathcal U^+$ to be $\cup_{j=1}^v V(d_j)$ and $\mathcal U^-$ to be $\cup_{j=v+1}^u V(d_j)$.
%

We then define $\bar {\mathbf M}$ to be $S \times [1/8,7/8] \setminus (\mathcal T^- \cup \mathcal T^+ \cup \mathcal U^- \cup \mathcal U^+)$, where $\mathcal T^-$ and $\mathcal T^-$ are unions of solid tori corresponding to non-isolated parabolic curves which were defined in the proof of Lemma \ref{modified models}.
This new $\bar {\mathbf M}$ also admits standard brick decomposition having five stages although the bottom and top levels are changed to $1/8$ and $7/8$.
The bottom brick and the top brick are denoted $\bar B_i$ and $\hat B_i$ as before.
By the same construction as in the proof of Lemma \ref{modified models}, we can construct labelled brick manifolds $\check{\mathbf M}_i$ by attaching geometrically finite blocks corresponding to $m_i$ and $n_i$, and  then model manifolds $\mathbf M_i'$ by filling in Margulis tubes.
As in the proof of Lemma \ref{modified models}, we put tube unions in $\bar B_i$ and $\hat B_i$, which are parts of block decomposition of $\check{\mathbf M}_i$, and decompose them into blocks.
These give rise to resolutions $\tau(\bar B_i)$ and $\tau(\hat B_i)$ as before.

Next we turn to the levels $[1/4,3/8]$ and $[5/8, 3/4]$.
On these levels, there lie bricks $B_i^j$, among which $B_i^1, \dots , B_i^q$ are contained in $S \times [5/8,3/4]$ and $B_i^{q+1}, \dots , B_i^p$ are contained in $S \times [1/4, 3/8]$.
For these, hierarchies $h(B_i^j)$ resolutions $\tau(B_i^j)$ are defined in the same way as before.
On the other hand, the decomposition of $\check{\mathbf M}_i$ into blocks determines a geodesic $h_i(d_j)$ supported on $A(d_j)$ for $j=1, \dots , u$.
If the length of $h_i(d_j)$ goes to $\infty$ as $i \rightarrow \infty$, then $d_j$ is of torus type, otherwise it is of annulus type.
Now, recall that for $n \in \naturals$, we defined a slice on $S$ extending the $n$-th slices on the $\tau(B_i^j)$, and then defined a sequence of markings $\mu^-(i,n)$.
In the present situation where there are isolated algebraic parabolic loci, we also have to take into account geodesics supported on the $A(d_j)$.
For each $d_j\ (j=v+1, \dots, u)$ of torus type, we take the  vertex of $h_i(d_j)$ which is the $n_0$-th counting from the last one, and denote it by $\tau(d_j)_{-n_0}$.
Starting from the last slice of $\tau(\bar B_i)$, we advance, by elementary moves, vertices on geodesics supported on $\Sigma^j\ (j=q+1, \dots ,p)$  and $A(d_j)\ (j=v+1, \dots , u)$ for $d_j$ of torus type until it gets to $\tau(B_i^j)_{-n}$ on $\Sigma^j$ and to $\tau(d_j)_{-n_0}$ on $A(d_j)$, and obtain a sequence of markings $\{\mu_-(i,n, n_0)\}$ as before.
Note that we do not move other vertices from those lying on the last slice of $\tau(\bar B_i)$.
In particular if $d_j\ (j=v+1, \dots , u)$ is of annulus type, the vertex on $A(d_j)$ remains to be the first one.
In the same way, we define $\tau(d_j)_n$ for $d_j\ (j=1, \dots v)$ of torus type and a sequence of markings $\{\mu_+(i,n, n_0)\}$.

We also need to modify the definition of the function $m$ in \cref{m}.
For each $\mu$ in $\mu^-(i, n, n_0)$, we define $m(\mu)$ to be a point in $\mathcal T(S)$  such that a clean marking compatible with $\mu$ is a shortest marking in $(S, m(\mu))$,   the length of each curve $d$ of torus type among $d_{v+1}, \dots, d_u$  is equal to  $\delta_0/d_{A(d)}(I(h_i(d)), \mu|A(d))$, where $\mu|A(d)$ is the marking on $A(d)$ determined by $\mu$, and the lengths of curves of annulus type among $d_{v+1}, \dots , d_u$ are equal to $\length_{m_i}$ of the curves.
We define $m$ for $\mu \in \mu^+(i,n)$ similarly.
Before Claim \ref{strong limit}, we introduced the number $\mathbf n_i'$ and considered the last marking in the sequence $\mu^\pm(i,\mathbf n_i')$, which we denoted by $\mu^\pm_\infty(i, \mathbf n_i')$.
In the present case, we define $\mathbf n_i'$ in the same way as before, and fix $n_0$ independently of $i$.
We define the function $m'$ in the same way as in \cref{m'} by decreasing gradually lengths of curves corresponding to non-isolated parabolic loci.
Then the rest of the construction in the proof of Proposition \ref{deformation to strong one} works without any change, and we can define  broken quasi-geodesics $\alpha_i^-$ and $\alpha^+$.

We need to show that the algebraic limits appearing Claims \ref{strong limit} and \ref{any limit} has $d_1, \dots , d_u$ as parabolic elements whatever point you choose on $\alpha^-$ or $\alpha^+$ by our definitions of $m$  and $m'$ as above.
Suppose first that $c$ is a curve in $d_1, \dots , d_u$ of annulus type.
Let $A$ be a boundary component of $\mathbf M'$ having $c$ as a core curve.
Since $c$ is an isolated curve this means that the geometrically finite block is split along $c$, which implies that either $\length_{m_i}(c) \rightarrow 0$ or $\length_{n_i}(c) \rightarrow 0$ depending on $c$ is lower or upper.
Therefore, by our modified definition of $m$ above, we see that the length of $c$ with respect to the sequences appearing Claims \ref{strong limit} and \ref{any limit} goes to $0$, and hence $c$ is parabolic in their algebraic limits.

Suppose next that $c$ is of torus type.
We first consider the algebraic limit of $qf(m'(\mu^-_\infty(i, \mathbf n_i', n_0)), m'(\mu^+(i,\mathbf n_i', n_0))$, corresponding to Claim \ref{strong limit}.
By our definition of the function $m$ (and $m'$) above, we see that if $c$ is among $d_1, \dots , d_v$, then $\length_{m'(\mu^+(i,\mathbf n_i, n_0))}(c)$ goes to $0$ as $i \rightarrow \infty$, and if it is among $d_{v+1}, \dots , d_u$, then $\length_{m'(\mu^-(i,\mathbf n_i, n_0))}(c)$ goes to $0$.
This shows that $c$ is parabolic in the algebraic limit of this sequence.

We now turn to considering the algebraic limit of the sequence $\{\beta_i(t_i)\}$, where $\beta_i(t)=qf(\alpha_i^-(t), \alpha_i^+(t))$.
Suppose that $c$ is a curve of torus type among $d_{v+1}, \dots , d_u$, and we consider the point $\alpha_i^-(t_i)$ in the Teichm\"{u}ller space.
There are essentially three cases to consider.
In the first and the second cases, we assume that $\alpha_i^-(t_i)$ lies on a Teichm\"{u}ller geodesic with endpoints $m'(\mu_1(i)), m'(\mu_2(i))$ for some adjacent slices $\mu_1(i), \mu_2(i)$ in $\mu^-(i, \mathbf n_i, n_0)$.
In the first case, we further assume that $\mu_1(i)$ contains the $k(i)$-th vertex on $h_i(c)$, where $\length(\gamma(i))-k(i) \rightarrow \infty$.
Then the same holds for $\mu_2(i)$ since they are adjacent.
Then for the hierarchy of tight geodesics connecting  for $\alpha_i^-(t_i)$ with $\alpha_i^+(t_i)$ has a geodesic supported on $A(c)$ with length at least $\length(\gamma(i))-k(i)-1$, which goes to $\infty$.
Therefore in the geometric limit $\mathbf M_\infty(\beta)$, the curve $c$ lies on a torus boundary.
This implies that  $c$ is parabolic in the algebraic limit.
In the second case, we assume on the contrary that $\length(\gamma(i))-k(i)$ is bounded, which implies $k(i) \to \infty$.
Then by our modified definition of the function $m$ (and $m'$), we have $\length_{m'(\mu_1(i))}(c)$ and $\length_{m'(\mu_2(i))}(c)$ go to $0$, which implies that $\length_{\alpha_i^-(t_i)}(c)$ also goes to $0$.
Therefore $c$ is a parabolic curve in the algebraic limit also in this case.
The third case is when $\alpha_i^-(t_i)$ lies a quasi-geodesic connecting $m_i$ with the first marking $\mu_1^-(i,\mathbf n_i, n_0)$.
Then the situation is the same as the first case and $c$ lies on a torus boundary of the geometric limit $\mathbf M_\infty(\beta)$ since the initial and the terminal markings are the same as the case of connecting $\mu_1^-(i,\mathbf n_i, n_0)$ with $\mu_1^+(i, \mathbf n_i, n_0)$ and the hierarchy of tight geodesics connecting them has a geodesic of length $h_i(c) \to \infty$ supported on $A(c)$.
Thus we have shown that $c$ is a parabolic curve in all the cases.
The same argument works also when $c$ is of torus type among $d_1, \dots , d_v$, and we can prove them to be upper parabolic of the algebraic limit of $\{\beta_i(t_i)\}$. 
We note that  $\mathbf M_\infty(\beta)$ has the geometrically infinite ends corresponding to $E^1 , \dots , E^p$ by our definition of $\alpha_i^\pm$.
This implies that $\mathbf M_\infty(\beta)$ has no isolated algebraic parabolic curves other than $d_1, \dots , d_u$.
Therefore we also see that there are no other isolated parabolic curves other than $d_1, \dots , d_u$ in the algebraic limit of $\{\beta_i(t_i)\}$.

Thus we have shown that Claims \ref{strong limit} and \ref{any limit} holds under the assumption that $g'$ does not go around a torus boundary component.
In particular, this shows that the algebraic limit of $\beta_i(t_i)$ is a quasi-conformal deformation of $(\Gamma, \psi)$.


(2) {\bf To show that $g'$ does not go around a torus boundary component.}
It remains to show that the standard immersion $g'$ does not go around a torus boundary component in $\mathbf M'$ in our setting.
We consider first the case when $\Gamma$ is a b-group.
Let $\{(G_i, \phi_i)=qf(m_0, n_i)\}$ and $\{(\bar G_i, \bar \phi_i)=qf(m_0, \bar n_i)\}$ be sequences in the Bers slice converging to $(\Gamma, \psi)$.
Let $\mathbf M'$ and $\bar{\mathbf M}'$ be model manifolds for non-cuspidal parts of geometric limits $G_\infty$ and $G_\infty'$ of $\{G_i\}$ and $\{G_i'\}$, which are non-cuspidal parts of  geometric limits of the model manifolds $\mathbf M_i'$ and $\bar{\mathbf M}_i'$ constructed as in \cref{modified models} respectively.
Let $g': S \rightarrow \mathbf M'$ and $\bar g': S \rightarrow \bar{\mathbf M}'$ be standard algebraic immersions.
Since the lower geometrically finite ends must be lifted to the algebraic limits, neither $g'$ nor $\bar g'$ can go around torus boundary components, hence homotopic to horizontal surfaces in $\mathbf M'$ and $\bar{\mathbf M}'$ respectively.
Therefore the argument of the proof of Theorem \ref{bumping theorem} works, and we see that for any neighbourhood $U$ of $(\Gamma, \psi)$, there is an arc $\alpha_i$ connecting $(G_i,\phi_i)$ and $(\bar G_i, \bar \phi_i)$ in the Bers slice $U \cap B_{m_0}$.

Next suppose that $\Gamma$ is not a b-group.
Let $\{(G_i, \phi_i) \in QF(S)\}$ be a sequence converging to $(\Gamma, \psi)$.
Again, we let $\mathbf M'$ be a model manifold of the non-cuspidal part of the geometric limit $G_\infty$, which is the non-cuspidal part of the geometric limit of $\mathbf M_i'$.
Suppose that $g'$ goes around a torus component $T$ of $\mathbf M'$ counter-clockwise.
Let $c$ be a longitude of $T$.
Then, $c$, regarded as lying on $S$, is an upper parabolic curve of $M_\Gamma=\hyperbolic^3/\Gamma$.

Since we assumed that $\Omega_\Gamma/\Gamma$ consists of thrice-punctured spheres, there is either a lower parabolic curve $d$  or an ending lamination $\lambda$ of a lower end, intersecting $c$ essentially.
If there is a lower parabolic curve $d$, then there is a boundary component $T'$ of $\mathbf M'$ whose core curve or a longitude is homotopic to $g'(d)$ and which is situated below $g'(S)$.
This is impossible since $g'$ goes around $T$ whose longitude intersects $d$ essentially on $S$.
If there is an ending lamination $\lambda$, then there is a algebraic simply degenerate end contained in $\Sigma \times (s,t]$ in $\mathbf M$ having $\lambda$ as the ending lamination.
Then $\Sigma \times \{s+\epsilon\}$ must be homotopic to $g'(\Sigma)$.
Again this is impossible since $g'$ goes around $T$ whose longitude intersects $\lambda$ essentially on $S$.
Thus, we are lead to a contradiction in both cases, and see that $g'$ cannot go around a torus boundary component counter-clockwise.
The case when $g'$ goes around a torus boundary component clockwise can be deal with in the same way just by turning everything upside down.
\end{proof}

\section{Proof of Theorems \ref{generalised limit laminations} and \ref{generalised main}}
\subsection{Proof of Theorem \ref{generalised limit laminations}}
After passing to a subsequence, $\{G_i\}$ can be assumed to converge geometrically to a Kleinian group $G_\infty$ containing $\Gamma$ as before.
This induces a pointed Gromov convergence of $(M_i, y_i)$ to $(M_\infty, y_\infty)$ with $M_\infty=\hyperbolic^3/G_\infty$.
Let $\mathbf M$ be a model manifold of $(\hyperbolic^3/G_\infty)_0$, and $\mathbf M_i$ that of $(M_i)_0=(\hyperbolic^3/G_i)_0$ as before.
Then $\mathbf M_i[0]$ converges geometrically to $\mathbf M[0]$ by taking a basepoint $x_i$ in $\mathbf M_i$ which is mapped by the model map to a point within uniformly bounded distance from the basepoint $y_i$ of $M_i$.
As in Lemma \ref{position of S}, we take a standard algebraic immersion $g' : S \rightarrow \mathbf M$.

Let $E$ be an algebraic simply degenerate end of $\mathbf M$ with ending lamination $\lambda_E$, and $B=\Sigma \times J$ a brick of $\mathbf M $ containing $E$.
We assume that $E$ is an upper end.
The case when $E$ is a lower end can be argued in the same way by just turning everything upside down as usual.
Then by Proposition \ref{simply degenerate brick}, there is a tight geodesic $\gamma_i$ in $\CC(\Sigma)$ one of whose endpoints (possibly at infinity) converges to $\lambda_E$ as $i \rightarrow \infty$.
In the case when $\gamma_i$ is a geodesic ray, its endpoint at infinity is an ending lamination of an upper end of $\mathbf M_i$, which is contained in $e_+(i)$.
Therefore, we see that $\lambda_E$ is contained in the Hausdorff limit of $e_+(i)$.

Next suppose that $\gamma_i$ is a finite geodesic.
Then, by Theorem 3.1  and \S 6 of Masur-Minsky \cite{MaMi}, the last vertex of $\gamma_i$ and  $\pi_\Sigma(e_+(i))$ are within uniformly bounded distance, and hence the endpoint of $\gamma_i$ and $\pi_\Sigma(e_+(i))$ converge to the same lamination $\lambda_E$ with respect to the topology of $\mathcal{UML}(\Sigma)$ as $i \rightarrow \infty$.
Since $\lambda_E$ is arational, we see  that the Hausdorff limit of $e_+(i)|\Sigma$ contains $\lambda_E$.

The converse can be shown by the same argument as the proof of Theorem \ref{limit laminations}.

\subsection{Proof of Theorem \ref{generalised main}}
Suppose, seeking a contradiction, that a sequence $\{(G_i, \phi_i)\}$ as in the statement converges.
Then by Theorem \ref{generalised limit laminations}, $\lambda$ is an ending lamination of an upper end and $\mu$ is an ending lamination of a lower end of the algebraic limit.
This implies the shared boundary component of the minimal supporting surface of $\mu$ is a parabolic locus which is both upper and lower.
This is impossible, and we have completed the proof of Theorem \ref{generalised main}.

\end{document}